\documentclass[11pt]{amsart}

\usepackage[pdfauthor={Ziyang GAO}, 
        pdftitle={Ax-Lindemann}%
      dvips,colorlinks=true]{hyperref}
\hypersetup{
  bookmarksnumbered=true,
  linkcolor=black,
  citecolor=black,
  pagecolor=black, 
  urlcolor=black,  
}

\usepackage[utf8]{inputenc}

\usepackage{xtab}
\usepackage{amsmath, amssymb, cancel,verbatim}
\usepackage[all]{xy}

\usepackage{mathtools,booktabs}
\usepackage{amscd}
\usepackage{bbm, stmaryrd}
\usepackage{yfonts}
\usepackage{color}

\usepackage{epstopdf} 
\usepackage{booktabs}

\newtheorem{teo}{Theorem}[section]
\newtheorem{thm}[teo]{Theorem}
\newtheorem{prop}[teo]{Proposition}
\newtheorem{lemma}[teo]{Lemma}
\newtheorem{cor}[teo]{Corollary}
\newtheorem{conj}[teo]{Conjecture}

\newtheorem{defn}[teo]{Definition}

\newtheorem{rmk}[teo]{Remark}
\newtheorem{fact}{Fact}
\newtheorem*{claim}{Claim}
\newtheorem{def-prop}[teo]{Definition-Proposition}
\newtheorem{notation}[teo]{Notation}

\newcommand\numberthis{\addtocounter{equation}{1}\tag{\theequation}}
\numberwithin{equation}{section}

\usepackage[small,nohug,heads=LaTeX]{diagrams}
\diagramstyle[labelstyle=\scriptstyle]

  \newcommand{\A}{\mathbb{A}}
  \newcommand{\C}{\mathbb{C}}

  \newcommand{\G}{\mathbb{G}}
  \renewcommand{\H}{\mathbb{H}}  
  
  \newcommand{\N}{\mathbb{N}}
  
  \newcommand{\Q}{\mathbb{Q}}
  \newcommand{\R}{\mathbb{R}}  
  \renewcommand{\S}{\mathbb{S}}
  \newcommand{\V}{\mathbb{V}}
  \newcommand{\Z}{\mathbb{Z}}

  \renewcommand{\epsilon}{\varepsilon}
  
  \newcommand{\im}{\text{Im}}

  \renewcommand{\cong}{\simeq}
  \renewcommand{\bar}{\overline}
  \renewcommand{\tilde}{\widetilde}
  \renewcommand{\hat}{\widehat}

  \providecommand{\frac}[1]{\operatorname{Frac}(#1)}

  \renewcommand{\hom}{\operatorname{Hom}}

  \newcommand{\rank}{\operatorname{rank}}
  
  \newcommand{\gal}{\operatorname{Gal}}

  \newcommand{\stab}{\operatorname{Stab}}
  \newcommand{\MT}{\operatorname{MT}}
  \newcommand{\GL}{\operatorname{GL}}

  \renewcommand{\ker}{\operatorname{Ker}}
  
  \newcommand{\Max}{\operatorname{Max}}
  \newcommand{\codim}{\operatorname{codim}}

  \renewcommand{\lim}{\operatorname{lim}}
  \renewcommand{\deg}{\operatorname{deg}}
  
  \newcommand{\ord}{\operatorname{ord}}

  \newcommand{\End}{\operatorname{End}}
  
  \newcommand{\disc}{\operatorname{disc}}
  \newcommand{\lie}{\operatorname{Lie}}
  \newcommand{\Gr}{\operatorname{Gr}}
  \newcommand{\Alb}{\operatorname{Alb}}
  \newcommand{\alb}{\operatorname{alb}}
  \newcommand{\der}{\mathrm{der}}
  \newcommand{\ad}{\mathrm{ad}}
  \newcommand{\sm}{\mathrm{sm}}
  \newcommand{\de}{\mathrm{de}}
  
  \newcommand{\cp}{\mathrm{cp}}
  \newcommand{\mon}{\mathrm{mon}}
  \newcommand{\Zar}{\mathrm{Zar}}
  \newcommand{\GSp}{\mathrm{GSp}}
  \newcommand{\Sp}{\mathrm{Sp}}
  
  \renewcommand{\Re}{\mathfrak{R}\mathfrak{F}}
  
  \newcommand{\res}{\mathrm{Res}}

  \newcommand{\unif}{\mathrm{unif}}

\newcommand{\cA}{\mathcal{A}}

\newcommand{\cD}{\mathcal{D}}

\newcommand{\cF}{\mathcal{F}}
\newcommand{\cG}{\mathcal{G}}

\newcommand{\cO}{\mathcal{O}}
\newcommand{\cR}{\mathcal{R}}

\newcommand{\cT}{\mathcal{T}}
\newcommand{\cV}{\mathcal{V}}

\newcommand{\cX}{\mathcal{X}}
\newcommand{\cY}{\mathcal{Y}}
\newcommand{\cZ}{\mathcal{Z}}

\newcommand{\upperarrow}[1]{
\setlength{\unitlength}{0.03\DiagramCellWidth}
\begin{picture}(0,0)(0,0)
\linethickness{1pt} 
\bezier{15}(-46,9)(-24,23)(-2,9)
\put(-24,19){\makebox(0,0)[b]{$\scriptstyle {#1}$}}
\put(-46.6,8.7){\vector(-2,-1){0}}
\end{picture}
}
\newcommand{\upperarrowtwo}[1]{
\setlength{\unitlength}{0.03\DiagramCellWidth}
\begin{picture}(0,0)(0,0)
\linethickness{1pt} 
\bezier{15}(-40,9)(-21,23)(-2,9)
\put(-21,19){\makebox(0,0)[b]{$\scriptstyle {#1}$}}
\put(-40.6,8.7){\vector(-2,-1){0}}
\end{picture}
}

\newcommand\supervisor[1]{\def\@supervisor{#1}}

\newcounter{elno}

\renewcommand{\cong}{\simeq}

\begin{document}
\title[Andr\'{e}-Oort: Ax-Lindemann and lower bound for Galois orbits of CM points]{Towards the Andre-Oort conjecture for mixed Shimura varieties: the Ax-Lindemann theorem and lower bounds for Galois orbits of special points}
\author{Ziyang Gao}
\address{Institut des Hautes \'{E}tudes Scientifiques, 
Le Bois-Marie 35, route de Chartres, 
91440 Bures-sur-Yvette,
France}
\email{ziyang.gao@math.u-psud.fr}
\maketitle

\begin{abstract}
We prove in this paper the Ax-Lindemann-Weierstra{\ss} theorem for all mixed Shimura varieties and discuss the lower bounds for Galois orbits of special points of mixed Shimura varieties. In particular we reprove a result of Silverberg \cite{SilverbergTorsion-points-} in a different approach. Then combining these results we prove the Andr\'{e}-Oort conjecture unconditionally for any mixed Shimura variety whose pure part is a subvariety of $\cA_6^n$ and under GRH for all mixed Shimura varieties of abelian type.
\end{abstract}

\tableofcontents

\section{Introduction}
\subsection{Background}
Every connected Shimura variety, being the quotient of a Hermitian symmetric domain by an arithmetic group, can be realized as a moduli space for pure Hodge structures plus tensors. Unlike the Hermitian symmetric domains themselves, connected Shimura varieties are algebraic varieties. This was proved by Baily-Borel \cite{BailyCompactificatio}. The prototype for all Shimura varieties is the Siegel moduli space of principally polarized abelian varieties of dimension $g$ with a level structure. The points in this moduli space corresponding to CM abelian varieties, which are called \textit{special points}, play a particularly important role in the theory of Shimura varieties. A major reason is that the Galois actions on special points are fairly completely determined by the Shimura-Taniyama theorem \cite[Theorem 4.19]{DeligneTravaux-de-Shim} and its generalization by Milne-Shih \cite{DeligneHodge-cycles-mo} to Galois conjugates of CM abelian varieties. The concept of special points and the results concerning the Galois action on them have been generalized to arbitrary Shimura varieties. Every Shimura variety has a Zariski dense subset of special points (\cite[Proposition 5.2]{DeligneTravaux-de-Shim}), and hence the results above have led to the concept of the \textit{canonical model} of a Shimura variety over a number field: see Deligne \cite{DeligneVariete-de-Shim} and Milne \cite{MilneCanonical-model}.

It is of course a natural problem to look for suitable compactifications of a given Shimura variety. The first compactification is the Baily-Borel (or minimal) compactification \cite{BailyCompactificatio}, which is canonical. However this compactification has bad singularities along the boundary. Next we have the toroidal compactifications \cite{AshSmooth-compacti}, which are no longer canonical but provide smooth compactifications of Shimura varieties. To construct these compactifications one needs to study the boundary of a Shimura variety. As one approaches the boundary of a Hermitian symmetric domain, pure Hodge structures degenerate into mixed Hodge structures, and as one approaches the boundary of a Shimura variety, abelian varieties degenerate into 1-motives. This will lead to a new object, generalizing the notion of Shimura varieties and parametrizing variations of mixed Hodge structures (all of whose pure constituents are polarizable), which we shall call a mixed Shimura variety. In order to distinguish, we will use the term ``pure Shimura variety" to denote the Shimura varieties in the first paragraph. Here we copy a list of some objects attached to a Shimura variety and the corresponding object attached to a mixed Shimura variety from Milne \cite[Introduction]{MilneCanonical-model},

\begin{tabular}{c|c}
\textbf{(pure) Shimura variety} & \textbf{mixed Shimura variety}\\ \hline
Hermitian (or bounded) symmetric domain & Siegel domain (of the third kind) \\
pure Hodge structure & mixed Hoge structure \\
reductive group over $\Q$ & algebraic group over $\Q$ with 3-step filtration \\
abelian variety & 1-motive \\
motive & mixed motive
\end{tabular}

Here are several important examples of mixed Shimura varieties: 
\begin{itemize}
\item the universal family of abelian varieties of dimension $g$ with a level structure;
\item the $\G_m$-torsor over such a universal family which corresponds to an ample line bundle over this family;
\item Poincar\'{e} bi-extension;
\item products of the above.
\end{itemize}
There is also the concept of special points for mixed Shimura varieties, e.g. special points of the universal family of abelian varieties are those which correspond to torsion points on CM abelian varieties. Similar results about the Galois action on special points and the canonical models of pure Shimura varieties hold for mixed Shimura varieties: see Pink \cite{PinkThesis}.

An irreducible component of a mixed Shimura subvariety of a mixed Shimura variety, or its image under a Hecke operator, is called a special subvariety. In particular, special points are precisely special subvarieties of dimension 0. As for pure Shimura varieties, every special subvariety contains a Zariski dense subset of special points (\cite[11.7]{PinkThesis}).
The aim of this article is to study the following conjecture, which is the converse of this fact.
\begin{conj}[Generalized Andr\'{e}-Oort]\label{Andre-Oort}
Let $Y$ be a closed irreducible subvariety of a mixed Shimura variety. If $Y$ contains a Zariski-dense set of special points, then it is special.
\end{conj}

The original Andr\'{e}-Oort conjecture, in which one replaces ``mixed Shimura variety'' by ``pure Shimura variety'', has been established in many cases (unconditionally or under GRH without using o-minimality) by Moonen \cite{MoonenLinearity-prope}, Andr\'{e} \cite{AndreFinitude-des-co}, Edixhoven \cite{EdixhovenSpecial-points-, EdixhovenOn-the-Andre-Oo}, Edixhoven-Yafaev \cite{EdixhovenSubvarieties-of} and Yafaev \cite{YafaevOn-a-result-of-, YafaevA-conjecture-of}. It was proved under GRH by Klingler, Ullmo and Yafaev \cite{UllmoGalois-orbits-a, KlinglerThe-Andre-Oort-}, where equidistribution results of Clozel-Ullmo \cite{ClozelEquidistributio1} were used. Later Daw \cite{DawDegrees-of-stro} removed the ergodic theory from Klingler-Yafaev's result. Our generalized version of the Andr\'{e}-Oort conjecture was suggested by Y.Andr\'{e} in \cite[Lecture 3]{AndreShimura-varieti}, where he also gave a proof of the case of the universal family of elliptic curves. Results for mixed Shimura varieties have been obtained by Habegger \cite{HabeggerSpecial-points-} for fibered powers of elliptic surfaces. Remark that Conjecture \ref{Andre-Oort} is not only a direct generalization of the original the Andr\'{e}-Oort conjecture, but also contains the Manin-Mumford conjecture for any complex semi-abelian variety whose abelian variety part is CM.

\subsection{Ax-Lindemann-Weierstra{\ss}}
A recent development of this conjecture was made by Pila-Zannier. Its origin was the proof of the Manin-Mumford conjecture \cite{PilaRational-points}. Afterwards using tools of o-minimality and Pila-Wilkie's counting theorem \cite{PilaThe-rational-po}, Pila proved the Andr\'{e}-Oort conjecture for $\cA_1^n$ (product of modular curves) unconditionally \cite{PilaO-minimality-an}. Daw-Yafaev later proved the Andr\'{e}-Oort conjecture unconditionally for Hilbert modular sufaces \cite{DawAn-unconditiona}. In this strategy of proving the Andr\'{e}-Oort conjecture, a key step is to establish the following generalization of the functional Lindemann-Weierstrass theorem \cite{AxOn-Schanuels-co}, which is the main result of this paper:

\begin{thm}[Ax-Lindemann-Weierstra{\ss} Theorem]\label{Ax-Lindemann for type star}
 Let $S$ be a connected mixed Shimura variety, let $\unif\colon \cX^+\rightarrow S$
 be its uniformization and let $Y$ be a closed irreducible subvariety of $S$. Let $\tilde{Z}$ be an irreducible algebraic subset of $\cX^+$ contained in $\unif^{-1}(Y)$, maximal for these properties. Then $\unif(\tilde{Z})$ is weakly special.
\end{thm}

We explain briefly the notions in this theorem. A connected mixed Shimura variety is defined to be a connected component of a mixed Shimura variety. As for the pure case, the uniformizing space $\cX^+$ can be realized as an open subset (w.r.t. the archimedean topology) of a complex algebraic variety $\cX^\vee$ ($\mathsection$\ref{Realization of X}), and an \textit{irreducible algebraic subset} of $\cX^+$ is defined to be an a complex analytic irreducible component of the intersection of a closed algebraic subvariety of $\cX^\vee$ and $\cX^+$ (Definition \ref{algebraicity on the top}). Consider Shimura morphisms of connected mixed Shimura varieties $T^\prime\xleftarrow{[\varphi]}T\xrightarrow{[i]}S$ and a point $t^\prime\in T^\prime$. Following Pink \cite{PinkA-Combination-o}, an irreducible component of $[i]([\varphi]^{-1}(t^\prime))$ is called a \textit{weakly special subvariety} of $S$ (Definition \ref{definition of weakly special subvariety}). In the case of pure Shimura varieties, Moonen \cite{MoonenLinearity-prope} proved that the weakly special subvarieties are precisely the totally geodesic subvarieties.

It is worth remarking that if we prove Conjecture \ref{Andre-Oort} via Theorem \ref{Ax-Lindemann for type star}, then we also prove the Manin-Mumford conjecture for all semi-abelian varieties (not only those whose abelian variety parts are CM). Theorem \ref{Ax-Lindemann for type star} was proved for (semi-)abelian varieties by Ax \cite{AxSome-topics-in-}, and then refound and reproved by Pila-Zannier \cite{PilaRational-points} and Peterzil-Starchenko \cite{PeterzilAround-Pila-Zan}, where proofs of Manin-Mumford via Ax-Lindemann-Weierstra{\ss} can be found. For the hyperbolic case (pure Shimura varieties), this theorem has also been established in several different cases by Pila (for $\cA_1^n$ \cite{PilaO-minimality-an} and later a product of universal families of elliptic curves \cite{PilaSpecial-point-p}\footnote{The definition of weakly special subvarieties in this paper looks quite different and a lot more complicated than the one we use here. They probably coincide but I did not check it.}), Ullmo-Yafaev (for projective pure Shimura varieties \cite{UllmoThe-Hyperbolic-}) and Pila-Tsimerman (for $\cA_2$ \cite{PilaAbelianSurfaces} and then $\cA_g$ \cite{PilaAxLindemannAg}). Klingler-Ullmo-Yafaev have recently proved this theorem for all pure Shimura varieties \cite{KlinglerThe-Hyperbolic-} using the idea of calculating the volumes of algebraic curves near the boundary, which was firstly executed in \cite{UllmoThe-Hyperbolic-} and then in \cite{PilaAxLindemannAg}. Our proof is based on the result of \cite{KlinglerThe-Hyperbolic-}. A main ingredient to prove all the  results above (including the whole Theorem \ref{Ax-Lindemann for type star}) is Pila-Wilkie's counting theorem; however unlike the pure case, the ``family version" of this counting theorem is crucially used in this paper. Some other difficulties to prove Theorem \ref{Ax-Lindemann for type star} for mixed Shimura varieties which we do not encounter in the pure case are listed in $\mathsection$\ref{Ax-Lindemann Part 1: Outline of the proof} before Lemma \ref{projection of maximal still maximal}. We hope that this may make the strategy of our proof more clear.

We close this subsection with the following comment about weakly special subvarieties. To study them, it is useful to describe the smallest weakly special subvariety containing a given subvariety $Y$ of a connected mixed Shimura variety $S$. We shall prove in Theorem \ref{smallest weakly special containing a subvariety} (sometimes called \textbf{\textit{Ax's theorem of log type}}\footnote{This is pointed to me by Daniel Bertrand.}), generalizing Moonen's result \cite[3.6, 3.7]{MoonenLinearity-prope} for pure Shimura varieties, that this smallest weakly special subvariety is precisely the one defined by the connected algebraic monodromy group of $Y^{\sm}$. The proof uses Andr\'{e}'s \cite{AndreMumford-Tate-gr} and Wildeshaus'\cite{WildeshausThe-canonical-c} earlier work about variations of mixed Hodge structure (over mixed Shimura varieties). As a consequence of this description, we shall prove a characterization of weakly special subvarieties in terms of ``bi-algebraicity" (Corollary \ref{weakly special iff algebraic on the top and on the bottom}), which is a direct generalization of the main result of Ullmo-Yafaev \cite{UllmoA-characterisat}.

\subsection{From Ax-Lindemann-Weierstra{\ss} to Andr\'{e}-Oort}
Ullmo and Pila-Tsimerman explained separately in \cite{UllmoQuelques-applic} \cite{PilaAxLindemannAg} how to deduce the Andr\'{e}-Oort conjecture from the Ax-Lindemann-Weierstra{\ss} theorem for pure Shimura varieties. The proof of Ullmo is generalized to mixed Shimura varieties in this paper ($\mathsection$\ref{Consequence of Ax-Lindemann}). They showed that in order to prove the Andr\'{e}-Oort conjecture for pure Shimura varieties of abelian type, the only ingredient (and obstacle) left is a suitable lower bound for the Galois orbit of a special point of a pure Shimura variety conjectured by Edixhoven \cite{EdixhovenOpen-problems-i}. We prove that what we need to prove Conjecture \ref{Andre-Oort} (for any mixed Shimura variety whose pure part is of abelian type) is the same lower bound. More explicitly, we prove that the naturally expected good lower bound for the Galois orbit of a special point, i.e. the product of the lower bounds of the base and the fiber, is fulfilled (Proposition \ref{galois orbit for the fiber}). As a special case, this provides a new proof for the result of Silverberg \cite{SilverbergTorsion-points-} (Corollary \ref{result of Silverberg})

\begin{thm}
Let $A$ be a complex abelian variety of CM type of dimension $g$. Its field of definition $k$ is then a number field by CM theory. Let $t$ be a torsion point of $A$ of order $N(t)$. If we denote by $k(t)$ the field of definition of $t$ over $k$, then $\forall\epsilon\in(0,1)$,
\[
[k(t):k]\gg_{g,\epsilon}N(t)^{1-\epsilon}.
\]
\end{thm}

In Silverberg's work, the constant on the right hand side also depends on the field $k$.
The lower bound for pure Shimura varieties is known under GRH (\cite{TsimermanBrauer-Siegel-f}, \cite{UllmoNombre-de-class}). The best unconditional result is given by Tsimerman \cite{TsimermanBrauer-Siegel-f}. He established the lower bound unconditionally for $g\leqslant6$ (for $g\leqslant3$ this was also proved by Ullmo-Yafaev by a similar method \cite{UllmoNombre-de-class}). Therefore as a consequence we prove (Theorem \ref{Andre-Oort for L6})
\begin{thm}\label{Andre-Oort A6}
Under GRH, the generalized Andr\'{e}-Oort conjecture (Conjecture \eqref{Andre-Oort}) holds for any mixed Shimura variety whose pure part is a closed subvariety of $\cA_g^n$. This result is unconditional (i.e. we do not need GRH) if $g\leqslant 6$.
\end{thm}

\subsection{Zilber-Pink}
Finally it is worth remarking that Conjecture \ref{Andre-Oort} is part of the more general Zilber-Pink Conjecture \cite{PinkA-Combination-o, ZilberExponential-sum, RemondAutour-de-la-co}. Some unlikely intersections results of type Andr\'{e}-Pink \cite[Conjecture 1.6]{PinkA-Combination-o} about pure Shimura varieties beyond the Andr\'{e}-Oort conjecture have been obtained by Pink \cite[Theorem 7.6]{PinkA-Combination-o} (Galois generic points in $\cA_g$), Habegger-Pila \cite{HabeggerSome-unlikely-i} (curves in $\cA_1^n$) and Orr \cite{OrrFamilies-of-abe} (curves in $\cA_g$). I shall not talk about the case of algebraic groups (see \cite{Chambert-LoirRelations-de-de} for a summary). As for mixed Shimura varieties, Bertrand, Bertrand-Edixhoven, Bertrand-Pillay and  Bertrand-Masser-Pillay-Zannier have recently been working on Poincar\'{e} biextensions \cite{BertrandSpecial-points-, BertrandA-Lindemann-Wei, BertrandUnlikely-inters, BertrandPinks-conjectur, BertrandRelative-Manin-}. They have got several interesting results, some of which provide good examples for this paper.

\subsection*{Structure of this paper} In $\mathsection$\ref{Connected mixed Shimura varieties} we recall some basic facts about mixed Shimura varieties following Pink \cite{PinkThesis}. $\mathsection$\ref{Variations of mixed (Z-)Hodge structure} is a summary of variations of mixed Hodge structure. In $\mathsection$\ref{Realization of X} we discuss the realization of the uniformizing space of any given mixed Shimura variety. In particular we give a realization of it which is at the same time semi-algebraic and complex analytic (Proposition \ref{realization of the uniformizing space}). $\mathsection$\ref{(Weakly) special subvarieties} is exploited to study (weakly) special subvarieties following Pink \cite{PinkA-Combination-o}. In $\mathsection$\ref{Algebraicity in the uniformizing space} we define algebraic subsets of the uniformizing space and prove the functoriality of the algebraicity. In $\mathsection$\ref{Results for the unipotent part} we list and prove some results for the unipotent part, with the statement of the Ax-Lindemann-Weierstra{\ss} Theorem for the unipotent part which we will eventually prove in $\mathsection$\ref{Ax-Lindemann Part 3: The unipotent part}. In $\mathsection$\ref{The smallest weakly special subvariety containing a given subvariety} we will have our first important results, i.e. the description of the smallest weakly special subvariety containing a given subvariety $Y$ of a connected mixed Shimura variety $S$ (Theorem \ref{smallest weakly special containing a subvariety}) and a criterion of weakly special subvarieties in terms of ``bi-algebraicity" (Corollary \ref{weakly special iff algebraic on the top and on the bottom}). The core of this paper is the proof of Theorem \ref{Ax-Lindemann for type star}, and it is executed in $\mathsection$\ref{Ax-Lindemann Part 1: Outline of the proof}, $\mathsection$\ref{Ax-Lindemann Part 2: Estimate} and $\mathsection$\ref{Ax-Lindemann Part 3: The unipotent part}. The proof is quite technical, and for readers' convenience we organize it as follows: the outline of the proof is presented in $\mathsection$\ref{Ax-Lindemann Part 1: Outline of the proof}, a key proposition leading to the theorem is proved in $\mathsection$\ref{Ax-Lindemann Part 2: Estimate} by using Pila-Wilkie's counting theorem and we shall prove Ax-Lindemann-Weierstra{\ss} for the unipotent part (the fiber) separately in $\mathsection$\ref{Ax-Lindemann Part 3: The unipotent part}. In $\mathsection$\ref{Consequence of Ax-Lindemann} we derive a corollary from Theorem \ref{Ax-Lindemann for type star}, which will be used to prove Theorem \ref{Andre-Oort A6} in $\mathsection$\ref{From Ax-Lindemann to Andre-Oort}.2 together with a suitable lower bound discussed in $\mathsection$\ref{From Ax-Lindemann to Andre-Oort}.1. In the Appendix we reprove the Ax-Lindemann-Weierstra{\ss} theorem for algebraic tori over $\C$ and complex abelian varieties by this method of calculating volumes and counting points.

\subsection*{Acknowledgement}
This topic was introduced to me by Emmanuel Ullmo. This first part of this paper ($\mathsection1\sim\mathsection8$) was done in Leiden University, while the two main theorems were proved when I was in Universit\'{e} Paris-Sud. I would like to express my gratitude to my supervisors Emmanuel Ullmo and Bas Edixhoven for weekly discussions and their valuable suggestions for the writing. I would like to thank Martin Orr for having pointed out a serious gap in $\mathsection$\ref{Ax-Lindemann Part 1: Outline of the proof} in a previous version as well as his several valuable remarks, especially for the last part of $\mathsection$\ref{Ax-Lindemann Part 2: Estimate}. Ya'acov Peterzil pointed out to me that the proof of the definability in $\mathsection$\ref{Ax-Lindemann Part 2: Estimate}.1 in a previous version was wrong. I have benefited a lot from the discussion with him and Sergei Starchenko for this definability problem. I also had some interesting discussion with Daniel Bertrand and Chao Zhang. Yves Andr\'{e}, Daniel Bertrand, Bruno Klingler and Martin Orr have read a previous version of the manuscript and gave me some suggestions to improve the writing of both math and language. I would also like to thank them here. Finally I thank the referee for his/her careful reading and helpful suggestions thanks to which this article has been improved.

\subsection*{Conventions}
For $x=(x_1,...,x_n)\in\Q^n$, we define the height of $x$ as $H(x)=\Max(H(x_1),...,H(x_n))$
where for $a,b\in\Z\backslash\{0\}$ coprime $H(a/b)=\Max(|a|,|b|)$, and $H(0)=0$.

For an algebraic group $P$ over a field $k$, when we talk about a subgroup of $P$, we always mean a $k$-subgroup unless it is claimed not to be.

For the theory of o-minimality, we refer to \cite[$\mathsection$3]{UllmoThe-Hyperbolic-} (for a concise version) and \cite[$\mathsection$2,3]{PilaO-minimality-an} (for a more detailed version). In this paper, ``semi-algebraic'' will always mean $\R$-semi-algebraic. The o-minimal structure we consider is $\R_{an,\exp}$, i.e. by saying a set $A$ is definable we mean that $A$ is definable in $\R_{an,\exp}$.

\section{Connected mixed Shimura varieties}\label{Connected mixed Shimura varieties}

\subsection{Definition and basic properties}(cf. \cite[Chapter 1,2,3]{PinkThesis}) Let $\S:=\res_{\C/\R}(\G_m)$ be the Deligne-torus.
\begin{defn}\label{connected mixed Shimura datum}
A \textbf{mixed Shimura datum} $(P,\cX,h)$ is a triple where
\begin{itemize}
\item $P$ is a connected linear algebraic group over $\Q$ with unipotent radical $W$ and with another algebraic subgroup $U\subset W$ which is normal in $P$ and uniquely determined by $\cX$ using condition \eqref{definition of U} below;
\item $\cX$ is a left homogeneous space under the subgroup $P(\R)U(\C)\subset P(\C)$, and $\cX\xrightarrow{h}\hom(\S_\C,P_\C)$ is a $P(\R)U(\C)$-equivariant map s.t. every fibre of $h$ consists of at most finitely many points,
\end{itemize}
such that for some (equivalently for all) $x\in\cX$,
\begin{enumerate}
\item the composite homomorphism $\S_\C\xrightarrow{h_x}P_\C\rightarrow(P/U)_\C$ is defined over $\R$,
\item the adjoint representation induces on $\lie P$ a rational mixed Hodge structure of type
\[
\{(-1,1),(0,0),(1,-1)\}\cup\{(-1,0),(0,-1)\}\cup\{(-1,-1)\},
\]
\item\label{definition of U}
the weight filtration on $\lie P$ is given by
\[
W_n(\lie P)=\begin{cases}
0 & \text{if }n<-2 \\
\lie U &\text{if }n=-2 \\
\lie W &\text{if }n=-1 \\
\lie P &\text{if }n\geqslant 0
\end{cases},
\]
\item the conjugation by $h_x(\sqrt{-1})$ induces a Cartan involution on $G_\R^\ad$ where $G:=P/W$, and $G^\ad$ possesses no $\Q$-factor $H$ s.t. $H(\R)$ is compact,
\item $P/P^\der=Z(G)$ is an almost direct product of a $\Q$-split torus with a torus of compact type defined over $\Q$.
\end{enumerate}
In practice, we often omit the map ``$h$'' and write a mixed Shimura datum as a pair $(P,\cX)$. If in addition $P$ is reductive, then $(P,\cX)$ is called a \textbf{pure Shimura datum}.
\end{defn}

\begin{rmk}\label{sufficiently small congruence subgroup cotained in derivative}
\begin{enumerate}
\item Let $\omega:\G_{m,\R}\hookrightarrow\S$ be $t\in\R^*\mapsto t\in\C^*$. Conditions (2) and (3) together imply that the composite homomorphism $\G_{m,\C}\xrightarrow{\omega}\S_\C\xrightarrow{h_x}P_\C\rightarrow(P/U)_\C$ is a co-character of the center of $P/W$ defined over $\R$. This map is called the weight. Furthermore, condition (5) implies that the weight is defined over $\Q$.
\item Condition (5) also implies that every sufficiently small congruence subgroup $\Gamma$ of $P(\Q)$ is contained in $P^{\der}(\Q)$ (cf \cite[the proof of 3.3(a)]{PinkThesis}). Fix a Levi decomposition $P=W\rtimes G$ (\cite[Theorem 2.3]{AlgebraicGroupBible}), then $P^\der=W\rtimes G^\der$, and hence for any congruence subgroup $\Gamma<P^\der(\Q)$, $\Gamma$ is Zariski dense in $P^\der$ by condition (4) (\cite[Theorem 4.10]{AlgebraicGroupBible}).
\item Condition (5) in the definition does not make the situation less general because we are only interested in a connected component of $\cX$ (\cite[1.29]{PinkThesis}).
\end{enumerate}
\end{rmk}

\begin{defn} Let $(P,\cX)$ be a mixed Shimura datum and let $K$ be an open compact subgroup of $P(\A_f)$ where $\A_f$ is the ring of finite ad\`{e}le of $\Q$. The corresponding \textbf{mixed Shimura variety} is defined as
\[
M_K(P,\cX):=P(\Q)\backslash\cX\times P(\A_f)/K,
\]
where $P(\Q)$ acts diagonally on both factors on the left and $K$ acts on $P(\A_f)$ on the right.
\end{defn}

In this article, we only consider connected mixed Shimura data and connected mixed Shimura varieties except in $\mathsection$\ref{From Ax-Lindemann to Andre-Oort}.
\begin{defn}
\begin{enumerate}
\item A \textbf{connected mixed Shimura datum} is a pair $(P,\cX^+)$ satisfying the conditions of Definition \ref{connected mixed Shimura datum}, where $\begin{diagram}
\cX^+ &\rInto^h &\hom(\S_\C,P_\C)
\end{diagram}$
is an orbit under the subgroup $P(\R)^+U(\C)\subset P(\C)$.
\item A \textbf{connected mixed Shimura variety} $S$ associated with $(P,\cX^+)$ is of the form $\Gamma\backslash\cX^+$ for some congruence subgroup $\Gamma\subset P(\Q)\cap P(\R)_+$, where $P(\R)_+$ is the stabilizer of $\cX^+\subset\hom_\C(\S_\C,P_\C)$.
\end{enumerate}
\end{defn}

Every connected mixed Shimura variety is a connected component of a mixed Shimura variety, and vice versa (\cite[3.2]{PinkThesis}). A connected mixed Shimura variety is a complex analytic space with at most finite quotient singularities, and if $\Gamma$ is sufficiently small (for example if $\Gamma$ is neat), then $\Gamma\backslash\cX^+$ is smooth. For details we refer to \cite[Fact 2.3]{PinkA-Combination-o} or \cite[1.18, 3.3, 9.24]{PinkThesis}.

Recall the following definition, which Pink calls ``irreducible" in \cite[2.13]{PinkThesis}.
\begin{defn}\label{irreducible mixed Shimura datum}
A connected mixed Shimura datum $(P,\cX^+)$ is said to \textbf{have generic Mumford-Tate group} if $P$ possesses no proper normal subgroup $P^\prime$ defined over $\Q$ s.t. for one (equivalently all) $x\in\cX^+$, $h_x$ factors through $P^\prime_\C\subset P_\C$. We shall denote this case by $P=\MT(\cX^+)$. (This terminology will be explained in Remark \ref{stabilizer of Hodge generic point}). 
\end{defn}

\begin{prop}\label{properties of irreducible mixed Shimura datum}
Let $(P,\cX^+)$ be a connected mixed Shimura datum, then
\begin{enumerate}
\item there exists a connected mixed Shimura datum $(P^\prime,\cX^{\prime+})\hookrightarrow(P,\cX^+)$ s.t. $P^\prime=\MT(\cX^{\prime+})$ and $\cX^{\prime+}\cong\cX^+$ under this embedding;
\item if $(P,\cX^+)$ has generic Mumford-Tate group, then $P$ acts on $U$ via a scalar. In particular, any subgroup of $U$ is normal in $P$.
\end{enumerate}
\begin{proof}
\cite[2.13, 2.14]{PinkThesis}.
\end{proof}
\end{prop}

\begin{defn} A \textbf{(Shimura) morphism} of connected mixed Shimura data $(Q,\cY^+)\rightarrow(P,\cX^+)$ is a homomorphism $\varphi\colon Q\rightarrow P$ of algebraic groups over $\Q$ which induces a map $\cY^+\rightarrow\cX^+$, $y\mapsto\varphi\circ y$. A Shimura morphism of connected mixed Shimura varieties is a morphism of varieties induced by a Shimura morphism of connected mixed Shimura data.
\end{defn}

\begin{fact}\label{quotient}(cf \cite[2.9]{PinkThesis})
Let $(P,\cX^+)$ be a connected mixed Shimura datum and let $P_0$ be a normal subgroup of $P$. Then there exists a quotient connected mixed Shimura datum $(P,\cX^+)/P_0$ and a morphism $(P,\cX^+)\rightarrow(P,\cX^+)/P_0$ unique up to isomorphism s.t. every morphism $(P,\cX^+)\rightarrow(P^\prime,\cX^{+\prime})$, where the homomorphism $P\rightarrow P^\prime$ factors through $P/P_0$, factors in a unique way through $(P,\cX^+)/P_0$. Such a Shimura morphism $(P,\cX^+)\rightarrow(P,\cX^+)/P_0$ is called a quotient Shimura morphism.
\end{fact}

\begin{notation}
For convenience, we fix some notation here. Given a connected mixed Shimura datum $(P,\cX^+)$, we always denote by $W=\cR_u(P)$ the unipotent radical of $P$, $G:=P/W$ the reductive part, $U\lhd P$ the weight $-2$ part, $V:=W/U$ the weight $-1$ part and $(P/U,\cX^+_{P/U}):=(P,\cX^+)/U$ (resp. $(G,\cX^+_G):=(P,\cX^+)/W$) the corresponding connected mixed Shimura datum whose weight $-2$ part is trivial (resp. pure Shimura datum). If we have several connected mixed Shimura data, say $(P,\cX^+)$ and $(Q,\cY^+)$, then we distinguish the different parts associated with them by adding subscript $W_P$, $W_Q$, $G_P$, $G_Q$, etc. For a connected mixed Shimura variety $S$, we denote by $S_{P/U}$ (resp. $S_G$) its image under the Shimura morphism induced by $(P,\cX^+)\rightarrow(P/U,\cX^+_{P/U})$ (resp. $(P,\cX^+)\rightarrow(G,\cX^+_G)$).
\end{notation}

\begin{prop}\label{WQ to WP GQ to GP}
Let $(Q,\cY)\xrightarrow{f}(P,\cX)$ be a Shimura morhpism, then $f(W_Q)\subset W_P$ (resp. $f(U_Q)\subset f(U_P)$), and hence $f$ induces
\[
\bar{f}:(G_Q,\cY_{G_Q})\rightarrow(G_P,\cX_{G_P})\qquad\text{(resp. }\bar{f}^\prime:(Q/Q_U,\cY_{Q/U_Q})\rightarrow(P/U_P,\cX_{P/U_P})).
\]
Furthermore, if the underlying homomorphism of algebraic groups $f$ is injective, then so are $\bar{f}$ and $\bar{f}^\prime$.
\begin{proof} Since
\[
\lie W_P=W_{-1}(\lie P)\qquad\text{and}\qquad\lie W_Q=W_{-1}(\lie Q),
\]
by the following commutative diagram
\[
\begin{diagram}
\lie W_Q &\rTo &\lie W_P \\
\dTo_\exp^\wr & &\dTo_\exp\\
W_Q &\rTo^f &P
\end{diagram}
\]
(here $\exp$ is algebraic and is an isomorphism as a morphism between algebraic varieties because $W_Q$ is unipotent), $f(W_Q)\subset W_P$.

Hence $f$ induces a map $G_Q\rightarrow G_P$. Now the existence of $\bar{f}$ follows from the universal property of the quotient Shimura datum (\cite[2.9]{PinkThesis}).

Furthermore, suppose now that $f$ is injective. By Levi decomposition, the exact sequence
\[
1\rightarrow W_Q\rightarrow Q\xrightarrow{\pi_Q}G_Q\rightarrow 1
\]
splits. Choose a splitting $s_Q\colon G_Q\rightarrow Q$, then we have the following diagram whose solid arrows commute:
\[
\begin{diagram}
1 &\rTo & W_Q &\rTo &Q &\rTo_{\pi_Q} \upperarrow{s_Q} &G_Q &\rTo &1\\
& &\dTo & &\dTo^{f} &\ldDotsto_{\lambda} & \dTo_{\bar{f}} & &\\
1 &\rTo & W_P &\rTo &P &\rTo_{\pi_P} &G_P &\rTo &1
\end{diagram},
\]
where $\lambda:=f\circ s_Q$. Then $\lambda$ is injective since $f$, $s_Q$ are. And $\pi_P\circ\lambda=\pi_P\circ f\circ s_Q=\bar{f}\circ\pi_Q\circ s_Q=\bar{f}$, so we have
\[
\ker(\bar{f})=G_Q\cap W_P
\]
where the intersection is taken in $P$. $(G_Q\cap W_P)^\circ$ is smooth (since we are in the characteristic $0$), connected unipotent (since it is in $W_P$) and normal in $G_Q$ (since $W_P$ is normal in $P$), so it is trivial since $G_Q$ is reductive. So $G_Q\cap W_P$ is finite, hence trivial because $W_P$ is unipotent over $\Q$. To sum it up, $\bar{f}$ is injective.

The proof for the statements with $U$'s is similar.
\end{proof}
\end{prop}

\subsection{Structure of the underlying group}\label{structure of the underlying group}
(cf \cite[2.15]{PinkThesis}).

For a given connected mixed Shimura datum $(P,\cX^+)$, we can associate to $P$ a 4-tuple $(G,V,U,\Psi)$ which is defined as follows:
\begin{itemize}
\item $G:=P/\cR_u(P)$ is the reductive part of $P$;
\item $U$ is the normal subgroup of $P$ as in Definition \ref{connected mixed Shimura datum} and $V:=\cR_u(P)/U$. Both of them are vector groups with an action of $G$ induced by conjugation in $P$ (which factors through $G$ for reason of weight);
\item The commutator on $W:=\cR_u(P)$ induces a $G$-equivariant alternating form $\Psi\colon V\times V\rightarrow U$ by reason of weight as explained by Pink in \cite[2.15]{PinkThesis}. Moreover, $\Psi$ is given by a polynomial with coefficients in $\Q$.
\end{itemize}

On the other hand, $P$ is uniquely determined up to isomorphism by this 4-tuple in the following sense:
\begin{itemize}
\item let $W$ be the central extension of $V$ by $U$ defined by $\Psi$. More concretely, $W=U\times V$ as a $\Q$-variety and the group law on $W$ is (this can be proved using the Baker-Campbell-Hausdorff formula)
\[
(u,v)(u^\prime,v^\prime)=(u+u^\prime+\frac{1}{2}\Psi(v,v^\prime),v+v^\prime);
\]
\item define the action of $G$ on $W$ by $g((u,v)):=(gu,gv)$;
\item define $P:=W\rtimes G$.
\end{itemize}

\subsection{Siegel type}\label{SubsectionSiegelType}

(cf \cite[2.7, 2.25]{PinkThesis} for mixed Shimura data of Siegel type and \cite[10.1-10.14]{PinkThesis} for mixed Shimura varieties of Siegel type)

Let $g\in\N_{>0}$. Let $V_{2g}$ be a $\Q$-vector space of dimension $2g$ and let
\[
\Psi\colon V_{2g}\times V_{2g}\rightarrow U_{2g}:=\G_{a,\Q}
\]
be a non-degenerate alternating form. Define
\[
\GSp_{2g}:=\{g\in\GL(V_{2g})|\Psi(gv,gv^\prime)=\nu(g)\Psi(v,v^\prime)\text{ for some }\nu(g)\in\G_m\},
\]
and $\H_g$ the set of all homomorphisms
\[
\S\rightarrow\GSp_{2g,\R}
\]
which induce a pure Hodge structure of type $\{(-1,0),(0,-1)\}$ on $V_{2g}$ and for which either $\Psi$ or $-\Psi$ defines a polarization. Let $\H^+_g$ be the set of all such homomorphisms s.t. $\Psi$ defines a polarization.

$\GSp_{2g}$ acts on $U_{2g}$ by the scalar $\nu$, which induces a pure Hodge structure of type $(-1,-1)$ on $U_{2g}$. Let $W_{2g}$ be the central extension of $V_{2g}$ by $U_{2g}$ defined by $\Psi$, then the action of $\GSp_{2g}$ on $W_{2g}$ induces a Hodge structure of type $\{(-1,0),(0,-1),(-1,-1)\}$ on $W_{2g}$.

By \cite[2.16, 2.17]{PinkThesis}, there exist connected mixed Shimura data $(P_{2g,a},\cX^+_{2g,a})$ and $(P_{2g},\cX^+_{2g})$, where $P_{2g,a}:=V_{2g}\rtimes\GSp_{2g}$ and $P_{2g}:=W_{2g}\rtimes\GSp_{2g}$.

\begin{defn}
 The connected mixed Shimura data $(\GSp_{2g},\H^+_g)$, $(P_{2g,a},\cX^+_{2g,a})$ and $(P_{2g},\cX^+_{2g})$ are called \textbf{of Siegel type} (of genus $g$).
\end{defn}

Next we introduce connected mixed Shimura varieties of Siegel type. They have good modular interpretation (\cite[10.8-10.14]{PinkThesis}).

For $N\geqslant4$ and even, define
\begin{equation}\label{lattice of level N}
 \Gamma_{\GSp}(N):=\{g\in\GSp_{2g}(\Z)|g\equiv 1\text{ mod }N\}
\end{equation}
and
\[
 \Gamma_W(N):=(N\cdot U_{2g}(\Z))\times(N\cdot V_{2g}(\Z))
\]
under the identification $W\cong U\times V$ in $\mathsection$\ref{structure of the underlying group}. $\Gamma_W(N)$ is indeed a 
subgroup of $W(\Z)$ by the group operation (defined by $\Psi$). Let $\Gamma_V(N):=N\cdot V_{2g}(\Z)$, and write
\begin{align}
 \cA_g(N):=\Gamma_{\GSp}(N)\backslash\H^+_g \\
 \mathfrak{A}_g(N):=(\Gamma_V(N)\rtimes\Gamma_{\GSp}(N))\backslash\cX^+_{2g,a} \\
 \mathfrak{L}_g(N):=(\Gamma_W(N)\rtimes\Gamma_{\GSp}(N))\backslash\cX^+_{2g},
\end{align}
Then $\cA_g(N)$ is a moduli space of abelian varieties of dimension $g$ with a level structure, $\mathfrak{A}_g(N)\rightarrow\cA_g(N)$ is the universal 
family of abelian varieties (and hence a principally polarized abelian scheme of relative dimension $g$), 
and $\mathfrak{L}_g(N)\rightarrow\mathfrak{A}_g(N)$ is a $\G_m$-torsor which (up to replacing the $\G_m$-action by its inverse) corresponds to a relatively ample line bundle over $\mathfrak{A}_g(N)\rightarrow\cA_g(N)$. For more details see \cite[10.5, 10.9, 10.10]{PinkThesis}.

\begin{defn}\label{MixedShimuraVarietySiegelType}
 The connected mixed Shimura varieties $\cA_g(N)$, $\mathfrak{A}_g(N)$ and $\mathfrak{L}_g(N)$ are called \textbf{of Siegel type of level $N$} (and of genus $g$).
\end{defn}

Denote by $\GSp_0:=\G_m$ and $P_0:=\G_a\rtimes\G_m$ with the standard action of $\G_m$ on $\G_a$. Pink proved the following lemma (\cite[2.26]{PinkThesis})
\begin{lemma}[Reduction Lemma]\label{reduction lemma}
Let $(P,\cX^+)$ be a connected mixed Shimura datum with generic Mumford-Tate group.
\begin{enumerate}
\item If $V$ is trivial, then there exists an embedding
\[
(P,\cX^+)\hookrightarrow(G_0,\cD^+)\times\prod_{i=1}^r(P_0,\cX^+_0)
\]
where $r=\dim(U)$ (see \cite[2.8, 2.14]{PinkThesis} for definition of $(P_0,\cX^+_0)$);
\item If $V$ is not trivial, then there exist a connected pure Shimura datum $(G_0,\cD^+)$ and Shimura morphisms
\[
(P^\prime,\cX^{\prime+})\twoheadrightarrow(P,\cX^+)
\]
\[
\text{and }\begin{diagram}
(P^\prime,\cX^{\prime+}) &\rInto^{\lambda} &(G_0,\cD^+)\times\prod_{i=1}^r(P_{2g},\cX^+_{2g})
\end{diagram}
\]
s.t. $\ker(P^\prime\rightarrow P)$ is of dimension 1 and of weight -2. Moreover $\lambda|_V\colon V\cong V_{2g}\rightarrow\oplus_{i=1}^rV_{2g}$ is the diagonal map,  $\lambda|_{U^\prime}\colon U^\prime\cong\oplus_{i=1}^rU_{2g}$ and $G\xrightarrow{\lambda|_G}G_0\times\prod_{i=1}^r\GSp_{2g}\rightarrow\GSp_{2g}$ is non-trivial for each projection.
\end{enumerate}
\begin{proof}
The statement except the last claim of the ``Moreover" part is \cite[2.26 statement \& pp 45]{PinkThesis}. For the last part, call $p_i\colon G\rightarrow\GSp_{2g}$ the composite with the i-th projection. If $p_i$ is trivial, then $p_i(P^\prime,\cX^{\prime+})$ is trivial since a connected mixed Shimura datum is trivial if its pure part is trivial. This contradicts the dimension of $V$.
\end{proof}
\end{lemma}

\subsection{Decomposition of $V$ and $U$ as $G$-modules}
We start this subsection by the following group theoretical proposition.

\begin{prop}\label{group theoretical big diagram proposition}
Let $1\rightarrow N\rightarrow Q\xrightarrow{\varphi}Q^\prime\rightarrow1$ be an exact sequence of algebraic groups over $\Q$.
Then the following diagram with solid arrows is commutative and all the lines and columns are exact:
\begin{equation*}\label{the big diagram}
\begin{diagram}
& & 1 & & 1 & & 1\\
& &\dTo & &\dTo & &\dTo \\
1 &\rTo &W_N:=\cR_u(N) &\rTo &N &\rTo_{\pi_N} \upperarrow{s_N} &G_N:=N/W_N &\rTo &1 \\
& &\dTo & &\dTo & &\dTo \\
1 &\rTo &W_Q:=\cR_u(Q) &\rTo &Q &\rTo_{\pi_Q} \upperarrow{s_Q} &G_Q:=Q/W_Q &\rTo &1 \\
& &\dTo & &\dTo^\varphi & &\dTo^{\bar{\varphi}} \\
1 &\rTo &W_{Q^\prime}:=\cR_u(Q^\prime) &\rTo &Q^\prime~~ &\rTo_{\pi_{Q^\prime}} \upperarrow{s_{Q^\prime}} &G_{Q^\prime}:=Q^\prime/W_{Q^\prime} &\rTo &1 \\
& &\dTo & &\dTo & &\dTo \\
& & 1 & & 1 & & 1
\end{diagram}.
\end{equation*}
Moreover, if we fix a morphism $s_Q$ which splits the middle line (such an $s_Q$ exists by Levi decomposition), then we can deduce $s_N$ and $s_{Q^\prime}$ which split the other two lines. Note that in this case, the action of $G_N$ on $W_{Q^\prime}$ induced by $s_Q$ is trivial.
\begin{proof} The two bottom lines are already exact. By group theory, we know $\varphi(W_Q(\bar{\Q}))=W_{Q^\prime}(\bar{\Q})$ (\cite[Corollary 14.11]{BorelLinear-Algebrai}), and since the set of closed points of $W_Q$ (resp. $W_{Q^\prime}$) is dense on $W_Q$ (resp. $W_{Q^\prime}$), we have $\varphi(W_Q)=W_{Q^\prime}$. In consequence, we have the map $\bar{\varphi}$, which is surjective since $\varphi$ is. Now we get the solid diagram by snake-lemma. $G_N$ is reductive (\cite[14.2 Corollary(b)]{BorelLinear-Algebrai}).

If we have an $s_Q$, then to get a desired $s_{Q^\prime}$ (and $s_N$) is equivalent to prove that $\varphi\circ s_Q(G_N)$ is trivial, i.e. the intersection of this image with $W_{Q^\prime}$ (in $Q^\prime$) is trivial and the projection of this image to $G_{Q^\prime}$ (under $\pi_{Q^\prime}$) is trivial. The projection is trivial by a simple diagram-chasing. The neutral component of the intersection is trivial since it is reductive and unipotent, and hence the intersection is trivial since $W_{Q^\prime}$ is unipotent over $\Q$. Now the triviality of the action of $G_N$ on $W_{Q^\prime}$ induced by $s_Q$ is automatic.
\end{proof}
\end{prop}

\begin{cor}\label{decomposition of V,U as G-module}
Let $(P,\cX^+)$ be a connected mixed Shimura datum. Suppose $N\lhd P$. Then there are decompositions
\[
V=V_N\oplus V_N^\bot\qquad\text{(resp. }U=U_N\oplus U_N^\bot\text{)}
\]
as $G$-modules, where $V_N:=V\cap N$ (resp. $U_N:=U\cap N$), s.t. the action of $G_N:=N/V_N$ on $V_N^\bot$ (resp. $U_N^\bot$) is trivial.
\begin{proof} To prove the decomposition of $V$, apply Proposition \ref{group theoretical big diagram proposition} to the exact sequence
\[
1\rightarrow V_N\rtimes G_N\rightarrow V\rtimes G\rightarrow (V/V_N)\rtimes(G/G_N)\rightarrow 1,
\]
then since $G$ is reductive, the vertical line on the left (in the diagram of the proposition) splits. The conjugation by $P$ on $V$ factors through $G$ by reason of weights, and hence equals to the action of $G$ on $V$ induced by any Levi decomposition $s_P$. So the action of $G_N$ on $V_N^\bot$ is trivial by the last assertion of  Proposition \ref{group theoretical big diagram proposition}.

To prove the decomposition of $U$, it suffices to apply Proposition \ref{group theoretical big diagram proposition} to the exact sequence
\[
1\rightarrow U_N\rtimes G_N\rightarrow U\rtimes G\rightarrow (U/U_N)\rtimes(G/G_N)\rightarrow 1.
\]
\end{proof}
\end{cor}

In fact we have a better result if $(P,\cX^+)$ has generic Mumford-Tate group.
\begin{prop} Let $(P,\cX^+)$ be a connected mixed Shimura datum s.t. $P=\MT(\cX^+)$. Suppose $N\lhd P$ s.t. $N$ possesses no non-trivial torus quotient. Then $G_N$ acts trivially on $U$.
\begin{proof}
By Reduction Lemma (Lemma \ref{reduction lemma}), we may assume that 
$(P,\cX^+)\hookrightarrow(G_0,\cD^+)\times\prod_{i=1}^r(P_{2g},\cX^+_{2g})$ ($g\geqslant0$). Since $N$ possesses no non-trivial torus quotient, $G_N$ is semi-simple (the last line of the proof of Proposition \ref{particular choices of i and varphi}). So
\[
G_N=G_N^{\der}<G^{\der}<(G_0\times\prod_{i=1}^r\GSp_{2g})^{\der}=G_0^{\der}\times\prod_{i=1}^r\Sp_{2g}
\]
where $\Sp_0:=1$. Hence $G_N$ acts trivially on $U$ since $G_0^{\der}\times\prod_{i=1}^r\Sp_{2g}$ acts trivially on $\oplus_{i=1}^rU_{2g}$.
\end{proof}
\end{prop}

\section{Variations of mixed ($\Z$-)Hodge structure}\label{Variations of mixed (Z-)Hodge structure}

\subsection{Arbitrary variation of mixed Hodge structure}
\begin{defn}(\cite[Definition 14.44]{PetersMixed-Hodge-Str}) Let $S$ be a complex manifold. A \textbf{variation of mixed Hodge structure} on $S$ is a triple $(\V,W_\cdot,\cF^\cdot)$ with
\begin{enumerate}
\item a local system $\V$ of free $\Z$-module of finite rank on $S$;
\item a finite increasing filtration $\{W_m\}$ of the local system $\V_\Q:=\V\otimes_\Z\Q$ by local sub-systems (this is called the weight filtration);
\item a finite decreasing filtration $\{\cF^p\}$ of the holomorphic vector bundle $\cV:=\V\otimes_\Z\cO_S$ by holomorphic sub-bundles (this is called the Hodge filtration).
\end{enumerate}
s.t. 
\begin{enumerate}
\item for each $s\in S$, the filtrations $\{\cF^p(s)\}$ and $\{W_m\}$ of $\V(s)\cong\V_s\otimes_\Z\C$ define a mixed Hodge structure on the $\Z$-module of finite rank $\V_s$;
\item the connection $\nabla\colon\cV\rightarrow\cV\otimes_{\cO_S}\Omega_S^1$ ($\cV:=\V\otimes_\Z\cO_S$) whose sheaf of horizontal sections is $\V_\C$ satisfies the Griffiths' transversality condition
\[
\nabla(\cF^p)\subset\cF^{p-1}\otimes\Omega_S^1.
\]
\end{enumerate}
\end{defn}

\begin{defn} A variation of mixed Hodge structure on $S$ is called \textbf{(graded-)polarizable} if the induced variations of pure Hodge structure $\Gr^W_k\V$ are all polarizable, i.e. for each $k$, there exists a flat morphism of variations
\[
Q_k\colon\Gr^W_k\V\otimes\Gr^W_k\V\rightarrow\Z(-k)_S
\]
which induces on each fibre a polarization of the corresponding Hodge structure of weight $k$.
\end{defn}

Let $\pi\colon\tilde{S}\rightarrow S$ be a universal covering and choose a trivialization $\pi^*\V\cong\tilde{S}\times V$. For $s\in S$,  $\MT_s\subset\GL(\cV_s)$ denote the Mumford-Tate group of its fibre. The choice of a point $\tilde{s}\in\tilde{S}$ with $\pi(\tilde{s})=s$ gives an identification $\cV_s\cong V$, whence an injective homomorphism $i_{\tilde{s}}\colon\MT_s\hookrightarrow\GL(V)$. By \cite[$\mathsection 4$, Lemma 4]{AndreMumford-Tate-gr}, on $S^\circ:=S\setminus\Sigma$ where $\Sigma$ is a meager subset of $S$, $M:=\im(i_{\tilde{s}})\subset\GL(V)$ does not depend on $s$, nor on the choice of $\tilde{s}$. We call $S^\circ$ the \textit{``Hodge-generic" locus} and the group $M$ the \textit{generic Mumford-Tate group} of $(\V,W_\cdot,\cF^\cdot)$.

On the other hand, if we choose a base-point $s\in S$ and a point $\tilde{s}\in\tilde{S}$ with $\pi(\tilde{s})=s$, then then local system $\V$ corresponds to a representation $\rho\colon\pi_1(S,s)\rightarrow\GL(V)$, called the monodromy representation. The algebraic monodromy group is defined as the smallest algebraic subgroup of $\GL(V)$ over $\Q$ which contains the image of $\rho$. We write $H^\mon_s$ for its connected component of the identity, called the \textit{connected algebraic monodromy group}. Given the trivialization of $\pi^*\V$, the group $H^\mon_s\subset\GL(V)$ is independent of the choice of $s$ and $\tilde{s}$.

Suppose now that $(\V,W_\cdot,\cF^\cdot)$ is (graded-)polarizable, then $H^\mon_s<M$ for any $s\in S^\circ$ by \cite[$\mathsection 4$, Lemma 4]{AndreMumford-Tate-gr}.

\subsection{Admissible variations of mixed Hodge structure}
We now recall the concept of ``admissible" variations of mixed Hodge structure which was introduced by Steenbrick-Zucker \cite{SteenbrinkVariation-of-mi} and studied by Kashiwara \cite{KashiwaraA-study-of-vari} and Hain-Zucker \cite{HainUnipotent-varia}. We give the definition here, but instead of the exact definition, we shall only use the notion of ``admissibility" and the fact that it can be defined using ``curve test". We will use $\Delta$ (resp. $\Delta^*$) to denote the unit disc (resp. punctured unit disc).
\begin{defn}(see \cite[Definition 14.49]{PetersMixed-Hodge-Str})
\begin{enumerate}
\item A variation of mixed Hodge structure $(\V,W_\cdot,\cF^\cdot)$ over the punctured unit disc $\Delta^*$ is called admissible if
\begin{itemize}
\item it is (graded-)polarizable;
\item the monodromy $T$ is unipotent and the weight filtration $M(N,W_\cdot)$ of $N:=\log T$ relative to $W_\cdot$ exists;
\item the filtration $\cF^\cdot$ extends to a filtration $\tilde{\cF}^\cdot$ of $\tilde{\cV}$ which induced $^k\tilde{\cF}$ on $\Gr^W_k\tilde{\cV}$ for each $k$.
\end{itemize}
\item Let $S$ be a smooth connected complex algebraic variety and let $\bar{S}$ be a compactification of $S$ s.t. $\bar{S}\setminus S$ is a normal crossing divisor. A (graded-)polarizable variation of mixed Hodge structure $(\V,W_\cdot,\cF^\cdot)$ on $S$ is called admissible if for every holomorphic map $i\colon\Delta\rightarrow\bar{S}$ which maps $\Delta^*$ to $S$ and s.t. $i^*\V$ has unipotent monodromy, the variation $i^*(\V,W_\cdot,\cF^\cdot)$ is admissible. (This definition is sometimes called the ``curve test" version).
\end{enumerate}
\end{defn}
\begin{rmk} This definition is equivalent to the one given in \cite[1.5]{HainUnipotent-varia}. See \cite[Properties 3.13 \& Appendix]{SteenbrinkVariation-of-mi}, \cite[$\mathsection 1$ \& Theorem 4.5.2]{KashiwaraA-study-of-vari} and \cite[1.5]{HainUnipotent-varia} for details.
\end{rmk}
The following lemma is an easy property of admissibility and is surely known by many people, but I cannot find any reference, so I give a proof here.
\begin{lemma}\label{admissible stable under restriction}
Let $S$ be a smooth connected complex algebraic variety and let $(\V,W_\cdot,\cF^\cdot)$ be an admissible variation of mixed Hodge structure on $S$. Then for any smooth connected (not necessarily closed) subvariety $j\colon Y\hookrightarrow S$, $j^*(\V,W_\cdot,\cF^\cdot)$ is also admissible on $Y$.
\begin{proof} Take smooth compactifications $\bar{Y}$ of $Y$ and $\bar{S}$ of $S$ s.t. $\bar{Y}\setminus Y$ and $\bar{S}\setminus S$ are normal crossing divisors and s.t. $j\colon Y\hookrightarrow S$ extends to a morphism $\bar{j}\colon\bar{Y}\rightarrow\bar{S}$. This can be done by
first choosing any compactifications of $Y^\cp$ of $Y$ and $S^\cp$ of $S$ with normal crossing divisors and then taking a suitable resolution of singularities of the closure of the graph of $j$ in $Y^\cp\times S^\cp$. Now the conclusion follows from our ``curve test'' version of the definition.
\end{proof}
\end{lemma}

\subsection{Consequences of admissibility}
Y.Andr\'{e} proved that:
\begin{thm}\label{Y Andre monodromy normal}
Let $(\V,W_\cdot,\cF^\cdot)$ be an admissible variation of mixed Hodge structure over a smooth connected complex algebraic variety $S$, then for any $s\in S$, the connected monodromy group $H^\mon_s$ is a normal subgroup of the generic Mumford-Tate group $M$ and also its derived group $M^\der$.
\begin{proof} \cite[$\mathsection 5$, Theorem 1]{AndreMumford-Tate-gr} states that $H^\mon_s\lhd M^\der$, and in the proof he first proved that $H^\mon_s\lhd M$.
\end{proof}
\end{thm}

Now we state a theorem which roughly says that all the variations of mixed Hodge structure obtained from representations of the underlying group of a connected mixed Shimura datum are admissible. Explicitly, let $S$ be a connected mixed Shimura variety associated with the connected mixed Shimura datum $(P,\cX^+)$ and let $\unif\colon \cX^+\rightarrow S=\Gamma\backslash\cX^+$ be the uniformization. Suppose that $\Gamma$ is neat. Consider any $\Q$-representation $\xi\colon P\rightarrow\GL(V)$. By \cite[Proposition 4.2]{AlgebraicGroupBible}, there exists a $\Gamma$-invariant lattice $V_\Z$ of $V$. $\xi$ and $V_\Z$ together give rise to a VMHS on $S$ whose underlying local system is $\Gamma\backslash(\cX^+\times V_\Z)$. This variation is (graded-)polarizable by \cite[1.18(d)]{PinkThesis}.
\begin{thm}\label{all variations of mixed Shimura varieties are admissible}
Suppose that $S$, $(P,\cX^+)$, $\xi\colon P\rightarrow\GL(V)$ and $V_\Z$ are as in the previous paragraph, then the variation of mixed Hodge structure obtained as above is admissible.
\begin{proof}\cite[Theorem 2.2]{WildeshausThe-canonical-c} says that the corresponding $\Q$-variation is admissible, and $\Gamma$ gives a $\Z$-structure as in the discussion above.
\end{proof}
\end{thm}

\begin{rmk}\label{stabilizer of Hodge generic point}
In this language, we can rephrase Definition \ref{irreducible mixed Shimura datum} as: \textit{$P$ is the generic Mumford-Tate group (of the variation in Theorem \ref{all variations of mixed Shimura varieties are admissible})}. It is clear that for any Hodge generic point $x\in\cX^+$, the only $\Q$-subgroup $N$ of $P^{\der}$ s.t. $N(\R)^+U_N(\C)$, where $U_N:=U\cap N$, stabilizes $x$ is the trivial group.
\end{rmk}

\section{Realization of $\cX^+$}\label{Realization of X}
Let $(P,\cX^+)$ be a connected mixed Shimura datum. We first define the dual $\cX^\vee$ of $\cX^+$ (see \cite[1.7(a)]{PinkThesis} or \cite[Chapter VI, Proposition 1.3]{MilneCanonical-model}):

Let $M$ be a faithful representation of $P$ and take any $x_0\in\cX^+$. The weight filtration on $M$ is constant, so the Hodge filtration $x\mapsto\text{Fil}^\cdot_x(M_\C)$ gives an injective map $\cX^+\hookrightarrow\text{Grass}(M)(\C)$ to a certain flag variety. In fact, this injective map factors through
\[
\cX^+=P(\R)^+U(\C)/C(x_0)\hookrightarrow P(\C)/F^0_{x_0}P(\C)\hookrightarrow\text{Grass}(M)(\C)
\]
where $C(x_0)$ is the stabilizer of $x_0$ in $P(\R)^+U(\C)$. The first injection is an open immersion (\cite[1.7(a)]{PinkThesis} or \cite[Chapter VI, (1.2.1)]{MilneCanonical-model}). We define the dual $\cX^\vee$ of $\cX^+$ to be
\[
\cX^\vee:=P(\C)/F^0_{x_0}P(\C).
\]
Then $\cX^\vee$ is clearly a connected smooth complex algebraic variety.

\begin{prop} Under the open immersion $\cX^+\hookrightarrow\cX^\vee$, $\cX^+$ is realized as a semi-algebraic set which is also a complex manifold.
\begin{proof} $\cX^+$ is smooth since it is a homogeneous space, and the open immersion endows it with a complex structure. For semi-algebraicity, consider the diagram
\[
\begin{diagram}
\cX^+ &\rInto &\cX^\vee \\
\dTo^{\pi} & &\dTo^{\pi^\vee} \\
\cX^+_G &\rInto &\cX^\vee_G
\end{diagram}.
\]
Now $\cX^+=\{x\in\cX^\vee|\pi^\vee(x)\in\cX^+_G\}$ and $\pi^\vee$ is algebraic, so the conclusion follows from \cite[Lemme 2.1]{UllmoQuelques-applic}.
\end{proof}
\end{prop}

\begin{rmk} It is not hard to see that $\cX^\vee$ is a projective variety if and only if $(P,\cX^+)$ is pure. The argument is as follows: $\cX^\vee$ is a holomorphic vector bundle over $\cX_G^\vee$ where the fibre is homeomorphism to $W(\R)U(\C)$. But $\cX_G^\vee$ is projective, so $\cX^\vee$ is projective if and only if it is a trivial vector bundle over $\cX_G^\vee$, i.e. if and only if $W$ is trivial.
\end{rmk}

Let us take a closer look at the semi-algebraic structure of $\cX^+$. By \cite[pp 6]{WildeshausThe-canonical-c}, there exists a Shimura morphism $i\colon (G,\cX^+_G)\rightarrow (P,\cX^+)$ s.t. $\pi\circ i=\text{id}$. Then $i$ defines a Levi decomposition of $P=W\rtimes G$.
By definition $\cX^+\subset\hom(\S_\C,P_\C)$. Define a bijective map
\[
\begin{diagram}
W(\R)U(\C)\times\cX^+_G &\rTo &\cX^+ \\
(w,x) &\mapsto &\text{int}(w)\circ i(x)
\end{diagram}.
\]

Identify $P$ with the 4-tuple $(G,V,U,\Psi)$ as in $\mathsection$\ref{structure of the underlying group}.
Since $W\cong U\times V$ as $\Q$-varieties, we can define a bijection induced by the one above
\begin{equation}\label{identification of the realization}
\rho: U(\C)\times V(\R)\times\cX^+_G\xrightarrow{\sim}\cX^+
\end{equation}

$P(\R)^+U(\C)$ acts on $\cX^+$ by definition. There is also a natural action of $P(\R)^+U(\C)$ on 
$U(\C)\times V(\R)\times\cX^+_G$ which is defined as follows. 
Under the notation of $\mathsection$\ref{structure of the underlying group}, for any $(u,v,g)\in P(\R)^+U(\C)$ and $(u^\prime,v^\prime,x)\in U(\C)\times V(\R)\times\cX^+_G$,
\begin{equation}\label{action of P on X}
(u,v,g)\cdot(u^\prime,v^\prime,x):=(u+gu^\prime+\frac{1}{2}\Psi(v,v^\prime),v+gv^\prime,gx).
\end{equation}
This action is algebraic since $\Psi$ is a polynomial over $\Q$ (see $\mathsection$2.2). 
The morphism $\rho$ is $P(\R)^+U(\C)$-equivariant by an easy calculation.

\begin{prop}\label{realization of the uniformizing space}
The map $\rho$ is semi-algebraic. 
\begin{proof}
It is enough to prove that the graph of $\rho$ is semi-algebraic.
This is true since $\rho$ is $P(\R)^+U(\C)$-equivariant and the actions of $P(\R)^+U(\C)$ on both sides are algebraic and transitive.
Explicitly, fix a point $x_0\in U(\C)\times V(\R)\times\cX^+_G$, the graph of $\rho$
\begin{align*}
 Gr(\rho) &=\{(gx_0,\rho(gx_0))\in(U(\C)\times V(\R)\times\cX^+_G)\times\cX^+|~g\in P(\R)^+U(\C)\}\qquad\text{(transitivity)} \\
 &=\{(gx_0,g\rho(x_0))\in(U(\C)\times V(\R)\times\cX^+_G)\times\cX^+|~g\in P(\R)^+U(\C)\}\qquad\text{(equivariance)} \\
 &=P(\R)^+U(\C)\cdot (x_0,\rho(x_0))
\end{align*}
is semi-algebraic since the action of $P(\R)^+U(\C)$ on $(U(\C)\times V(\R)\times\cX^+_G)\times\cX^+$ is algebraic.
\end{proof}
\end{prop}

\begin{rmk}\label{GoodFundamentalDomainForUniversalFamily}
If $U$ is trivial, then the complex structure of $\cX^+$ given by $\cX^\vee$ is the same as the one given by \cite[Construction 2.9]{PinkA-Combination-o} since for the projection $\cX^+\xrightarrow{\pi}\cX^+_G$, the complex structure of any fibre $\cX^+_{x_G}$ ($x_G\in\cX^+_G$) given by $\cX^\vee$ is the same as the one obtained from $\cX^+_{x_G}\cong V(\C)/F^0_{x_G}V(\C)$ (see \cite[3.13, 3.14]{PinkThesis}). In particular this holds for $\cX^+_{2g,a}$ (see $\mathsection$\ref{SubsectionSiegelType} for notation). Therefore for any $\mathfrak{A}_g(N)$, the fundamental set $[0,N)^{2g}\times\cF_G\subset V_{2g}(\R)\times\H^+_g\cong\cX^+_{2g,a}$ is the one considered in \cite{PeterzilDefinability-of}.
\end{rmk}

\section{(Weakly) special subvarieties}\label{(Weakly) special subvarieties}
\subsection{Weakly special subvarieties}
\begin{defn}(Pink, \cite[Definition 4.1(b)]{PinkA-Combination-o})\label{definition of weakly special subvariety}
Let $S$ be a connected mixed Shimura variety.
Consider any Shimura morphisms $T^\prime\xleftarrow{[\varphi]}T\xrightarrow{[i]}S$ and any point $t^\prime\in T^\prime$. Then any irreducible component of $[i]([\varphi]^{-1}(t^\prime))$ is called a \textbf{weakly special subvariety} of $S$. We will prove later in Remark \ref{weakly special closed} that weakly special subvarieties of $S$ are indeed closed subvarieties.
\end{defn}

Since any Shimura morphism is related to a Shimura morphism between Shimura data, we will try to rephrase this definition in the context of Shimura data:
\begin{defn}\label{definition of Pink}
Given a connected mixed Shimura datum $(P,\cX^+)$, a weakly special subset of $\cX^+$ is a connected component of $i(\varphi^{-1}(y^\prime))\subset\cX^+$ for a point $y^\prime\in\cY^{\prime +}$, where $i$, $\varphi$, $\cY^{\prime +}$ are in the following diagram of Shimura morphisms
\[
\begin{diagram}
& & (Q,\cY^+) & &\\
&\ldTo^{\varphi} & &\rdTo^{i} &\\
(Q^\prime,\cY^{\prime +}) & & & &(P,\cX^+)
\end{diagram}.
\]
\end{defn}

\begin{rmk}\label{discussions of weakly special on the top}
\begin{enumerate}
\item In the definition above, let $N:=\ker(Q\rightarrow Q^\prime)$ and let $U_N:=U_Q\cap N$, then $i(\varphi^{-1}(y^\prime))$ is a connected component of $N(\R)U_N(\C)y$ where $\varphi(y)=y^\prime$. So $i(\varphi^{-1}(y^\prime))$ is smooth as an analytic variety. In particular, its connected components and complex analytic irreducible components coincide. As a result, we can replace ``a connected component" by ``a complex analytic irreducible component" in Definition \ref{definition of Pink}.
\item If furthermore $N$ is connected, then $i(\varphi^{-1}(y^\prime))$ itself is connected (hence also complex analytic irreducible). The proof is as follows: Consider the image of $\varphi^{-1}(y^\prime)$ under the projection $(Q,\cY^+)\xrightarrow{\pi}(G_Q,\cY^+_{G_Q}):=(Q,\cY^+)/W_Q$. By the decomposition (\cite[3.6]{MoonenLinearity-prope})
\[
(G_Q^\ad,\cY^+_{G_Q})=(G_N^\ad,\cY^+_1)\times(G_2,\cY^+_2)
\]
where $G_N:=N/W\cap N$, $\pi(\varphi^{-1}(y^\prime))=\cY^+_1\times\{y_2\}$. So $\pi(\varphi^{-1}(y^\prime))=G_N(\R)^+\pi(y)$. But $W_N(\R)U_N(\C)$ ($W_N:=W\cap N$) is connected, hence $\varphi^{-1}(y^\prime)=N(\R)^+U_N(\C)y$, which is connected. In consequence, $i(\varphi^{-1}(y^\prime))$ also is connected.
\end{enumerate}
\end{rmk}

\begin{prop}\label{particular choices of i and varphi}
For any weakly special subvariety of $S$ (resp. weakly special subset of $\cX^+$), the Shimura morphisms in Definition \ref{definition of weakly special subvariety} (resp. Definition \ref{definition of Pink}) can be chosen such that
\begin{itemize}
\item the underlying homomorphism of algebraic groups $i$ is injective, and hence $i$ is an embedding in the sense of \cite[2.3]{PinkThesis};
\item the underlying homomorphism of algebraic groups $\varphi$ is surjective, and its kernel $N$ is connected. Moreover, $N$ possesses no non-trivial torus quotient (or equivalently, $G_N:=N/(W\cap N)$ is  semi-simple);
\item $\varphi$ is a quotient Shimura morphism.
\end{itemize}
\begin{proof} If $P=\MT(\cX^+)$, then the first two points except the statement in the bracket are proved by \cite[Proposition 4.4]{PinkA-Combination-o}. The general case follows directly from Proposition \ref{properties of irreducible mixed Shimura datum}(1). The third assertion can be proved by the universal property of quotient Shimura data (\cite[2.9]{PinkThesis}). Now we are left to prove the statement in the bracket.

$G_N\lhd G$ since $G_N=N/(W\cap N)\hookrightarrow G=P/W$ and $N\lhd P$, and hence $G_N$ is reductive (\cite[14.2, Corollary(b)]{BorelLinear-Algebrai}). By \cite[14.2 Proposition(2)]{BorelLinear-Algebrai}, $G_N$ is the almost-product of $G_N^{\der}$ and $Z(G_N)^\circ$, and $Z(G_N)^\circ$ equals the radical of $G_N$ which is a torus. So $N$ possesses no non-trivial torus quotient iff $G_N$ possesses no non-trivial torus quotient iff $G_N$ is semi-simple.
\end{proof}
\end{prop}

\begin{rmk}\label{weakly special closed}
We can now prove that weakly special subvarieties of $S$ are closed. By the proposition above, we can choose $i$ to be injective. Then $[i]$ is finite by \cite[3.8]{PinkThesis}. Hence $[i]([\varphi]^{-1}(t^\prime))$ is closed.
\end{rmk}

\begin{lemma}\label{weakly special preparation}
Suppose that the Shimura morphisms $T^\prime\xleftarrow{[\varphi]}T\xrightarrow{[i]}S$ are associated with the morphisms of mixed Shimura data
\[
(Q^\prime,\cY^{\prime +})\xleftarrow{\varphi}(Q,\cY^+)\xrightarrow{i}(P,\cX^+)
\]
so that we have the following commutative diagram
\[
\begin{diagram}
\cY^{\prime +} &\lTo^\varphi &\cY^+ &\rTo^i &\cX^+ \\
\dTo^{\unif_{\cY^{\prime +}}} & &\dTo^{\unif_{\cY^+}} & &\dTo^{\unif_{\cX^+}} \\
T^\prime=\Delta^\prime\backslash\cY^{\prime +} &\lTo^{[\varphi]} &T=\Delta\backslash\cY^+ &\rTo^{[i]} &S=\Gamma\backslash\cX^+
\end{diagram},
\]
then for any point $y^\prime\in\cY^{\prime +}$, any irreducible component of $\unif_{\cX^+}(i(\varphi^{-1}(y^\prime)))$ is also an irreducible component of $[i]([\varphi]^{-1}(\unif_{\cY^{\prime +}}(y^\prime)))$.
\begin{proof} Let $N:=\ker(\varphi)$ and let $U_Q$ be the weight $-2$ part of $Q$, then we have
\[
\unif_{\cX^+}(i(\varphi^{-1}(y^\prime)))\subset[i]([\varphi]^{-1}(\unif_{\cY^{\prime +}}(y^\prime))),
\]
and both of them are of constant dimension $d$, where $d$ is the dimension of any orbit of $N(\R)^+(U_Q\cap N)(\C)$. This allows us to conclude.
\end{proof}
\end{lemma}

The following Proposition tells us that the two definitions of weak specialness are compatible.
\begin{prop}\label{definitions of weakly special match}
Let $S$ be a connected mixed Shimura variety associated with the connected mixed Shimura datum $(P,\cX^+)$ and let $\unif\colon \cX^+\rightarrow S=\Gamma\backslash\cX^+$ be the uniformization, then a subvariety $Z$ of $S$ is weakly special if and only if $Z$ is the image of some weakly special subset of $\cX^+$.
\begin{proof} The ``if" part is immediate by Lemma \ref{weakly special preparation}. We prove the ``only if" part. We assume that $i$, $\varphi$ are as in Proposition \ref{particular choices of i and varphi}. For any weakly special subvariety $Z\subset S$, suppose that we have a diagram as in Lemma \ref{weakly special preparation} and $Z$ is an irreducible component of $[i]([\varphi]^{-1}(t^\prime))$. Since
\[
[i]([\varphi]^{-1}(t^\prime))\subset\bigcup_{y^\prime\in \unif_{\cY^+}^{-1}(t^\prime)}\unif_{\cX^+}(i(\varphi^{-1}(y^\prime)))=\unif_{\cX^+}(i(\varphi^{-1}(\unif_{\cY^+}^{-1}(t^\prime)))),
\]
there exists a $y^\prime\in\cY^{\prime+}$ lying over $t^\prime$ s.t. $Z$ is an irreducible component of $\unif_{\cX^+}(i(\varphi^{-1}(y^\prime)))$ by Lemma \ref{weakly special preparation}. By Remark \ref{discussions of weakly special on the top}.2, $i(\varphi^{-1}(y^\prime))$ is complex analytic irreducible, so $\unif_{\cX^+}(i(\varphi^{-1}(y^\prime)))$ is also complex analytic irreducible when $S$ is regarded as an analytic variety. Hence $\unif_{\cX^+}(i(\varphi^{-1}(y^\prime)))$ is irreducible as an algebraic variety. So $Z=\unif_{\cX^+}(i(\varphi^{-1}(y^\prime)))$.
\end{proof}
\end{prop}

We close this subsection by proving that this definition of weakly special subvarieties is compatible with the one (which is already known) for pure Shimura varieties.

\begin{prop}\label{weakly special for pure}
A weakly special subvariety of a pure Shimura variety $S$ is a subvariety of the same form as in \cite[Definition 2.1]{UllmoA-characterisat}.
\begin{proof} This is pointed out in \cite[Remark 4.5]{PinkA-Combination-o}. We give a (relatively) detailed proof here. We prove the result for weakly special subsets. Assume that $S$ is associated with the connected pure Shimura datum $(P,\cX^+)$. For a subset of the same form as in \cite[Definition 2.1]{UllmoA-characterisat}, take $(Q,\cY^+)=(H,X_H^+)$ and $(Q^\prime,\cY^{\prime+})=(H_1,X_1^+)$ (same notation as \cite[Definition 2.1]{UllmoA-characterisat}). Then by definition such a subset is weakly special (as in Definition \ref{definition of Pink}).

On the other hand, suppose that we have a weakly special subset $\tilde{F}$ defined by a diagram as in Definition \ref{definition of Pink} satisfying Proposition \ref{particular choices of i and varphi}. Let $N:=\ker(\varphi)$, then the homogeneous spaces of the connected pure Shimura data $(Q^\prime,\cY^{\prime+})=(Q,\cY^+)/N$ and $(Q,\cY^+)/Z(Q)N=(Q^\ad,\cY^{\ad+})/N^\ad$ are canonically isomorphic to each other (\cite[Proposition 5.7]{MilneIntroduction-to}). Hence we may replace $(Q^\prime,\cY^{\prime+})$ by $(Q^\ad,\cY^{\ad+})/N^\ad$. But by \cite[3.6, 3.7]{MoonenLinearity-prope}, $(Q^\ad,\cY^{\ad+})=(N^\ad,\cY^+_1)\times(Q_2,\cY^+_2)$. So $\tilde{F}$ is of the same form as in \cite[Definition 2.1]{UllmoA-characterisat}. 
\end{proof}
\end{prop}

\subsection{Special subvarieties}

\begin{defn} Let $S$ be a connected mixed Shimura variety associated with the connected mixed Shimura datum $(P,\cX^+)$.
\begin{enumerate}
\item A \textbf{special subvariety} of $S$ is the image of any Shimura morphism $T\rightarrow S$ of connected mixed Shimura varieties;
\item A point $x\in\cX^+$ and its image in $S$ are called special if the homomorphism $x:\S_\C\rightarrow P_\C$ factors through $T_\C$ for a torus $T\subset P$.
\end{enumerate}
\end{defn}

\begin{rmk} By definition, $x\in\cX^+$ is special if and only if it is the image of a Shimura morphism $(T,\cY^+)\hookrightarrow(P,\cX^+)$. Hence a special point is just a special subvariety of dimension $0$.
\end{rmk}

The following result is easy to prove.
\begin{lemma} Let $S$ be a connected mixed Shimura variety associated with the connected mixed Shimura datum $(P,\cX^+)$ and let $\unif\colon \cX^+\rightarrow S$ be the uniformizing map, then a subvariety of $S$ is special if and only if it is of the form $\unif(\cY^+)$ for some $(Q,\cY^+)\hookrightarrow(P,\cX^+)$.
\end{lemma}

\begin{prop} Every special subvariety of $S$ contains a Zariski dense subset of special points.
\begin{proof}
\cite[Proposition 4.14]{PinkA-Combination-o}.
\end{proof}
\end{prop}

The relation between special and weakly special subvarieties is:
\begin{prop}\label{relation between special and weakly special}
A subvariety of $S$ is special if and only if it is weakly special and contains one special point.
\begin{proof}
\cite[Proposition 4.2, Proposition 4.14, Proposition 4.15]{PinkA-Combination-o}.
\end{proof}
\end{prop}

\section{Algebraicity in the uniformizing space}\label{Algebraicity in the uniformizing space}
\begin{defn}\label{algebraicity on the top}
Let $\tilde{Y}$ be an analytic subset of $\cX^+$, then
\begin{enumerate}
\item $\tilde{Y}$ is called an irreducible algebraic subset of $\cX^+$ if it is a complex analytic irreducible component of the intersection of its Zariski closure in $\cX^\vee$ and $\cX^+$;
\item $\tilde{Y}$ is called algebraic if it is a finite union of irreducible algebraic subsets of $\cX^+$.
\end{enumerate}
\end{defn}

\begin{lemma}\label{weakly special subsets are irreducible algebraic}
Any weakly special subset of $\cX^+$ is irreducible algebraic.
\begin{proof} Suppose that $\tilde{Z}$ is a weakly special subset of $\cX^+$. Use the notation of Definition \ref{definition of Pink} and assume that $i$ and $\varphi$ satisfy the properties in Proposition \ref{particular choices of i and varphi}. Let $N:=\ker(Q\rightarrow Q^\prime)$ and let $y$ be a point of the weakly special subset, then $\tilde{Z}=N(\R)^+U_N(\C)y$ is complex analytic irreducible by Remark \ref{discussions of weakly special on the top}.2. But $N(\R)^+U_N(\C)y=N(\C)y\cap\cX^+$ and $N(\C)y$ is algebraic, so $\tilde{Z}$ is irreducible algebraic by definition.
\end{proof}
\end{lemma}

\begin{lemma}[functoriality of algebraicity]\label{algebraicity functorial using Borel embedding}
Let $f\colon (Q,\cY^+)\rightarrow(P,\cX^+)$ be a Shimura morphism. Then there exists a unique morphism $f^\vee\colon \cY^{\vee}\rightarrow\cX^\vee$ of algebraic varieties such that the diagram commutes:
\[
\begin{diagram}
\cY^+ &\rInto &\cY^\vee \\
\dTo^f & &\dTo^{f^\vee} \\
\cX^+ &\rInto &\cX^\vee
\end{diagram}.
\]
Furthermore, for any irreducible algebraic subset $\tilde{Z}$ of $\cY^+$, the closure w.r.t the archimedean topology of $f(\tilde{Z})$ is irreducible algebraic in $\cX^+$ and $f(\tilde{Z})$ contains a dense open subset of this closure.

In particular, if $f$ is an embedding, then an irreducible algebraic subset of $\cY^+$ is an irreducible component of the intersection of an irreducible algebraic subset of $\cX^+$ with $\cY^+$.
\begin{proof}
Fix a point $x_0\in\cY^+$, then we have
\[
\begin{diagram}
\cY^+=Q(\R)^+U_Q(\C)/C(x_0) &\rInto &\cY^\vee=Q(\C)/F_{x_0}^0Q(\C) \\
\dTo^f & &\dTo^{f^\vee} \\
\cX^+=P(\R)^+U_P(\C)/C(f(x_0)) &\rInto &\cX^\vee=P(\C)/F_{f(x_0)}^0P(\C)
\end{diagram},
\]
where $C(x_0)$ (resp. $C(f(x_0))$) denotes the stabilizer of $x_0$ (resp. $f(x_0)$) in $Q(\R)U_Q(\C)$ (resp. $P(\R)U_P(\C)$). $f^\vee$ is unique since $Q(\R)U_Q(\C)/C(x_0)$ is dense in $\cY^\vee$.

To prove the second statement, it is enough to prove the result for $f^\vee(\bar{\tilde{Z}}^{\Zar})\subset\cX^\vee$  where $\bar{\tilde{Z}}^{\Zar}$ is the Zariski closure of $\tilde{Z}$ in $\cY^\vee$. This is then an algebro-geometric result, which follows easily from Chevalley's Theorem (\cite[Chapitre IV, 1.8.4]{EGA4.1}) and \cite[I.10, Theorem 1]{MumfordThe-red-book-of}.
\end{proof}
\end{lemma}

\section{Results for the unipotent part}\label{Results for the unipotent part}

Given a connected mixed Shimura variety $S$, let $S_G$ be its pure part so that we have a projection $S\rightarrow S_G$. For any point $b\in S_G$, denote by $E$ the fiber $S_b$.
Suppose that $S$ is associated with the mixed Shimura datum $(P,\cX^+)$, which can be further assumed to be irreducible by Proposition \ref{properties of irreducible mixed Shimura datum}. Let $\unif\colon \cX^+\rightarrow S=\Gamma\backslash\cX^+$ be the uniformization. Now $E=S_b\cong\Gamma_W\backslash W(\R)U(\C)$ with the complex structure determined by $b\in S_G$ ($E=S_b=\Gamma_W\backslash W(\C)/F^0_bW(\C)$), where $\Gamma_W:=\Gamma\cap W(\Q)$. Write $T:=\Gamma_U\backslash U(\C)$ and $A:=\Gamma_A\backslash V(\C)/F^0_bV(\C)$ where $\Gamma_U:=\Gamma\cap U(\Q)$ and $\Gamma_V:=\Gamma_W/\Gamma_U$, then $T$ is an algebraic torus over $\C$, $A$ is a complex abelian variety and $E$ is an algebraic torus over $A$ whose fibers are isomorphic to $T$.

\begin{lemma}\label{group automatically commutative}
If $E$ can be given the structure of an algebraic group whose group law is compatible with that of $W$, then $W$ (hence $E$) is commutative. In this case $E$ is a semi-abelian variety.
\begin{proof} If $E$ is an algebraic group, then $T$ is a normal subgroup of $E$. Hence $E$ acts on $T$ by conjugation, and this action factors via $A$, and then it is trivial by \cite[8.10 Proposition]{BorelLinear-Algebrai}. Therefore $T$ is in the center of $E$. Now consider the commutator pairing $E\times E\rightarrow E$. This factors through a morphism $A\times A\xrightarrow{f} T$. But this morphism is then constant. So the commutator pairing $E\times E\rightarrow E$ is trivial, and hence $E$ is commutative.

The commutator pairing $W\times W\rightarrow W$ induces an alternating form $\Psi\colon V\times V\rightarrow U$ (same as $\mathsection$\ref{structure of the underlying group}) which induces the morphism $f$ above. We have proved in the last paragraph that $\Psi(V(\R),V(\R))\subset\Gamma_U$ with $\Gamma_U:=\Gamma\cap U(\Q)$. But $\Psi(V(\R),v)$ is continuous for any $v\in V(\R)$ and $\Psi(0,V(\R))=0$, so $\Psi(V(\R),V(\R))=0$. Hence the commutator pairing $W\times W\rightarrow W$ is trivial, and therefore $W$ is commutative.
\end{proof}
\end{lemma}

\subsection{Weakly special subvarieties of a complex semi-abelian variety}
\begin{prop}\label{weakly special for semi-abelian}
Use the notation as at the beginning of the section. Weakly special subvarieties of $E$ are precisely the subsets of $E$ of the form
\[
\unif(W_0(\R)U_0(\C)\tilde{z})
\]
where $W_0$ is a $\MT(b)$-subgroup of $W$ (i.e. a subgroup of $W$ normalized by $\MT(b)$), $U_0:=W_0\cap U$, $\unif(\tilde{z}_G)=b$ and $\tilde{z}_V\in(N_W(W_0)/U)(\R)$ ($\tilde{z}=(\tilde{z}_U,\tilde{z}_V,\tilde{z}_G)$ under \eqref{identification of the realization}).

In particular, if $E$ can be given the structure of an algebraic group whose group law is compatible with that of $W$ (i.e. $W$ is commutative), then the weakly special subvarieties of $E$ are precisely the translates of subgroups of $E$.
\begin{proof}
Let $Z$ be a weakly special variety of $E$ and let $\tilde{Z}$ be a complex analytic irreducible component of $\unif^{-1}(Z)$, 
then there exists a diagram as in Definition \ref{definition of Pink} s.t. $\tilde{z}\colon \S_\C\rightarrow P_\C$ factors through $Q_\C$, $N\lhd Q$ and $\tilde{Z}=N(\R)^+U_N(\C)\tilde{z}$ for some $\tilde{z}\in\tilde{Z}$. As is explained in \cite[paragraph 2, pp 265]{PinkA-Combination-o}, $G_N=1$. We prove that $N=W_N$ satisfies the conditions which we require. Let $U_N:=W_N\cap U$, then $U_N$ is a $\MT(b)$-module by Proposition \ref{properties of irreducible mixed Shimura datum}(2). Denote by $V_N:=W_N/U_N$, $\pi_{P/U}\colon (P,\cX^+)\rightarrow (P/U,\cX_{P/U}^+)$ and $[\pi_{P/U}]\colon S\rightarrow S_{P/U}$. Then $[\pi_{P/U}](Z)$ is a subvariety of $A$ since $Z$ is a subvariety of $E$. So $\pi_{P/U}(\tilde{Z})=V_N(\R)+\pi_{P/U}(\tilde{z})$ is translate of a complex subspace of $V(\R)=V(\C)/F_b^0V(\C)$, and therefore $V_N$ is a $\MT(b)$-module. So $W_N$ is stable under the action of $\MT(b)$. Now $\tilde{z}_V\in(N_W(N)/U)(\R)$ since $\tilde{z}\colon \S_\C\rightarrow P_\C$ factors through $N_P(N)_\C$.

Conversely let $\tilde{Z}=W_0(\R)U_0(\C)\tilde{z}$ with $W_0$, $\tilde{z}$ as stated. Fix a Levi decomposition $P=W\rtimes G$. Let $G^\prime:=\MT(b)$, let $W^\prime:=N_W(W_0)$ and define $Q:=W^\prime\rtimes G^\prime$.
By definition of $Q$, $W_0\lhd Q$ and there exists a connected mixed Shimura datum $(Q,\cY^+)\hookrightarrow(P,\cX^+)$ with $b\in \unif(\cY^+)$. Now consider the morphisms of connected mixed Shimura data
\[
(Q,\cY^+)/W_0\xleftarrow{\varphi}(Q,\cY^+)\xrightarrow{i}(P,\cX^+).
\]
In the fibres above the point $b\in S_G$ these maps are simply
\[
S_{Q,b}/Z\twoheadleftarrow S_{Q,b}\hookrightarrow E=S_b.
\]
Hence $Z$ is a weakly special subvariety by definition.
\end{proof}
\end{prop}

\begin{cor}\label{complex subgroup for (semi-)abelian}
\begin{enumerate}
\item Weakly special subvarieties of a complex abelian variety are precisely the translates of abelian subvarieties;
\item Weakly special subvarieties of an algebraic torus over $\C$ are precisely the translates of subtori.
\end{enumerate}
\begin{proof}
This is a direct consequence of Proposition \ref{weakly special for semi-abelian}.
\end{proof}
\end{cor}

\subsection{Smallest weakly special subvariety containing a given subvariety of an abelian variety or an algebraic torus over $\C$}

\begin{prop}\label{smallest weakly special containing subvariety semiabelian}
\begin{enumerate}
\item Let $X$ be a complex abelian variety and let $Z$ be a closed irreducible subvariety of $X$. Denote by
\[
\tilde{X}=\pi_1(X,z)\otimes_\Z\R=H_1(X,\R)\cong\C^n\xrightarrow{u}X
\]
the universal cover of $X$ ($z\in Z^{\sm}$), then the smallest weakly special subvariety of $X$ containing $Z$ is a translate of $u(\pi_1(Z^\sm,z)\otimes\R)$.
\item Let $X$ be an algebraic torus over $\C$ and let $Z$ be a closed irreducible subvariety of $X$. Denote by
\[
\tilde{X}=\pi_1(X,z)\otimes_\Z\C=H_1(X,\C)\cong\C^n\xrightarrow{u}X
\]
the universal cover of $X$ ($z\in Z^{\sm}$), then the smallest weakly special subvariety of $X$ containing $Z$ is a translate of $u(\pi_1(Z^\sm,z)\otimes\C)$.
\end{enumerate}
\begin{proof}
\begin{enumerate}
\item If $X$ is a complex abelian variety, then the result is due to Ullmo-Yafaev. Their proof of \cite[Proposition 5.1]{UllmoA-characterisat} has in fact revealed this property. Here we restate the proof with more details.

Let $Z^\de\xrightarrow{s}Z$ be a desingularization of $Z^\de$ s.t. there exists a Zariski open subset $Z^\de_0$ of $Z^\de$ s.t. $Z^\de_0\xrightarrow[s]{\sim}Z^\sm$. By the commutative diagram
\[
\begin{diagram}
\pi_1(Z^\de_0,z) &\rTo^{\sim} &\pi_1(Z^\sm,z) \\
\dOnto & &\dTo &\rdTo\\
\pi_1(Z^\de,z) &\rTo &\pi_1(Z,z) &\rTo &\pi_1(X,z) \\
\end{diagram},
\]
where $z\in Z^\sm$ (the surjectivity on the left is due to \cite[2.10.1]{KollarShafarevich-map}), we know that the image of $\pi_1(Z^\de,z)$ and the image of $\pi_1(Z^\sm,z)$ in $\pi_1(X,z)$ are the same.

Let $\Alb(Z^\de)$ be the Albanese variety of $Z^\de$ normalized by $z$, then the map $\tau\colon Z^\de\rightarrow Z\rightarrow X$ factors uniquely through the Albanese morphism(\cite[Theorem 12.15]{VoisinHodge-theory-an}):
\[
\begin{diagram}
Z^\de &\rOnto &Z &\rInto &X \\
&\rdTo_{\alb} & &\ruDashto_\Gamma & \\
& & \Alb(Z^\de)
\end{diagram}
\]
Let $A:=\Gamma(\Alb(Z^\de))$, then it is the smallest weakly special subvariety (i.e. the translate of an abelian subvariety) of $X$ containing $Z$ since $\alb(Z^\de)$ generates $\Alb(Z^\de)$ (\cite[Lemma 12.11]{VoisinHodge-theory-an}).

It suffices to prove that the image of $\pi_1(Z^\de,z)$ in $\pi_1(X,z)\cong H_1(X,\Z)$ is of finite index in $H_1(A,\Z)$. This is true since the image of $\pi_1(Z^\de,z)$ in $H_1(X,\Z)$ contains
\[
(\Gamma\circ \alb)_*H_1(Z^\de,\Z)\cong\Gamma_*H_1(\Alb(Z^\de),\Z)\cong\Gamma_*\pi_1(\Alb(Z^\de))
\]
(the first isomorphism is given by the definition of Albanese varieties via Hodge theory, see e.g. the proof of \cite[Lemma 12.11]{VoisinHodge-theory-an}), which is of finite index in $\pi_1(A,z)\cong H_1(A,\Z)$ by \cite[2.10.2]{KollarShafarevich-map}.

\item If $X$ is an algebraic torus over $\C$, then we can first of all translate $Z$ by a point s.t. the translate contains the origin of $X$. Now we are done if we can prove that the smallest subtorus containing this translate of $Z$ is $u(\pi_1(Z^\sm,z)\otimes_\Z\C)$.

Suppose $T\cong(\C^*)^m$ is the smallest sub-torus of $X$ containing $Z$ with $j\colon Z^\sm\hookrightarrow T$ the inclusion. We are done if we can prove $[\pi_1(T,z):j_*\pi_1(Z^\sm,z)]<\infty$. If not, then
\begin{equation}\label{nul image of pi1(Y)}
\begin{diagram}
j_*\pi_1(Z^\sm,z) &~\subset~ &\ker(\Z^m\rOnto^{\rho}\Z)
\end{diagram}
\end{equation}
for some map $\rho$.
Since the covariant functor $T\mapsto X_*(T)$ ($X_*(T)$ is the co-character group of $T$) is an equivalence between the category \{algebraic tori over $\C$ and their morphisms as algebraic groups\} and the category \{free $\Z$-modules of finite rank\}, the map $\rho$ corresponds to a surjective map (with connected kernel) of tori $p\colon T\twoheadrightarrow T^\prime$. The composition of the maps $Z^\sm\xrightarrow{j}T\xrightarrow{p}T^\prime=\G_{m,\C}$ should be dominant by the choice of $T$. But then we have
\[
[\pi_1(T^\prime,p(z)):(p\circ j)_*\pi_1(Z^\sm,z)]<\infty
\]
(\cite[2.10.2]{KollarShafarevich-map}), which contradicts \eqref{nul image of pi1(Y)} by the following lemma.

\begin{lemma} For any $\C$-split torus $T\cong(\C^*)^n$, we have a canonical isomorphism
\[
X_*(T)\xrightarrow[\sim]{\psi_T}\pi_1(T,1).
\]
Here ``canonical" means that for any morphism (between algebraic groups) $f\colon T\rightarrow T^\prime$ between two such $\C$-split tori, the following diagram commutes:
\[
\begin{diagram}
X_*(T) &\rTo^{\psi_T}_{\sim} &\pi_1(T,1) \\
\dTo^{X_*(f)} & &\dTo^{f_*} \\
X_*(T^\prime) &\rTo^{\psi_{T^\prime}}_{\sim} &\pi_1(T^\prime,1)
\end{diagram}
\]
\begin{proof} Denote by $U_1:=\{z\in\C~|~|z|=1\}$ and $i\colon U_1\hookrightarrow\C^*$ the inclusion. Then the map $\psi_T$ is defined by
\[
\begin{diagram}
X_*(T) &\rTo^{\psi_T} &\pi_1(T,1) \\
\nu &\mapsto &[\nu\circ i]
\end{diagram}.
\]
This is clearly a group homomorphism. It is surjective since a representative of the generators of $\pi_1(T,1)$ is given by the $n$ coordinate embeddings $U_1\hookrightarrow\C^*\hookrightarrow T=(\C^*)^n$. $\psi_T$ is injective since $X_*(T)\cong \pi_1(T,1)\cong\Z^n$ is torsion-free. The rest of the lemma is immediate by the construction of $\psi_T$.
\end{proof}
\end{lemma}
\end{enumerate}
\end{proof}
\end{prop}

\subsection{Ax-Lindemann-Weierstra{\ss} for the unipotent part}\label{Ax-Lindemann for the unipotent part}
By abuse of notation we denote the uniformization of $E$ by $\unif\colon W(\R)U(\C)=W(\C)/F^0_bW(\C)\rightarrow E$. It is then the restriction of $\unif\colon\cX^+\rightarrow S$.

\begin{thm}\label{Ax-Lindemann for complex semi-abelian varieties}
Let $Y$ be a closed irreducible subvariety of $E$ and let $\tilde{Z}$ be a maximal irreducible algebraic subvariety which is contained in $\unif^{-1}(Y)$. Then $\tilde{Z}$ is weakly special.
\begin{proof} If $E$ is an algebraic torus over $\C$, this is a consequence of the Ax-Schanuel theorem \cite[Corollary 3.6]{MartinThesis}. If $E$ is an abelian variety, this is Pila-Zannier \cite[pp9, Remark 1]{PilaRational-points}. A proof using volume calculation and points counting method for these two cases can be found in the Appendix. The general case will be proved in $\mathsection$\ref{Ax-Lindemann Part 3: The unipotent part}.
\end{proof}
\end{thm}

\section{The smallest weakly special subvariety containing a given subvariety}\label{The smallest weakly special subvariety containing a given subvariety}
In this section, our goal is to prove a theorem (Theorem \ref{smallest weakly special containing a subvariety}) which (in some sense) generalizes \cite[3.6, 3.7]{MoonenLinearity-prope}. In particular, we get a criterion of weak specialness as a corollary (Corollary \ref{weakly special iff algebraic on the top and on the bottom}) which generalizes \cite[Theorem 4.1]{UllmoA-characterisat}.
Before the proof, let us do some technical preparation at first. 

Let $S$ be a connected mixed Shimura variety associated with the connected mixed Shimura datum $(P,\cX^+)$ and let $\unif\colon\cX^+\rightarrow S=\Gamma\backslash\cX^+$ be the uniformization. We may assume $P=\MT(\cX^+)$ by Proposition \ref{properties of irreducible mixed Shimura datum}. There exists a $\Gamma^\prime\leqslant\Gamma$ of finite index s.t. $\Gamma^\prime$ is neat. Let $S^\prime:=\Gamma^\prime\backslash\cX^+$ and let  $\unif^\prime\colon \cX^+\rightarrow S^\prime$ be its uniformization. Choose any faithful $\Q$-representation $\xi\colon P\rightarrow\GL(M)$ of $P$, then Theorem \ref{all variations of mixed Shimura varieties are admissible} claims that $\xi$ (together with a choice of a $\Gamma^\prime$-invariant lattice of $M$) gives rise to an admissible variation of mixed Hodge structure on $S^\prime$. The generic Mumford-Tate group of this variation is $P$.

Suppose that $Y$ is a closed irreducible subvariety of $S$. Let $Y^\prime$ be an irreducible component of $p^{-1}(Y)$ under $p\colon S^\prime=\Gamma^\prime\backslash\cX^+\rightarrow S=\Gamma\backslash\cX^+$, then $Y^\prime$ is a closed irreducible subvariety of $S^\prime$ which maps surjectively to $Y$ under $p$. The variation we constructed above can be restricted to $Y^{\prime\sm}$, and this restriction is still admissible by Lemma \ref{admissible stable under restriction}. The \textit{connected algebraic monodromy group associated with $Y^\sm$} is defined to be the connected algebraic monodromy group of the restriction of the VMHS defined in the last paragraph to $Y^{\prime\sm}$, i.e. the neutral component of the Zariski closure of the image of $\pi_1(Y^{\prime\sm},y^\prime)\rightarrow\pi_1(S^\prime,y^\prime)\rightarrow P$.

Let us briefly prove that the connected algebraic monodromy group associated with $Y^\sm$ is well-defined. Suppose that we have another covering $S^{\prime\prime}\xrightarrow{p^\prime}S^\prime$ with $S^{\prime\prime}$ smooth. Let $Y^{\prime\prime}$ be an irreducible component of $p^{\prime-1}(Y^\prime)$. Let $Y^{\prime\prime\sm}_0:=Y^{\prime\prime\sm}\cap p^{\prime-1}(Y^{\prime\sm})$, then by the commutative diagram
\[
\begin{diagram}
\pi_1(Y^{\prime\prime\sm}_0,y^{\prime\prime})=\pi_1(Y^{\prime\prime\sm},y^{\prime\prime}) &\rTo &\pi_1(S^{\prime\prime},y^{\prime\prime}) &\rTo &P \\
\dTo & &\dTo & &\dTo^{=} \\
\pi_1(Y^{\prime\sm},y^\prime) &\rTo &\pi_1(S^\prime,y^\prime) &\rTo &P
\end{diagram},
\]
where the equality in the top-left cornor is given by \cite[2.10.1]{KollarShafarevich-map} and the morphism on the left is of finite index by \cite[2.10.2]{KollarShafarevich-map}, the neutral components of the Zariski closures of the images of $\pi_1(Y^{\prime\prime\sm},y^{\prime\prime})$ and $\pi_1(Y^{\prime\sm},y^\prime)$ in $P$ coincide.

\begin{thm}\label{smallest weakly special containing a subvariety}
Let $S$ be a connected mixed Shimura variety associated with the connected mixed Shimura datum $(P,\cX^+)$ and let $\unif\colon\cX^+\rightarrow S=\Gamma\backslash\cX^+$ be the uniformization. Let $Y$ be a closed irreducible subvariety of $S$ and
\begin{itemize}
\item let $\tilde{Y}$ be an irreducible component of $\unif^{-1}(Y)$;
\item take $\tilde{y}_0\in\tilde{Y}$;
\item let $N$ be the connected algebraic monodromy group associated with $Y^\sm$.
\end{itemize}
Then
\begin{enumerate}
\item The set $\tilde{F}:=N(\R)^+U_N(\C)\tilde{y}_0$, where $U_N:=U\cap N$, is a weakly special subset of $\cX^+$ (or equivalently, $F:=\unif(\tilde{F})$ is a weakly special subvariety of $S$). Moreover $N$ is the largest subgroup of $Q$ s.t. $N(\R)^+U_N(\C)$ stabilizes $\tilde{F}$, where $(Q,\cY^+)$ is the smallest connected sub-mixed Shimura datum with $\tilde{F}\subset\cY^+$;
\item The Zariski closure of $\tilde{Y}$ in $\cX^+$ (which means the complex analytic irreducible component of the intersection of the Zariski closure of $\tilde{Y}$ in $\cX^\vee$ and $\cX^+$ which contains $\tilde{Y}$) is $\tilde{F}$;
\item The smallest weakly special subset containing $\tilde{Y}$ is $\tilde{F}$ and $F$ is the smallest weakly special subvariety of $S$ containing $Y$.
\end{enumerate}
\begin{proof}
\begin{enumerate}
\item 
Let $S_Y$ be the smallest special subvariety containing $Y$. Such an $S_Y$ exists since the irreducible components of intersections of special subvarieties are special (which can easily be shown by means of generic Mumford-Tate group). By definition of special subvarieties, there exists a connected mixed Shimura subdatum $(Q,\cY^+)$ s.t. $S_Y$ is the image of $\Gamma_Q\backslash\cY^+$ in $S$ where $\Gamma_Q:=\Gamma\cap Q(\Q)$. We may furthermore assume $(Q,\cY^+)$ to have generic Mumford-Tate group by Proposition \ref{properties of irreducible mixed Shimura datum}.

Let $N$ be the connected algebraic monodromy group associated with $Y^\sm$, then $N\lhd Q$ (and also $N\lhd Q^\der$) by the discussion at the beginning of this section (which claims that the variation we use to define $N$ is admissible), Remark \ref{stabilizer of Hodge generic point} (which claims that the generic Mumford-Tate group of this variation is $Q$) and Theorem \ref{Y Andre monodromy normal}.

Then $\tilde{F}$ is a weakly special subset of $\cY^+$ since it is the inverse image of the point $\varphi(\tilde{y}_0)$ under the Shimura morphism $(Q,\cY^+)\xrightarrow{\varphi}(Q,\cY^+)/N$. Then $\tilde{F}$ is also a weakly special subset of $\cX^+$ by definition. By the choice of $(Q,\cY^+)$, $\tilde{F}$ is Hodge generic in $\cY^+$, and hence $\varphi(\tilde{F})$ is a Hodge generic point in $\cY^{\prime+}$. Now $\stab_{Q^\der(\Q)}(\tilde{F})^\circ=N(\Q)$ by Remark \ref{stabilizer of Hodge generic point}.

\item We prove that $\tilde{F}$ is the Zariski closure of $\tilde{Y}$ in $\cX^+$. We first show that it suffices to prove $\tilde{Y}\subset\tilde{F}$. 
Let $\bar{\tilde{Y}}$ be the Zariski closure of $\tilde{Y}$ in $\cX^+$, then $\bar{\tilde{Y}}\subset\tilde{F}$ since $\tilde{Y}\subset\tilde{F}$ and $\tilde{F}$ is algebraic (Lemma \ref{weakly special subsets are irreducible algebraic}).
On the other hand, $\Gamma_{Y^\sm}:=\im(\pi_1(Y^\sm)\rightarrow\pi_1(S)\rightarrow P)$ stabilizes $\tilde{Y}$, so $\Gamma_{Y^\sm}\tilde{y}_0\subset\tilde{Y}$. The group $\Gamma_{Y^\sm}$ is Zariski dense in $N$, and hence Zariski dense in $N_{\C}$. As $\tilde{F}=N(\C)\tilde{y}_0\cap\cX^+$, $\Gamma_{Y^\sm}\tilde{y}_0$ is Zariksi dense in $\tilde{F}$. Hence we have $\tilde{F}\subset\bar{\tilde{Y}}$. As a result, $\tilde{F}=\bar{\tilde{Y}}$.

Next we prove that $\tilde{Y}\subset\tilde{F}$ (or equivalently, $Y\subset F$).

The fact that $\tilde{Y}\subset\tilde{F}$ has nothing to do with the level structure. Hence we may assume $\Gamma=\Gamma_W\rtimes\Gamma_G$ with $\Gamma_W\subset W(\Z)$, $\Gamma_U:=\Gamma_W\cap U\subset U(\Z)$, $\Gamma_V:=\Gamma_W/\Gamma_U\subset V(\Z)$ and $\Gamma_G\subset G(\Z)$ small enough s.t. they are all neat and s.t. $\Gamma\subset P^\der(\Q)$ (Remark \ref{sufficiently small congruence subgroup cotained in derivative}(2)). We write $\Gamma_{P/U}:=\Gamma/\Gamma_U$.

We may replace $(P,\cX^+)$ by $(Q,\cY^+)$ and $S$ by $S_Y$ (same notation as in (1)) since $\tilde{Y}$, $\tilde{F}\subset\cY^+$ and $Y$, $F\subset S_Y$. In other words, we may assume that $Y$ is Hodge generic in $S$ and $(P,\cX^+)$ is irreducible.

Consider the following diagram:
\[
\begin{diagram}
\cX^+ &\rTo^{\pi_{P/U}} &\cX^+_{P/U} &\rTo^{\pi_G} &\cX^+_G \\
\dTo^{pr} & &\dTo^{\unif_{P/U}} & &\dTo^{\unif_G} \\
S=\Gamma\backslash\cX^+ &\rTo^{[\pi_{P/U}]} &S_{P/U}:=\Gamma_{P/U}\backslash\cX^+_{P/U} &\rTo^{[\pi_G]} &S_G:=\Gamma_G\backslash\cX^+_G
\end{diagram}
\]

Denote by $\pi$ and $[\pi]$ the composites of the maps in the two lines respectively. Denote by $\tilde{Y}_G:=\pi(\tilde{Y})$, $Y_G:=[\pi](Y)$ and $\tilde{Y}_{P/U}:=\pi_{P/U}(\tilde{Y})$, $Y_{P/U}:=[\pi_{P/U}](Y)$; $\tilde{F}_G:=\pi(\tilde{F})$, $F_G:=[\pi](F)$ and $\tilde{F}_{P/U}:=\pi_{P/U}(\tilde{F})$, $F_{P/U}:=[\pi_{P/U}](F)$. Denote by $\tilde{y}_{0,P/U}:=\pi_{P/U}(\tilde{y}_0)$ and $\tilde{y}_{0,G}:=\pi(\tilde{y}_0)$.

Now to make the proof more clear, we divide it into several steps.

\textit{\underline{Step I.}} Prove that $\tilde{Y}_G\subset\tilde{F}_G$.

We begin the proof with the following lemma:

\begin{lemma}\label{monodromy group for projection}
In the context above, the connected algebraic monodromy group associated with $\bar{Y_G}^\sm$ (resp. $\bar{Y_{P/U}}^\sm$) is $G_N$ (resp. $N/U_N$ where $U_N:=U\cap N$).
\begin{proof} We only prove the statement for $\bar{Y_G}^\sm$. The proof for $\bar{Y_{P/U}}^\sm$ is similar. Take $Y^\sm_0:=Y^\sm\cap\pi^{-1}(Y_G^\sm)$, then we have the commutative diagram below:
\[
\begin{diagram}
\pi_1(Y^\sm_0,y) &\rTo &\pi_1(Y_G^\sm,y_G) &\rOnto &\pi_1(\bar{Y_G}^\sm,\zeta_G) \\
\dOnto & & &\rdTo &\dTo \\
\pi_1(Y^\sm,y) &\rTo &\pi_1(S,y) &\rTo &\pi_1(S_G,y_G) \\
& &\dTo & &\dTo \\
& & P &\rTo &G
\end{diagram}.
\]
Here, the morphism on the left and the right morphism on the top are surjective since $\codim_{Y^\sm}(Y^\sm-Y^\sm_0)\geqslant1$ and $\codim_{\bar{Y_G}^\sm}(\bar{Y_G}^\sm-Y_G^\sm)\geqslant 1$ (\cite[2.10.1]{KollarShafarevich-map}).
Now \cite[2.10.2]{KollarShafarevich-map} shows that the image of $\pi_1(Y^\sm_0,y)$ is of finite index in $\pi_1(Y_G^\sm,y_G)$, so the neutral components of the Zariski closures of $\pi_1(Y^\sm,y)$ and $\pi_1(\bar{Y_G}^\sm,y_G)$ in $G$ coincide. Hence we are done.
\end{proof}
\end{lemma}

Let $\tilde{Z}$ be the closure (w.r.t. archimedean topology) of $\tilde{Y}_G$ in $\cX^+_G$, then $\tilde{Z}$ is a complex analytic irreducible component of $\unif_G^{-1}(\bar{Y_G})$. For the pure connected Shimura datum $(G^\ad,\cX_G^+)$, we have a decomposition (\cite[3.6]{MoonenLinearity-prope})
\begin{equation*}\label{decomposition of the base}
(G^\ad,\cX^+_G)=(G_N^\ad,\cX^+_{G,1})\times (G_2,\cX^+_{G,2}).
\end{equation*}

By \cite[3.6, 3.7]{MoonenLinearity-prope} and Lemma \ref{monodromy group for projection},
$\tilde{Z}\subset\cX^+_{G,1}\times\{\tilde{y_{G,2}}\}$, i.e.
$\tilde{Z}\subset G_N(\R)^+\tilde{x}_G$ for some $\tilde{x}_G\in\cX_G^+$. But $\tilde{y}_{0,G}\in\tilde{Y}_G\subset\tilde{Z}$, so $\tilde{F}_G=G_N(\R)^+\tilde{y}_{0,G}\subset G_N(\R)^+\tilde{x}_G$. This implies that $\tilde{F}_G=G_N(\R)^+\tilde{x}_G$. As a result, $\tilde{Y}_G\subset\tilde{Z}\subset\tilde{F}_G$.

\textit{\underline{Step II.}} Consider the Shimura morphism
\[
\begin{diagram}
(P,\cX^+) &\rOnto^{\rho} &(P^\prime,\cX^{+\prime}):=(P,\cX^+)/N.
\end{diagram}
\]
Then $\tilde{F}=\rho^{-1}(\rho(\tilde{F}))$ by definition of $\rho$. So in order to prove $\tilde{Y}\subset\tilde{F}$, it is enough to show that $\rho(\tilde{Y})\subset\rho(\tilde{F})$. Hence we may replace $(P,\cX^+)$ by $(P^\prime,\cX^{+\prime})$. In other words, we may assume $N=\textbf{1}$.

In this case $\tilde{F}$ is just a point $\tilde{x}\in\cX^+$. Call $\tilde{x}_{P/U}:=\pi_{P/U}(\tilde{x})$, $\tilde{x}_G:=\pi(\tilde{x})$ and $x:=\unif(\tilde{x})$, $x_{P/U}:=\unif_{P/U}(\tilde{x}_{P/U})$, $x_G:=\unif_G(\tilde{x}_G)$. Then since $Y_G\subset F_G$, we have $Y\subset E$ where $E$ is the fibre of $S\xrightarrow{[\pi]}S_G$ over $x_G$. Denote by $A$ the fibre of $S_{P/U}\xrightarrow{[\pi]_G}S_G$ over $x_G$ and $T$ the fibre of $S\xrightarrow{[\pi_{P/U}]}S_{P/U}$ over $x_{P/U}$, then by \cite[3.13, 3.14]{PinkThesis} $A$ is an abelian variety and $T$ is an algebraic torus.

\textit{\underline{Step III.}} Prove that $\tilde{Y}_{P/U}\subset\tilde{F}_{P/U}$, i.e. $\tilde{Y}_{P/U}=\{\tilde{x}_{P/U}\}$.

By \textit{\underline{Step I}}, $Y_{P/U}\subset A$.
We have the following morphisms
\[
\pi_1(Y_{P/U}^\sm)\rightarrow\pi_1(A)\rightarrow\pi_1(S_{P/U})=\Gamma_{P/U}\rightarrow P/U=V\rtimes G.
\]
The neutral component of the Zariski closure of $\pi_1(Y_{P/U}^\sm)$ (resp. $\pi_1(A)$) in $P/U=V\rtimes G$ is $\textbf{1}$ (resp. $V$), so the image of
\[
\pi_1(Y_{P/U}^\sm)\rightarrow\pi_1(A)
\]
is a finite group.

Now $Y_{P/U}$ is irreducible since $Y$ is irreducible. So by Proposition \ref{smallest weakly special containing subvariety semiabelian}, $Y_{P/U}\subset A$ is a point. Equivalently, $\tilde{Y}_{P/U}$ is a point. So $\tilde{Y}_{P/U}\subset\tilde{F}_{P/U}$ since $\tilde{Y}_{P/U}\cap\tilde{F}_{P/U}\neq\emptyset$ (both of them contain $\tilde{y}_{0,P/U}$).

\textit{\underline{Step IV.}} Prove that $\tilde{Y}\subset\tilde{F}$, i.e. $\tilde{Y}=\{\tilde{x}\}$.

By \textit{\underline{Step I}}, $Y\subset E$. By \textit{\underline{Step III}}, $Y_{P/U}=\{x_{P/U}\}$. So $Y\subset T$. We have the following morphisms
\[
\pi_1(Y^\sm)\rightarrow\pi_1(T)\rightarrow\pi_1(S)=\Gamma\rightarrow P=W\rtimes G.
\]
The neutral component of the Zariski closure of $\pi_1(Y^\sm)$ (resp. $\pi_1(T)$) in $P=W\rtimes G$ is $\textbf{1}$ (resp. $U$), so the image of
\[
\pi_1(Y^\sm)\rightarrow\pi_1(T)
\]
is a finite group.

Now since $Y$ is irreducible, by Proposition \ref{smallest weakly special containing subvariety semiabelian}, $Y\subset T$ is a point. Equivalently, $\tilde{Y}$ is a point. So $\tilde{Y}\subset\tilde{F}$ since $\tilde{Y}\cap\tilde{F}\neq\emptyset$ (both of them contain $\tilde{y}_0$).

\item 
Since every weakly special subset of $\cX^+$ is algebraic by Lemma \ref{weakly special subsets are irreducible algebraic}, $\tilde{F}$ is also the smallest weakly special subset which contains $\tilde{Y}$. Therefore $F$ is the smallest weakly special subvariety of $S$ which contains $Y$.
\end{enumerate}
\end{proof}
\end{thm}

\begin{cor}\label{weakly special iff algebraic on the top and on the bottom}
Let $S$ be a connected mixed Shimura variety associated with the connected mixed Shimura datum $(P,\cX^+)$ and let $\unif\colon \cX^+\rightarrow S=\Gamma\backslash\cX^+$ be the uniformization map. Let $Y$ be a closed irreducible subvariety of $S$, then $Y$ is weakly special if and only if one (equivalently any) irreducible component of $\unif^{-1}(Y)$ is algebraic.

If $Y$ is weakly special, then $Y=\unif(N(\R)^+U_N(\C)\tilde{y})$ where $N$ is the connected algebraic monodromy group associated with $Y^{\sm}$, $U_N:=U\cap N$ and $\tilde{y}$ is any point of $\unif^{-1}(Y)$.
\begin{proof} The ``only if" part is immediate by Lemma \ref{weakly special subsets are irreducible algebraic}. Now we prove the ``if" part.

We first of all quickly show that if one irreducible component of $\unif^{-1}(Y)$ is algebraic, so are the others. The proof is the same as \cite[first paragraph of the proof of Theorem 4.1]{UllmoA-characterisat}. Suppose that $\tilde{Y}$ is an irreducible component of $\unif^{-1}(Y)$ which is algebraic, i.e. $\tilde{Y}$ is an irreducible component of $\cX^+\cap Z$ for some algebraic subvariety $Z$ of $\cX^\vee$. Then for any $\gamma\in\Gamma\subset P(\R)U(\C)$,
\[
\gamma\tilde{Y}=\gamma(\cX^+\cap Z)\subset\cX^+\cap\gamma Z=\gamma\gamma^{-1}(\cX^+\cap\gamma Z)\subset\gamma\tilde{Y}.
\]
Hence it follows that $\gamma\tilde{Y}=\cX^+\cap\gamma Z$ is algebraic.

Next under the notation of Theorem \ref{smallest weakly special containing a subvariety}, $\tilde{Y}=\bar{\tilde{Y}}=\tilde{F}$ since $\tilde{Y}$ is algebraic. Hence $\tilde{Y}$ is weakly special, and so is $Y$.

Finally if $Y$ is weakly special, then for any $\tilde{y}\in \unif^{-1}(Y)$ and $\tilde{Y}$ the irreducible component of $\unif^{-1}(Y)$ which contains $\tilde{y}$, $\tilde{Y}=\tilde{F}=N(\R)^+U_N(\C)\tilde{y}$ by Theorem \ref{smallest weakly special containing a subvariety}, and hence $Y=\unif(N(\R)^+U_N(\C)\tilde{y})$.
\end{proof}
\end{cor}

\section{Ax-Lindemann-Weierstra{\ss} Part 1: Outline of the proof}\label{Ax-Lindemann Part 1: Outline of the proof}

In the following three sections, we are going to prove Theorem \ref{Ax-Lindemann for type star}.
The organization of the proof is as follows:
the outline of the proof is given in this section. After some preparation, the key proposition (Proposition \ref{projection of stabilizer is stabilizer of the base}) leading to the theorem will be stated and exploited (together with Theorem \ref{Ax-Lindemann for complex semi-abelian varieties}) to finish the proof in Theorem \ref{conclusion}. We prove this key proposition in the next section using Pila-Wilkie's counting theorem and Theorem \ref{Ax-Lindemann for complex semi-abelian varieties} will be proved in $\mathsection$\ref{Ax-Lindemann Part 3: The unipotent part}, where a simple proof of Ax-Lindemann-Weierstra{\ss} for complex abelian varieties can be found.

Now let us fix some notation which will be used through the whole proof:

\begin{notation}\label{notations for the proof of Ax-Lindemann}
Consider the following diagram:
\[
\begin{diagram}
\cX^+ &\rTo^{\pi} &\cX^+_G \\
\dTo^{pr} & &\dTo^{\unif_G} \\
S=\Gamma\backslash\cX^+ &\rTo^{[\pi]} &S_G:=\Gamma_G\backslash\cX^+_G
\end{diagram}
\]
Denote by $\tilde{Y}_G:=\pi(\tilde{Y})$, $Y_G:=[\pi](Y)$ and $\tilde{Z}_G:=\pi(\tilde{Z})$.
\end{notation}

Now we begin the proof of Theorem~\ref{Ax-Lindemann for type star}.
Let us first of all do some reduction:
\begin{itemize}
\item Since every point of $\cX^+$ is weakly special, we may assume $\dim(\tilde{Z})>0$.
\item Let $(Q,\cY^+)$ be the smallest mixed Shimura subdatum of $(P,\cX^+)$ s.t $\tilde{Z}\subset\cY^+$ and let $S_Q$ be the corresponding special subvariety of $S$. Then $Q=\MT(\cY^+)$ by Proposition~\ref{properties of irreducible mixed Shimura datum}(1). If we replace $(P,\cX^+)$ by $(Q,\cY^+)$, $S$ by $S_Q$, $\unif\colon\cX^+\rightarrow S$ by $\unif_Q\colon\cY^+\rightarrow S_Q$ and $Y$ by an irreducible component $Y_0$ of $Y\cap S_Q$, then $\tilde{Z}$ is again a maximal irreducible algebraic subset of $\unif_Q^{-1}(Y_0)$. By definition, $\tilde{Z}$ is weakly special in $\cX^+$ iff it is weakly special in $\cY^+$. So we may assume $P=\MT(\cX^+)$ and that $\tilde{Z}$ is Hodge generic.
\item Furthermore, let $Y_0$ by the minimal irreducible subvariety of $S$ such that $\tilde{Z}\subset \unif^{-1}(Y_0)$, then $\tilde{Z}$ is still maximal irreducible algebraic in $\unif^{-1}(Y_0)$. Hence we may assume that $Y=Y_0$. In fact it is not hard to see that after this reduction, $Y=\bar{\unif(\tilde{Z})}$ and $\tilde{Z}$ is weakly special iff $Y$ is weakly special.
\item By the previous reduction, there is a unique complex analytic irreducible component of $\unif^{-1}(Y)$ which contains $\tilde{Z}$. Denote it by $\tilde{Y}$. Denote by $\tilde{Y}_G:=\pi(\tilde{Y})$, $Y_G:=[\pi](Y)$ and $\tilde{Z}_G:=\pi(\tilde{Z})$. Remark that by Lemma~\ref{algebraicity functorial using Borel embedding}, $\bar{\tilde{Z}_G}$ is an algebraic subset of $\cX^+_G$.
\item Replacing $\Gamma$ by a subgroup of finite index does not matter for this problem, so we may assume that $\Gamma$ is neat and $\Gamma\subset P^\der(\Q)$ 
(Remark~\ref{sufficiently small congruence subgroup cotained in derivative}(2)).
\end{itemize}

Let $\tilde{F}$ be the smallest weakly special subset containing $\tilde{Y}$. By Theorem~\ref{smallest weakly special containing a subvariety}, $\tilde{F}=N(\R)^+U_N(\C)\tilde{z}$ some $\tilde{z}\in\tilde{Z}\subset\tilde{Y}$, where $N$ is the connected algebraic monodromy group associated with $Y^\sm$ and $U_N:=U\cap N$. The set $\tilde{F}$ is Hodge generic in $(P,\cX^+)$ since $\tilde{Z}$ is, so $N\lhd P$ and $N\lhd P^\der$ by Theorem~\ref{Y Andre monodromy normal}.

Define
\[
\Gamma_{\tilde{Z}}:=\{\gamma\in\Gamma|\gamma\cdot\tilde{Z}=\tilde{Z}\}\qquad\text{(resp. }\Gamma_{G,\bar{\tilde{Z}_G}}:=\{\gamma_G\in\Gamma_G|\gamma_G\cdot\bar{\tilde{Z}_G}=\bar{\tilde{Z}_G}\})
\]
and
\[
H_{\tilde{Z}}:=(\bar{\Gamma_{\tilde{Z}}}^{\Zar})^\circ\qquad\text{(resp. }H_{\bar{\tilde{Z}_G}}:=(\bar{\Gamma_{G,\bar{\tilde{Z}_G}}}^{\Zar})^\circ).
\]
Define $U_{H_{\tilde{Z}}}:=U\cap H_{\tilde{Z}}$ and $W_{H_{\tilde{Z}}}:=W\cap H_{\tilde{Z}}$. Both of them are normal in $H_{\tilde{Z}}$. Then $H_{\tilde{Z}}$ (resp. $H_{\bar{\tilde{Z}_G}}$) is the largest connected subgroup of $P^{\der}$ (resp. $G^{\der}$) such that $H_{\tilde{Z}}(\R)^+U_{H_{\tilde{Z}}}(\C)$ (resp. $H_{\bar{\tilde{Z}_G}}(\R)^+$) stabilizes $\tilde{Z}$ (resp. $\bar{\tilde{Z}_G}$).

Define $V_{H_{\tilde{Z}}}:=W_{H_{\tilde{Z}}}/U_{H_{\tilde{Z}}}$ and $G_{H_{\tilde{Z}}}:=H_{\tilde{Z}}/W_{H_{\tilde{Z}}}\hookrightarrow P/W=G$.

The following two lemmas were proved for the pure case in \cite{PilaAxLindemannAg} and \cite{KlinglerThe-Hyperbolic-}.
\begin{lemma}\label{stabiliser of Z stabilises Y}
The set $\tilde{Y}$ is stable under $H_{\tilde{Z}}(\R)^+U_{H_{\tilde{Z}}}(\C)$.
\begin{proof} Every fiber of $\cX^+\rightarrow\cX^+_{P/U}$ can be canonically identified with $U(\C)$. So it is enough to prove that $\tilde{Y}$ is stable under $H_{\tilde{Z}}(\R)^+$: If $U_{H_{\tilde{Z}}}(\R)\tilde{y}\subset\tilde{Y}$ for $\tilde{y}\in\tilde{Y}$, then $U_{H_{\tilde{Z}}}(\C)\tilde{y}\subset\tilde{Y}$ because $\tilde{Y}$ is complex analytic and $U_{H_{\tilde{Z}}}(\C)\tilde{y}$ is the smallest complex analytic subset of $\cX^+$ containing $U_{H_{\tilde{Z}}}(\R)\tilde{y}$.

If not, then since $H_{\tilde{Z}}(\Q)$ is dense (w.r.t. the archimedean topology) in $H_{\tilde{Z}}(\R)^+$, there exists $h\in H_{\tilde{Z}}(\Q)$ such that $h\tilde{Y}\neq\tilde{Y}$. The set $\tilde{Z}$ is contained in $\tilde{Y}\cap h\tilde{Y}$ by definition of $H_{\tilde{Z}}$, and hence contained in a complex analytic irreducible component $\tilde{Y}^\prime$ of it.

Consider the Hecke operator $T_h$. Then $T_h(Y)=\unif(h\cdot \unif^{-1}(Y))$. %(remark that in \textit{loc.cit.}, $P(\Q)^+$ represents for $P(\Q)\cap P(\R)_+$ with $P(\R)_+$ the stabilizer of $\cX^+\subset\hom(\S_\C,P_\C)$ in $P(\R)$. See \textit{Fact 2.3(a)(c) loc.cit.}).
Hence
\[
Y\cap T_h(Y)=\unif(\unif^{-1}(Y)\cap(h\cdot \unif^{-1}(Y))).
\]
On the other hand, $T_h(Y)$ is equidimensional of the same dimension as $Y$ by definition, hence by reason of dimension, $h\tilde{Y}$ is an irreducible component of $\unif^{-1}(T_h(Y))=\Gamma h\Gamma\tilde{Y}$. So $\unif(h\tilde{Y})$ is an irreducible component of $T_h(Y)$.

Since $\tilde{Y}^\prime$ is a complex analytic irreducible component of $\tilde{Y}\cap h\tilde{Y}$, it is also a complex analytic irreducible component of $\unif^{-1}(Y)\cap(h\tilde{Y})=\Gamma\tilde{Y}\cap h\tilde{Y}$. So $Y^\prime:=\unif(\tilde{Y}^\prime)$ is a complex analytic irreducible component of $Y\cap \unif(h\tilde{Y})$. So $Y^\prime$ is a complex analytic irreducible component of $Y\cap T_h(Y)$, and hence is algebraic since $Y\cap T_h(Y)$ is.

Since $h\tilde{Y}\neq\tilde{Y}$ and $Y$ is irreducible, $\dim(Y^\prime)<\dim(Y)$. But $\tilde{Z}\subset\tilde{Y}\cap h\tilde{Y}\subset \unif^{-1}(Y^\prime)$. This contradicts the minimality of $Y$.
\end{proof}
\end{lemma}

\begin{lemma}\label{stabilizer normal in N}
$H_{\tilde{Z}}\lhd N$.
\begin{proof} We have $\tilde{Z}\subset\tilde{F}=N(\R)^+U_N(\C)\tilde{z}$ for some $\tilde{z}\in\tilde{Z}$, so the image of $\tilde{Z}$ under the morphism
\[
(P,\cX^+)\rightarrow(P,\cX^+)/N
\]
is a point. But $H_{\tilde{Z}}/(H_{\tilde{Z}}\cap N)$ stabilizes this point which is Hodge generic (since $\tilde{F}$ is Hodge generic in $\cX^+$), and therefore is trivial by Remark~\ref{stabilizer of Hodge generic point}. So $H_{\tilde{Z}}<N$.

Let $H^\prime$ be the algebraic group generated by $\gamma^{-1}H_{\tilde{Z}}\gamma$ for all $\gamma\in\Gamma_{Y^\sm}$, where $\Gamma_{Y^\sm}$ is the monodromy group of $Y^\sm$. Since $H^\prime$ is invariant under conjugation by $\Gamma_{Y^\sm}$, it is invariant under $\bar{\Gamma_{Y^\sm}}^{\Zar}$, therefore invariant under conjugation by $N$.

By Lemma~\ref{stabiliser of Z stabilises Y}, $\tilde{Y}$ is invariant under $H_{\tilde{Z}}(\R)^+U_{H_{\tilde{Z}}}(\C)$. On the other hand, 
$\tilde{Y}$ is also invariant under $\Gamma_{Y^\sm}$ by definition. So $\tilde{Y}$ is invariant under the action of $H^\prime(\R)^+U_{H^\prime}(\C)$ where $U_{H^\prime}:=U\cap H^\prime$. Since $H^\prime(\R)^+U_{H^\prime}(\C)\tilde{Z}$ is semi-algebraic, there exists an irreducible algebraic subset of $\cX^+$, say $\tilde{E}$, which contains $H^\prime(\R)^+U_{H^\prime}(\C)\tilde{Z}$ and is contained in $\tilde{Y}$ by \cite[Lemma~4.1]{PilaAbelianSurfaces}.
Now $\tilde{Z}\subset\tilde{E}\subset\tilde{Y}$, so 
$\tilde{Z}=\tilde{E}=H^\prime(\R)^+U_{H^\prime}(\C)\tilde{Z}$ by maximality of $\tilde{Z}$, and therefore 
$H^\prime=H_{\tilde{Z}}$ by definition of $H_{\tilde{Z}}$. So $H_{\tilde{Z}}$ is invariant under conjugation by $N$. Since $H_{\tilde{Z}}<N$, $H_{\tilde{Z}}$ is normal in $N$.
\end{proof}
\end{lemma}

\begin{cor}\label{normal subgroups}
\[
G_{H_{\tilde{Z}}},~H_{\bar{\tilde{Z}_G}}\lhd G^\der\text{ and }G_{H_{\tilde{Z}}}\lhd H_{\bar{\tilde{Z}_G}}.
\]
\begin{proof} We have $G_{H_{\tilde{Z}}}\lhd G_N\lhd G^\der$, and so $G_{H_{\tilde{Z}}}\lhd G^\der$ since all the three groups are reductive.

Working with $((G,\cX^+_G),~\bar{Y_G},~\bar{\tilde{Z}_G})$ instead of $((P,\cX^+),~Y,~\tilde{Z})$, we can prove (similar to Lemma~\ref{stabilizer normal in N}) that $H_{\bar{\tilde{Z}_G}}\lhd G_N$. Hence $H_{\bar{\tilde{Z}_G}}\lhd G^\der$ by the same reason for $G_{H_{\tilde{Z}}}$.

By definition $G_{H_{\tilde{Z}}}<H_{\bar{\tilde{Z}_G}}$. So $G_{H_{\tilde{Z}}}\lhd H_{\bar{\tilde{Z}_G}}$ since $G_{H_{\tilde{Z}}}\lhd G^\der$.
\end{proof}
\end{cor}

So far the proof looks similar to the pure case. From now on it will be quite different. For the readers' convenience, we list here some differences between the proof of Ax-Lindemann-Weierstra{\ss} for mixed Shimura varieties and for the pure case:
\begin{itemize}
\item We shall prove that $\tilde{Z}$ is an $H_{\tilde{Z}}(\R)^+U_{H_{\tilde{Z}}}(\C)$-orbit. To prove this, it suffices to prove $\dim H_{\tilde{Z}}>0$ when $S$ is a pure Shimura variety. However this is far from enough for the mixed case, since this does not exclude the naive counterexample when $\dim\tilde{Z}_G>0$ but $H_{\tilde{Z}}$ is unipotent. To overcome it, we should at least prove $\dim G_{H_{\tilde{Z}}}>0$. In fact we shall directly prove $G_{H_{\tilde{Z}}}=H_{\bar{\tilde{Z}_G}}$ (Proposition~\ref{projection of stabilizer is stabilizer of the base}). This equality is not obvious because, as appears in the proof of Lemma~\ref{projection of maximal still maximal}, there is no reason a priori why $\bar{\tilde{Z}_G}$, which is obviously algebraic in $\unif^{-1}(Y_G)$, should be maximal for this property. If one could prove direcly this is the case, then Klingler-Ullmo-Yafaev \cite[Theorem~1.3]{KlinglerThe-Hyperbolic-} would give directly the result.
\item As mentioned in the Introduction, we shall make essential use of the ``family'' version of Pila-Wilkie's theorem (Remark~\ref{EssentialUseOfPilaWilkie});
\item If $P=G$ is reductive, then $H_{\tilde{Z}}\lhd N\lhd P$ implies directly $H_{\tilde{Z}}\lhd P$. This is obviously false when $P$ is not reductive.
\item For a general mixed Shimura variety $S$, the fiber of $S\xrightarrow{[\pi]}S_G$ is not necessarily an algebraic group (Lemma~\ref{group automatically commutative}), hence not a semi-abelian variety. We do not have Ax-Lindemann-Weierstra{\ss} for the fiber for this case. Thus we should execute a proof of Ax-Lindemann-Weierstra{\ss} for the fiber. As the readers will see in $\mathsection$\ref{Ax-Lindemann Part 3: The unipotent part}, the proof of this case calls for much more careful study of $\tilde{Z}$. First of all, when doing the estimate and using the family version of Pila-Wilkie for the fiber (\textit{Step~I}), we should introduce a seemingly strange subgroup which serves as $G_N$ in the section. The reason for this will be explained in Remark~\ref{WhyWeShouldDoBothSteps}. 
Secondly, to prove that $W_{H_{\tilde{Z}}}$ is normal in $W$ is not trivial, and the key to the solution (\textit{Step~IV}) is a well-known fact: any holomorphic morphism from a complex abelian variety to an algebraic torus over $\C$ is trivial.
\end{itemize}

Before proceeding, we prove the following lemma:
\begin{lemma}\label{projection of maximal still maximal}
\begin{enumerate}
\item 
$\bar{\tilde{Y}_G}$ is weakly special. Hence $\bar{\tilde{Y}_G}=G_N(\R)^+\tilde{z}_G$ for any point $\tilde{z}_G\in\tilde{Z}_G$;
\item
$\bar{\unif_G({\tilde{Z}_G}})=\bar{Y_G}$.
\end{enumerate}
\begin{proof}
\begin{enumerate}
\item 
Let $\tilde{Z}^\prime$ be an irreducible algebraic subset of $\cX^+_G$ which contains $\bar{\tilde{Z}_G}$ and is contained in $\unif^{-1}(\bar{Y_G})$, maximal for these properties. By \cite[Theorem~1.3]{KlinglerThe-Hyperbolic-}, $Z^\prime:=\unif_G(\tilde{Z}^\prime)$ is weakly special, and therefore Zariski closed by definition. Now $\tilde{Z}\subset\pi^{-1}(\tilde{Z}^\prime)\cap \unif^{-1}(Y)$. However,
\[
\unif(\pi^{-1}(\tilde{Z}^\prime)\cap \unif^{-1}(Y))=\unif(\pi^{-1}(\tilde{Z}^\prime))\cap Y=[\pi]^{-1}(Z^\prime)\cap Y.
\]
Then we must have $Y\subset[\pi]^{-1}(Z^\prime)$ since $Y$ is the minimal irreducible closed subvariety of $S$ such that $\tilde{Z}\subset \unif^{-1}(Y)$. Therefore $\bar{Y_G}\subset Z^\prime$. But $Z^\prime\subset\bar{Y_G}$ by definition of $Z^\prime$, so $Z^\prime=\bar{Y_G}$. This means that $\bar{Y_G}$ is weakly special.

\item
Let $Y^\prime:=\bar{\unif_G(\tilde{Z}_G)}$, then $\bar{\tilde{Z}_G}\subset \unif_G^{-1}(Y^\prime)$. So $\tilde{Z}\subset\pi^{-1}(\unif_G^{-1}(Y^\prime))=\unif^{-1}([\pi]^{-1}(Y^\prime))$, and so
\[
\tilde{Z}\subset \unif^{-1}([\pi]^{-1}(Y^\prime))\cap \unif^{-1}(Y)=\unif^{-1}([\pi]^{-1}(Y^\prime)\cap Y).
\]
Hence there exists an irreducible component $Y^{\prime\prime}$ of $[\pi]^{-1}(Y^\prime)\cap Y$ such that  $\tilde{Z}\subset \unif^{-1}(Y^{\prime\prime})$. But
\[
[\pi](Y^{\prime\prime})\subset[\pi]([\pi]^{-1}(Y^\prime)\cap Y)=Y^\prime\cap Y_G,
\]
so $\dim([\pi](Y^{\prime\prime}))\leqslant\dim(Y^\prime\cap Y_G)$. If $Y^\prime\neq\bar{Y_G}$, then $\dim(Y^\prime\cap Y_G)<\dim(Y_G)$ and therefore $\dim(Y^{\prime\prime})<\dim(Y)$, which contradicts the minimality of $Y$. So $Y^\prime=\bar{Y_G}$.
\end{enumerate}
\end{proof}
\end{lemma}

\begin{prop}[key proposition]\label{projection of stabilizer is stabilizer of the base}
The set $\bar{\tilde{Z}_G}$ is weakly special and $G_{H_{\tilde{Z}}}=H_{\bar{\tilde{Z}_G}}$. In other words,
\[
\bar{\tilde{Z}_G}=G_{H_{\tilde{Z}}}(\R)^+\tilde{z}_G
\]
for any point $\tilde{z}_G\in\tilde{Z}_G$.
\end{prop}

Now let us show how this proposition together with Theorem~\ref{Ax-Lindemann for complex semi-abelian varieties} implies Theorem~\ref{Ax-Lindemann for type star}. Before proceeding to the final argument, we shall prove the following group theoretical lemma:
\begin{lemma}\label{compatible Levi}
Fixing a Levi decomposition $H_{\tilde{Z}}=W_{H_{\tilde{Z}}}\rtimes G_{H_{\tilde{Z}}}$, there exists a compatible Levi decomposition $P=W\rtimes G$.
\begin{proof} Suppose that the fixed Levi decomposition of $H_{\tilde{Z}}$ is given by $s_1\colon G_{H_{\tilde{Z}}}\rightarrow H_{\tilde{Z}}$. Define $P_*:=\pi^{-1}(G_{H_{\tilde{Z}}})$, then $H_{\tilde{Z}}<P_*$. Now choose any Levi decomposition $P=W\rtimes G$ defined by $s_2\colon G\rightarrow P$. Then $G_{H_{\tilde{Z}}}$, being a subgroup of $G$, is realized as a subgroup of $P$ via $s_2$. Hence $s_2$ induces a Levi-decomposition $P_*=W\rtimes^{s_2} G_{H_{\tilde{Z}}}$. We have thus a diagram
\[
\begin{diagram}
1 &\rTo &W_{H_{\tilde{Z}}} &\rTo &H_{\tilde{Z}} &\rTo\upperarrowtwo{s_1} &G_{H_{\tilde{Z}}} &\rTo &1 \\
& &\dInto & &\dInto & &\dTo^{=} & \\
1 &\rTo &W &\rTo &P_* &\rTo\upperarrowtwo{s_1} &G_{H_{\tilde{Z}}} &\rTo &1
\end{diagram},
\]
where the morphism $s_1$ in the second line is induced by the one in the first line. Now $s_1$, $s_2$ define two Levi decompositions of $P_*$. They differ by the conjugation by an element $w_0$ of $W(\Q)$ by \cite[Theorem~2.3]{AlgebraicGroupBible}. So replacing $s_2$ by its conjugation by $w_0$ we can find a Levi decomposition of $P$ which is compatible with the fixed $H_{\tilde{Z}}=W_{H_{\tilde{Z}}}\rtimes G_{H_{\tilde{Z}}}$.
\end{proof}
\end{lemma}

\begin{thm}\label{conclusion}
\begin{enumerate}
\item $\tilde{Z}=H_{\tilde{Z}}(\R)^+U_{H_{\tilde{Z}}}(\C)\tilde{z}$ for any $\tilde{z}\in\tilde{Z}$;
\item $H_{\tilde{Z}}\lhd P$.
\end{enumerate}
Hence $\tilde{Z}$ is weakly special by definition.
\begin{proof}
\begin{enumerate}
\item Consider a fibre of $\tilde{Z}$ over a Hodge-generic point $\tilde{z}_G\in \tilde{Z}_G$ such that $\pi|_{\tilde{Z}}$ is flat at $\tilde{z}_G$ (such a point exists by \cite[$\mathsection 4$, Lemma~1.4]{AndreMumford-Tate-gr} and generic flatness). Suppose that $\tilde{W}$ is an irreducible algebraic component of $\tilde{Z}_{\tilde{z}_G}$ such that $\dim(\tilde{Z}_{\tilde{z}_G})=\dim(\tilde{W})$, then since $\pi|_{\tilde{Z}}$ is flat at $\tilde{z}_G$, 
\[
\dim(\tilde{Z})=\dim(\tilde{Z}_G)+\dim(\tilde{Z}_{\tilde{z}_G})=\dim(\tilde{Z}_G)+\dim(\tilde{W}).
\]

Consider the set $\tilde{E}:=H_{\tilde{Z}}(\R)^+U_{H_{\tilde{Z}}}(\C)\tilde{W}$. It is semi-algebraic (since $\tilde{W}$ is algebraic and the action of $P(\R)^+U(\C)$ on $\cX^+$ is algebraic). The fact $\tilde{W}\subset\tilde{Z}$ implies that $\tilde{E}\subset\tilde{Z}$. By \cite[Lemma~4.1]{PilaAbelianSurfaces}, there exists an irreducible algebraic subset of $\cX^+$, say $\tilde{E}_{\mathrm{alg}}$, which contains $\tilde{E}$ and is contained in $\tilde{Z}$.
Now we have by Proposition~\ref{projection of stabilizer is stabilizer of the base}
\[
\pi(\tilde{E})=G_{H_{\tilde{Z}}}(\R)^+\tilde{z}_G=H_{\bar{\tilde{Z}_G}}(\R)^+\tilde{z}_G=\tilde{Z}_G
\]
and that the $\R$-dimension of every fiber of $\pi|_{\tilde{E}}$ is at least $\dim_\R(\tilde{W})$. So
\[
\dim(\tilde{E}_{\mathrm{alg}})\geqslant\dim(\pi(\tilde{E}))+\dim(\tilde{W})=\dim(\tilde{Z}_G)+\dim(\tilde{W})=\dim(\tilde{Z}).
\]
So $\tilde{E}=\tilde{Z}$ since $\tilde{Z}$ is irreducible.

Next let $\tilde{W}^\prime$ be an irreducible algebraic subset which contains $\tilde{Z}_{\tilde{z}_G}$ and is contained in $\unif^{-1}(Y)_{\tilde{z}_G}$, maximal for these properties. Then $\tilde{W}^\prime$ is weakly special by Theorem~\ref{Ax-Lindemann for complex semi-abelian varieties}. We have $\tilde{W}^\prime\subset\tilde{Y}$ since $\tilde{Y}$ is an irreducible component of $\pi^{-1}(Y)$. Consider $\tilde{E}^\prime:=H_{\tilde{Z}}(\R)^+U_{H_{\tilde{Z}}}(\C)\tilde{W}^\prime$. Then $\tilde{E}^\prime\subset\tilde{Y}$ 
by Lemma~\ref{stabiliser of Z stabilises Y}. But $\tilde{E}^\prime$ is semi-algebraic, so by \cite[Lemma~4.1]{PilaAbelianSurfaces}, there exists an irreducible algebraic subset of $\cX^+$, say $\tilde{E}^\prime_{\mathrm{alg}}$ which contains $\tilde{E}^\prime$ and is contained in $\tilde{Y}$. So $\tilde{Z}=\tilde{E}\subset\tilde{E}^\prime_{\mathrm{alg}}\subset\tilde{Y}$, and hence $\tilde{Z}=\tilde{E}^\prime_{\mathrm{alg}}=\tilde{E}^\prime$ by the maximality of $\tilde{Z}$. So $\tilde{Z}_{\tilde{z}_G}=\tilde{W}^\prime$ is weakly special.

Write $\tilde{Z}_{\tilde{z}_G}=W^\prime(\R)U^\prime(\C)\tilde{z}$ with $W^\prime<W$, $U^\prime=W^\prime\cap U$ and $\tilde{z}\in\tilde{Z}_{\tilde{z}_G}$. Then $W_{H_{\tilde{Z}}}<W^\prime$. The complex structure of $\pi^{-1}(\tilde{z}_G)$ comes from $W(\R)U(\C)\cong W(\C)/F^0_{\tilde{z}_G}W(\C)$, where $F^0_{\tilde{z}_G}W(\C)=\exp(F^0_{\tilde{z}_G}\lie W_{\C})$. So the fact that $\tilde{Z}_{\tilde{z}_G}$ is a complex subspace of $\pi^{-1}(\tilde{z}_G)$ implies that $W^\prime/U^\prime$ is a $\MT(\tilde{z}_G)=G$-module. Hence $W^\prime$ is a $G$-group. 

Define $P^\prime:=W^\prime H_{\tilde{Z}}$, then $P^\prime$ is a subgroup of $P$ since $W^\prime>W_{H_{\tilde{Z}}}$ and $G_{H_{\tilde{Z}}}W^\prime=W^\prime$. Now we have
\[
\tilde{Z}=H_{\tilde{Z}}(\R)^+U_{H_{\tilde{Z}}}(\C)\tilde{Z}_{\tilde{z}_G}=H_{\tilde{Z}}(\R)^+U_{H_{\tilde{Z}}}(\C)W^\prime(\R)U^\prime(\C)\tilde{z}=P^\prime(\R)^+U^\prime(\C)\tilde{z}.
\]
So $H_{\tilde{Z}}=P^\prime$ because $H_{\tilde{Z}}$ is the largest subgroup of $P^{\der}$ such that $H_{\tilde{Z}}(\R)^+U_{H_{\tilde{Z}}}(\C)$ stabilizes $\tilde{Z}$. So we have $\tilde{Z}=H_{\tilde{Z}}(\R)^+U_{H_{\tilde{Z}}}(\C)\tilde{z}$.

\item First of all, $U_{H_{\tilde{Z}}}\lhd P$ by Proposition~\ref{properties of irreducible mixed Shimura datum}(2).

Next consider the complex structure of $\pi^{-1}(\tilde{z}_G)$ which comes from $W(\R)U(\C)$ $\cong W(\C)/F^0_{\tilde{z}_G}W(\C)$. So the fact that $\tilde{Z}_{\tilde{z}_G}$ is a complex subspace of $\pi^{-1}(\tilde{z}_G)$ implies that $V_{H_{\tilde{Z}}}$ is a $\MT(\tilde{z}_G)=G$-module. Hence $W_{H_{\tilde{Z}}}$ is a $G$-group. Besides, $G_{H_{\tilde{Z}}}\lhd G$ by Proposition~\ref{projection of stabilizer is stabilizer of the base}. In particular, $G_{H_{\tilde{Z}}}$ is reductive.

Then let us prove $W_{H_{\tilde{Z}}}\lhd P$. It suffices to prove $W_{H_{\tilde{Z}}}\lhd W$. For any $\tilde{z}\in\tilde{Z}$, we have proved in (1) that $\tilde{Z}_{\tilde{z}_G}=W_{H_{\tilde{Z}}}(\R)U_{H_{\tilde{Z}}}(\C)\tilde{z}$ is weakly special. Hence by Proposition~\ref{particular choices of i and varphi}, there is a connected mixed Shimura subdatum $(Q,\cY^+)\hookrightarrow (P,\cX^+)$ such that $\tilde{z}\in\cY^+$ and $W_{H_{\tilde{Z}}}\lhd Q$. Define $W^*$ to be the $G$-subgroup (of $W$) generated by $W_Q:=\cR_u(Q)$, then $W_{H_{\tilde{Z}}}\lhd W^*$ since $W_{H_{\tilde{Z}}}$ is a $G$-group.

Fix a Levi decomposition $H_{\tilde{Z}}=W_{H_{\tilde{Z}}}\rtimes G_{H_{\tilde{Z}}}$ and choose a compatible Levi decomposition $P=W\rtimes G$ (as is shown in Lemma~\ref{compatible Levi}). Let $P^*$ be the group generated by $GQ$, then $\cR_u(P^*)=W^*$ and $P^*/W^*=G$. The group $P^*$ defines a connected mixed Shimura datum $(P^*,\cX^{*+})$ with $\cX^{*+}=P^*(\R)^+U^*(\C)\tilde{z}$. Now $\tilde{Z}=H_{\tilde{Z}}(\R)^+U_{H_{\tilde{Z}}}(\C)\tilde{z}\subset\cX^{*+}$. But $\tilde{Z}$ is Hodge generic in $\cX^+$ by assumption, hence $P=P^*$ and $W=W^*$. So $W_{H_{\tilde{Z}}}\lhd W$ and hence $W_{H_{\tilde{Z}}}\lhd P$.

Use the notation in $\mathsection$\ref{structure of the underlying group}. We are done if we can prove:
\[
\forall u\in U,~\forall v\in V,\text{ \textit{and} }\forall g\in G_{H_{\tilde{Z}}},~(u,v,1)(0,0,g)(-u,-v,1)\in H_{\tilde{Z}}.
\]
By Corollary \ref{decomposition of V,U as G-module}, there exist decompositions
\[
U=U_N\oplus U_N^\bot\qquad V=V_N\oplus V_N^\bot
\]
as $G$-modules such that $G_N$ acts trivially on $U_N^\bot$ and $V_N^\bot$. Now 
\begin{align*}
(u,v,1)(0,0,g)(-u,-v,1) &
=(u,v,g)(-u,-v,1) \\
&=(u-g\cdot u,v-g\cdot v,g) \\
&=((u_N+u_N^\bot)-g\cdot(u_N+u_N^\bot),(v_N+v_N^\bot)-g\cdot(v_N+v_N^\bot),g) \\
&=(u_N-g\cdot u_N,v_N-g\cdot v_N,g) \\
&=(u_N,v_N,1)(0,0,g)(-u_N,-v_N,1)\in H_{\tilde{Z}},
\end{align*}
where the last inclusion follows from Lemma~\ref{stabilizer normal in N}.
\end{enumerate}
\end{proof}
\end{thm}

\section{Ax-Lindemann-Weierstra{\ss} Part 2: Estimate}\label{Ax-Lindemann Part 2: Estimate}

This section is devoted to prove Proposition~\ref{projection of stabilizer is stabilizer of the base}. The proof uses essentially the ``block family'' version of Pila-Wilkie's counting theorem \cite[Theorem~3.6]{PilaO-minimality-an}.

Keep notation and assumptions as in the last section and denote by $\pi\colon(P,\cX^+)$ $\rightarrow(G,\cX^+_G)$. The group $G=Z(G)^\circ H_1...H_r$ is an almost direct product, where $H_i$'s are non-trivial simple groups and are normal in $G$. We have a decomposition
\[
(G^{\ad},\cX_G^+)\cong\prod_{i=1}^r(H_i^{\ad},\cX_{H,i}^+)
\]
by \cite[3.6]{MoonenLinearity-prope}. Let $S_G^{\ad}:=\Gamma_G^{\ad}\backslash\cX_G^+$. Shrinking $\Gamma_G^{\ad}$ if necessary, we may assume $S_G^{\ad}\cong\prod_{i=1}^rS_{H,i}$, where $S_{H,i}$ is a connected pure Shimura variety associated with $(H_i^{\ad},\cX_{H,i}^+)$.

Without loss of generality we may assume $G_N=H_1...H_l$. It suffices to prove $H_i<G_{H_{\tilde{Z}}}$ for each $i=1,...,l$. The case $l=0$ is trivial, so we assume that $l\geqslant 1$. Define $Q_i:=\pi^{-1}(H_i)$.

\subsection{Fundamental set and definability}
The goal of this subsection is to prove that there exists $\cF\subset\cX^+$ a fundamental set for the action of $\Gamma$ on $\cX^+$ such that $\unif|_{\cF}$ is definable.

First of all, by the Reduction Lemma~(Lemma~\ref{reduction lemma}), it suffices to prove the existence of such a fundamental set for $(P,\cX^+)$ pure and $(P,\cX^+)=(P_{2g},\cX^+_{2g})$. The case where $(P,\cX^+)$ is pure is guaranteed by Klingler-Ullmo-Yafaev \cite[Theorem~4.1]{KlinglerThe-Hyperbolic-}. Now we prove the case $(P,\cX^+)=(P_{2g},\cX^+_{2g})$.

We draw the following diagram to make the notation more clear:
\[
\begin{diagram}
\cX_{2g}^+ &\rOnto^{\pi_{P/U}} &\cX_{2g,\mathrm{a}}^+\\
\dTo^{\unif} & &\dTo^{\unif_{P/U}} \\
S &\rOnto^{[\pi_{P/U}]} &S_{P/U}
\end{diagram}.
\]

In this case, $[\pi_{P/U}]\colon S\rightarrow S_{P/U}$ is an algebraic $\G_m$-torsor. By Peterzil-Starchenko \cite[Theorem~1.3]{PeterzilDefinability-of}, there exists a fundamental set $\cF_{P/U}$ for the action of $\Gamma/\Gamma_U$ on $\cX^+_{2g,\mathrm{a}}$ such that $\unif_{P/U}|_{\cF_{P/U}}$ is definable (recall that if $g=0$, then $\cX^+_{2g}=\C$, $S=\C^*$, $\unif=\exp$ and $S_{P/U}$ is a point). Let us now construct a fundamental set for the action of $\Gamma$ on $\cX^+_{2g}$ such that $\unif|_{\cF}$ is definable and $\pi_{P/U}(\cF)=\cF_{P/U}$.

Since any variety over a field is quasi-compact in the Zariski topology, there exists a finite Zariski open covering $\{V_\alpha\}_{\alpha\in\Lambda}$ of $S_{P/U}$ such that $S|_{V_\alpha}\cong \C^*\times V_\alpha$ and these isomorphisms are algebraic. Define $U_\alpha:=S|_{V_\alpha}=[\pi_{P/U}]^{-1}(V_\alpha)$ for every $\alpha\in\Lambda$.
Then we have
\[
\unif|_{\unif^{-1}(U_\alpha)}\colon \unif^{-1}(U_\alpha)\xrightarrow[\varphi]{\sim} U_{2g}(\C)\times \unif_{P/U}^{-1}(V_\alpha)\rightarrow (\C^*)\times V_{\alpha}\cong U_\alpha,
\]
where $\varphi$ is semi-algebraic (Proposition~\ref{realization of the uniformizing space}), the last isomorphism is algebraic and the middle morphism is $(\exp,\unif_{P/U}|_{\unif_{P/U}^{-1}}(V_\alpha))$. Let $\cF_U:=\{s\in\C|$ $-1<\Re(s)<1\}$ and let $\cF_{\alpha}:=\varphi^{-1}(\cF_U\times\cF_{P/U,\alpha})$. Then $\unif|_{\cF_{\alpha}}$ is definable.
Now $\cF:=\cup\cF_{\alpha}$ (remember that this is a finite union) satisfies the conditions we want.

Now we return to arbitrary $(P,\cX^+)$. We have proved the existence of an $\cF$ as stated at the beginning of this subsection. Let us choose such an $\cF$ more carefully. First of all replace $\cF$ by $\gamma\cF$ if necessary to make sure $\cF\cap\tilde{Z}\neq\emptyset$. Next define $\cF_G:=\pi(\cF)\subset\cX_G^+\cong\prod_{i=1}^r\cX_{H,i}^+$. Denote by $q_i$ the $i$-th projection and $\cF_{H,i}:=q_i(\cF_G)$. There exist some $\gamma_{1}=1,...,\gamma_{s}\in\Gamma_G<\Gamma$ such that $\prod_{i=1}^r\cF_{H,i}\subset\cup_{j=1}^s\gamma_{j}\cF_G$. Consider
\[
\cF^\prime:=(\bigcup_{j=1}^s\gamma_j\cF)\cap\pi^{-1}(\prod_{i=1}^r\cF_{H,i}),
\]
then $\cF^\prime$ is a fundamental set for the action of $\Gamma$ on $\cX^+$ and $\unif|_{\cF^\prime}$ is definable. Furthermore, $\pi(\cF^\prime)=\prod_{i=1}^r\cF_{H,i}$ and $\cF_{H,i}=q_i\pi(\cF^\prime)$. We still have $\cF^\prime\cap\tilde{Z}\neq\emptyset$ since $\cF\subset\cF^\prime$. Now replace $\cF$ by $\cF^\prime$.

\subsection{Counting points and conclusion}
We shall work from now on with an $\cF$ satisfying the conditions in the last paragraph of the previous subsection.
By Lemma~\ref{projection of maximal still maximal}, $\bar{\tilde{Y}_G}=\prod_{i=1}^lH_i(\R)^+\tilde{z}_G$. Fix a point 
$\tilde{z}\in\cF\cap\tilde{Z}$. Define the following Shimura morphisms for each $i=1,...,l$
\[
\begin{diagram}
(G,\cX_G^+) &\rTo^{p_i} &(G_i,\cX_{G,i}^+):=(G^{\ad},\cX_G^+)/\prod_{j\neq i}H_j^{\ad} \\
\dTo^{\unif_G} & &\dTo^{\unif_{G,i}} \\
S_G &\rTo^{[p_i]} &S_{G,i}
\end{diagram}.
\]
Fix $i\in\{1,...,l\}$. Define $\tilde{Y}_{G,i}:=p_i(\tilde{Y}_G)=H_i^{\ad}(\R)^+\pi_i(\tilde{z}_G)$, $\tilde{Z}_{G,i}:=p_i(\tilde{Z}_G)$ and $Y_{G,i}:=[p_i](Y_G)$, then $\unif_{G,i}(\tilde{Z}_{G,i})$ is Zariski dense in $\bar{Y_{G,i}}$ by Lemma~\ref{projection of maximal still maximal}. If $\dim(\tilde{Z}_{G,i})=0$, then $\tilde{Z}_{G,i}$ is a finite set of points since it is algebraic. But then $\unif_{G,i}(\tilde{Z}_{G,i})$, and hence $\bar{Y_{G,i}}=\bar{\unif_{G,i}(\tilde{Z}_{G,i})}$ is also a finite set of points. So $\dim(Y_{G,i})=0$, which contradicts $\tilde{Y}_{G,i}=H_i^{\ad}(\R)^+\pi_i(\tilde{z}_G)$. To sum it up, $\dim(\tilde{Z}_{G,i})>0$. For further convenience, we will denote by $\pi_i:=p_i\circ\pi$.

Take an algebraic curve $C_{G,i}\subset \tilde{Z}_{G,i}$ passing through $\pi_i(\tilde{z})$. Now $\pi_i(\tilde{Z}\cap\pi_i^{-1}(C_{G,i}))=\tilde{Z}_{G,i}\cap C_{G,i}=C_{G,i}$, and hence there exists an algebraic curve $C\subset\tilde{Z}\cap\pi_i^{-1}(C_{G,i})$ passing through $\tilde{z}$ such that $\dim(\pi_i(C))=1$.

Let $\cF_{G,i}:=p_i(\cF_G)$, then it is a fundamental set of $\unif_{G,i}$ and $\unif_{G,i}|_{\cF_{G,i}}$ is definable. We define for any irreducible semi-algebraic subvariety $A$ (resp. $A_{G,i}$) of $\unif^{-1}(Y)$ (resp. $\unif_{G,i}^{-1}(\bar{Y_{G,i}})$) the following sets:
define
\small
\[
\begin{array}{cc}
\Sigma^{(i)}(A):=\{g\in Q_i(\R)|\dim(gA\cap \unif^{-1}(Y)\cap\cF)=\dim(A)\} \\
(\text{resp. }\Sigma^{(i)}_G(A_{G,i}):=\{g\in H_i^{\ad}(\R)|\dim(gA_{G,i}\cap \unif_{G,i}^{-1}(\bar{Y_{G,i}})\cap\cF_{G,i})=\dim(A_{G,i})\})
\end{array}
\]
\normalsize
and
\small
\[
\begin{array}{cc}
\Sigma^{\prime(i)}(A):=\{g\in Q_i(\R)|g^{-1}\cF\cap A\neq\emptyset\} \\
(\text{resp. }\Sigma^{\prime(i)}_G(A_{G,i}):=\{g\in H_i^{\ad}(\R)|g^{-1}\cF_{G,i}\cap A_{G,i}\neq\emptyset\}).
\end{array}.
\]
\normalsize
Then $\Sigma^{(i)}(A)$ and $\Sigma^{(i)}_G(A_{G,i})$ are by definition definable. Let $\Gamma_{G,i}^{\ad}:=p_i(\Gamma_G^{\ad})$.

\begin{lemma}\label{the two Sigma-sets coincide}
$\Sigma^{\prime(i)}(A)\cap\Gamma=\Sigma^{(i)}(A)\cap\Gamma$ (resp. $\Sigma^{\prime(i)}_G(A_{G,i})\cap\Gamma_{G,i}^{\ad}=\Sigma^{(i)}_G(A_{G,i})\cap\Gamma_{G,i}^{\ad}$).
\begin{proof} The proof, which we include for completeness, is the same as \cite[Lemma~5.2]{UllmoThe-Hyperbolic-}. First of all $\Sigma^{(i)}(A)\cap\Gamma\subset\Sigma^{\prime(i)}(A)\cap\Gamma$ by definition. Conversely for any $\gamma\in\Sigma^{\prime(i)}(A)\cap\Gamma$, $\gamma^{-1}\cF\cap A$ contains an open subspace of $A$ since $\cF$ is by choice open in $\cX^+$. Hence $\gamma A\cap \unif^{-1}(Y)\cap\cF=\gamma A\cap\cF$ contains an open subspace of $\gamma A$ which must be of dimension $\dim(A)$. Hence $\gamma\in\Sigma^{(i)}(A)\cap\Gamma$. The proof for $A_{G,i}$ is the same.
\end{proof}
\end{lemma}

This lemma implies
\vspace{-2mm}
\small
\begin{equation}\label{relation between Sigma of curve and Sigma of Z}
\begin{array}{cc}
\Sigma^{(i)}(C)\cap\Gamma=\Sigma^{\prime(i)}(C)\cap\Gamma\subset
\Sigma^{\prime(i)}(\tilde{Z})\cap\Gamma=\Sigma^{(i)}(\tilde{Z})\cap\Gamma \\
(\text{resp. }\Sigma^{(i)}_G(C_{G,i})\cap\Gamma_{G,i}^{\ad}=\Sigma^{\prime(i)}_G(C_{G_i})\cap\Gamma_{G,i}^{\ad}\subset\Sigma^{\prime(i)}_G(\bar{\tilde{Z}_{G,i}})\cap\Gamma_{G,i}^{\ad}=\Sigma^{(i)}(\bar{\tilde{Z}_{G,i}})\cap\Gamma_{G,i}^{\ad})
\end{array}.
\end{equation}
\normalsize

\begin{lemma}\label{projection of theta is theta}
$\pi_i(\Gamma\cap\Sigma^{(i)}(C))=\Gamma_{G,i}^{\ad}\cap\Sigma^{(i)}_G(C_{G,i})$.
\begin{proof}
By Lemma~\ref{the two Sigma-sets coincide}, it suffices to prove
$\pi_i(\Gamma\cap\Sigma^{\prime(i)}(C))=\Gamma_{G,i}^{\ad}\cap\Sigma^{\prime(i)}_G(C_{G,i})$.
The inclusion $\subset$ is clear by definition. For the other inclusion, $\forall\gamma_{G,i}\in\Gamma_{G,i}^{\ad}\cap\Sigma^{\prime(i)}_G(C_{G,i})$, $\exists c_{G,i}\in C_{G,i}$ such that $\gamma_{G,i}\cdot c_{G,i}\in\cF_{G,i}$.

Take a point $c\in C$ such that $\pi_i(c)=c_{G,i}$ and define $c_G:=\pi(c)\in\cX_G^+$. Suppose that under the decomposition
\[
(G^{\ad},\cX_G^+)\cong\prod_{i=1}^r(H_i^{\ad},\cX_{H,i}^+)
\]
of \cite[3.6]{MoonenLinearity-prope}, $c_G=(c_{G,1},...,c_{G,r})$. Then by choice of $\cF_G$, there exists $\gamma_G^\prime\in\Gamma_G^{\ad}$ whose $i$-th coordinate is precisely the $\gamma_{G,i}$ in the last paragraph such that $\gamma_G^\prime\cdot c_G\in\cF_G$.

Let $\gamma_G\in\Gamma_G$ be such that its image under $\Gamma_G\rightarrow\Gamma_G^{\ad}$ is $\gamma_G^\prime$, then $\gamma_G\cdot c\in\pi^{-1}(\cF_G)$.
Therefore there exist $\gamma_V\in\Gamma_V,~\gamma_U\in\Gamma_U$ such that $(\gamma_U,\gamma_V,\gamma_G)c\in\cF$. 
Denote by $\gamma=(\gamma_U,\gamma_V,\gamma_G)$, then $\gamma\in\Gamma\cap\Sigma^{\prime(i)}(C)$ and $\pi_i(\gamma)=\gamma_{G,i}$.
\end{proof}
\end{lemma}

For $T>0$, define
\[
\Theta^{(i)}_G(C_{G,i},T):=\{\gamma_G\in\Gamma_{G,i}^{\ad}\cap\Sigma^{(i)}_G(C_{G,i})|H(\gamma_G)\leqslant T\}.
\]

\begin{prop}\label{theta for the base big enough}
There exists a constant $\delta>0$ s.t. for all $T\gg0$, $|\Theta^{(i)}_G(C_{G,i},T)|\geqslant T^\delta$.
\begin{proof} This follows directly from \cite[Theorem~1.3]{KlinglerThe-Hyperbolic-} applied to \small$((G_i,\cX_{G,i}^+), S_{G,i}, \bar{\tilde{Z}_{G,i}})$.
\end{proof}
\end{prop}
\normalsize
Let us prove how these facts imply $H_i<G_{H_{\tilde{Z}}}$.

Take a faithful representation $G^{\ad}\hookrightarrow\GL_n$ which sends $\Gamma_G^{\ad}$ to $\GL_n(\Z)$. Consider the definable set $\Sigma^{(i)}_G(C_{G,i})$. By the theorem of Pila-Wilkie (\cite[Theorem~3.6]{PilaO-minimality-an}), there exist $J=J(\delta)$ definable block families
\[
B^j\subset\Sigma^{(i)}_G(C_{G,i})\times\R^l,\qquad j=1,...,J
\]
and $c=c(\delta)>0$ such that for all $T\gg0$, $\Theta^{(i)}_G(C_{G,i},T^{1/2n})$ is contained in the union of at most $cT^{\delta/4n}$ definable blocks of the form $B^j_y$ ($y\in\R^l$). By Proposition~\ref{theta for the base big enough}, there exist a $j\in\{1,...,J\}$ and a block $B_{G,i}:=B^j_{y_0}$ of $\Sigma^{(i)}_G(C_{G,i})$
containing at least $T^{\delta/4n}$ elements of $\Theta^{(i)}_G(C_{G,i},T^{1/2n})$.

Let $\Sigma^{(i)}:=\Sigma^{(i)}(C)\cap\Sigma^{(i)}(\tilde{Z})$, which is by definition a definable set. Consider $X^j:=(\pi_i\times 1_{\R^l})^{-1}(B^j)\cap(\Sigma^{(i)}\times\R^l)$, which is a definable family since $\pi_i$ is algebraic.

By \cite[Ch. 3, 3.6]{DriesTame-Topology-a}, there exists a number $n_0>0$ such that each fibre $X^j_y$ has at most $n_0$ connected components.
So the definable set $\pi_i^{-1}(B_{G,i})\cap\Sigma^{(i)}$ has at most $n_0$ connected components. 
Now
\small
\[
\pi_i(\pi_i^{-1}(B_{G,i})\cap\Sigma^{(i)}\cap\Gamma)=B_{G,i}\cap\pi_i(\Sigma^{(i)}(C)\cap\Gamma)=B_{G,i}\cap\Sigma^{(i)}_G(C_{G,i})\cap\Gamma_{G,i}^{\ad}=B_{G,i}\cap\Gamma_{G,i}^{\ad}
\]
\normalsize
by \eqref{relation between Sigma of curve and Sigma of Z} and Lemma~\ref{projection of theta is theta}. So there exists a connected component $B$ of $\pi_i^{-1}(B_{G,i})\cap\Sigma^{(i)}$ such that $\pi_i(B\cap\Gamma)$ contains at least $T^{\delta/4n}/n_0$ elements of $\Theta^{(i)}_G(C_{G,i},T^{1/2n})$.

We have $B\tilde{Z}\subset \unif^{-1}(Y)$ since $\Sigma^{(i)}(\tilde{Z})\tilde{Z}\subset \unif^{-1}(Y)$ by analytic continuation, and $\tilde{Z}\subset\sigma^{-1}B\tilde{Z}$ for any $\sigma\in B\cap\Gamma$. But $B$ is connected, and therefore $\sigma^{-1}B\tilde{Z}=\tilde{Z}$ by maximality of $\tilde{Z}$ and \cite[Lemma~4.1]{PilaAbelianSurfaces}. So $\forall\sigma\in B\cap\Gamma$, 
\[
 B\subset\sigma\stab_{Q_i(\R)}(\tilde{Z}).
\]

Fix a $\gamma_0\in B\cap\Gamma$ such that $\pi_i(\gamma_0)\in\Theta^{(i)}_G(C_{G,i},T^{1/2n})$. We have already shown that $\pi_i(B\cap\Gamma)$ contains at least $T^{\delta/4n}/n_0$ elements of $\Theta^{(i)}_G(C_{G,i},T^{1/2n})$. For any $\gamma^\prime_{G,i}\in\pi_i(B\cap\Gamma)\cap\Theta^{(i)}_G(C_{G,i},T^{1/2n})$, let $\gamma^\prime$ be one of its pre-images in $B\cap\Gamma$. Then $\gamma:=\gamma^{\prime-1}\gamma_0$ is an element of $\Gamma\cap\stab_{Q_i(\R)}(\tilde{Z})=\Gamma_{\tilde{Z}}\cap Q_i(\R)$ such that $H(\pi_i(\gamma))\ll T^{1/2}$. Therefore for $T\gg0$, $\pi_i(\Gamma_{\tilde{Z}})\cap H_i^{\ad}(\R)$ contains at least $T^{\delta/4n}/n_0$ elements $\gamma_{G,i}$ such that $H(\gamma_{G,i})\leqslant T$. Hence $\dim(\pi_i(H_{\tilde{Z}})\cap H_i^{\ad})>0$ since $\pi_i(H_{\tilde{Z}})\cap H_i^{\ad}$ contains infinitely many rational points. But $\pi_i(H_{\tilde{Z}})=p_i\pi(H_{\tilde{Z}})=p_i(G_{H_{\tilde{Z}}})$ by definition. So $H_i^{\ad}<p_i(G_{H_{\tilde{Z}}})$ since $H_i^{\ad}$ is simple and $p_i(G_{H_{\tilde{Z}}})\cap H_i^{\ad}\lhd H_i^{\ad}$ by Corollary \ref{normal subgroups}.

As a normal subgroup of $G_N$, $G_{H_{\tilde{Z}}}$ is the almost direct product of some $H_j$'s $(j=1,...,l)$. So $H_i^{\ad}<p_i(G_{H_{\tilde{Z}}})$ implies $H_i<G_{H_{\tilde{Z}}}$. Now we are done.

\begin{rmk}\label{EssentialUseOfPilaWilkie}
In the proof of the pure case by Klingler-Ullmo-Yafaev \cite{KlinglerThe-Hyperbolic-}, it suffices to use a non-family version of Pila-Wilkie (\cite[Theorem~6.1]{KlinglerThe-Hyperbolic-}). However this is not enough for our proof, since otherwise the $n_0$ would depend on $T$. Hence it is important to use a family version of Pila-Wilkie (\cite[Theorem~3.6]{PilaO-minimality-an}).
\end{rmk}

\section{Ax-Lindemann-Weierstra{\ss} Part 3: The unipotent part}\label{Ax-Lindemann Part 3: The unipotent part}

We prove in this section Theorem~\ref{Ax-Lindemann for complex semi-abelian varieties}. We use the same notation as the first paragraph of $\mathsection$\ref{Results for the unipotent part} and $\mathsection$\ref{Ax-Lindemann for the unipotent part}. Assume $\dim_\C T=m$ and $\dim_\C A=n$.

\begin{proof}[Proof of Theorem~\ref{Ax-Lindemann for complex semi-abelian varieties}]
First of all we may assume that $\tilde{Z}$ is of positive dimension since every point is a weakly special subvariety of dimension 0. For any fundamental set $\cF$ of the action of $\Gamma_W$ on $W(\R)U(\C)$, define
\[
\Sigma(\tilde{Z}):=\{g\in W(\R)|\dim(g\tilde{Z}\cap \unif^{-1}(Y)\cap\cF)=\dim(\tilde{Z})\}
\]
and
\[
\Sigma^\prime(\tilde{Z}):=\{g\in W(\R)|g^{-1}\cF\cap\tilde{Z}\neq\emptyset\},
\]
then by Lemma~\ref{the two Sigma-sets coincide}, \begin{equation}\label{two Sigma coincide Z semi-abelian}
\Sigma(\tilde{Z})\cap\Gamma_W=\Sigma^\prime(\tilde{Z})\cap\Gamma_W
\end{equation}

Let $\Gamma_U:=\Gamma\cap U(\Q)$ and let $\Gamma_V:=\Gamma_W/\Gamma_U$.

\textit{\boxed{Case~i:~$E=A$.}} This is \cite[Theorem~2.1 and pp9 Remark 1]{PilaRational-points}. A proof can be found in Appendix. In this case, $W=V$ and $\Gamma_V=\oplus_{i=1}^{2n}\Z e_i\subset\lie(A)=\C^n=\R^{2n}$ is a lattice. Denote by $\unif\colon\lie(A)\rightarrow A$. Let $\cF_V:=\Sigma_{i=1}^{2n}(-1,1)e_i$, then $\cF_V$ is a fundamental set for the action of $\Gamma_V$ on $\lie(A)$ such that $\unif|_{\cF_V}$ is definable.

\textit{\boxed{Case~ii:~$E=T$.}} This is a consequence of Ax's theorem \cite{AxOn-Schanuels-co} \cite[Corollary 3.6]{MartinThesis}. A proof of this can be found in Appendix. In this case, $W=U$. Let $\cF_U:=\{s\in\C|-1<\Re(s)<1\}^m$, then $\cF_U$ is a fundamental set for the action of $\Gamma_U$ on $U(\C)$ such that $\unif|_{\cF_U}$ is definable.

\textit{\boxed{Case~iii:~general~$E$.}} Unlike the rest of the paper, the symbol $\pi$ in this section denotes the map
\begin{equation}\label{diagram for the fiber}
\begin{diagram}
W(\R)U(\C) &\rTo^{\pi} &V(\R) \\
\dTo_{\unif} & &\dTo_{\unif_V} \\
E &\rTo^{[\pi]} &A
\end{diagram}.
\end{equation}

Take $\cF_V\subset V(\R)$ any fundamental set for the action of $\Gamma_V$ on $V(\R)$ such that $\unif_V|_{\cF_V}$ is definable. We claim that:
\begin{equation}\label{FundamentalSetForW}
\begin{array}{c}
\text{There exists a fundamental set }\cF\text{ for the action of }\Gamma_W\text{ on }W(\R)U(\C)\\ \text{ such that }\unif|_{\cF}\text{ is definable and }\pi(\cF)=\cF_V.
\end{array}
\end{equation}
By Reduction Lemma~(Lemma~\ref{reduction lemma}), it suffices to prove this for $E=E_1\times_A...\times_AE_m$ where $E_i$'s are $\G_m$-torsors over $A$. But then it suffices to prove for the case $m=1$. For this case, the proof is similar to $\mathsection$\ref{Ax-Lindemann Part 2: Estimate}.1.

Let $Y_0$ be the minimal closed irreducible subvariety of $E$ such that $\tilde{Z}\subset \unif^{-1}(Y_0)$, then $\tilde{Z}$ is maximal irreducible algebraic in $\unif^{-1}(Y_0)$. Hence we may assume that $Y=Y_0$. Let $N$ be the connected algebraic monodromy group of $Y^{\sm}$ and let $V_N:=(N\cap W)/(N\cap U)$. Let $\tilde{Y}$ be the complex analytic irreducible component of $\unif^{-1}(Y)$ which contains $\tilde{Z}$. For further convenience, we will denote by $\tilde{Z}_V:=\pi(\tilde{Z})$, $\tilde{Y}_V:=\pi(\tilde{Y})$ and $Y_V:=[\pi](Y)$.

Repeating the proof of Lemma~\ref{projection of maximal still maximal} (but using the conclusion of \textit{Case i} instead of \cite[Theorem~1.1]{KlinglerThe-Hyperbolic-}), we get that $\bar{\tilde{Y}_V}=V_N(\R)+\tilde{z}_V$ for some $\tilde{z}_V\in\tilde{Z}_V$ is weakly special, and $\bar{\unif_V(\tilde{Z}_V)}=\bar{Y_V}$. Remark that by GAGA, these closures could be taken in the complex analytic topology (i.e. the topology whose closed sets are complex analytic sets) or the Zariski topology. If $V_N$ is trivial, then we are actually in the situation of \textit{Case ii}, and therefore $\tilde{Z}$ is weakly special. From now on, suppose that $\dim(V_N)>0$. Replace $S$ by its smallest special subvariety containing $Y_0$, then $N\lhd P$ by Theorem~\ref{all variations of mixed Shimura varieties are admissible}. Hence $V_N$ is a $G=\MT(b)$-submodule of $V$.

Define $W_0:=(\bar{\Gamma_W\cap\stab_{W(\R)U(\C)}(\tilde{Z})}^{\Zar})^{\circ}$, $U_0:=W_0\cap U$ and $V_0:=\pi(W_0)=W_0/U_0$. The proof is somehow technical, so we will divide it into several steps.

\textbf{\textit{\underline{Step~I.}}} Let $V^\dagger$ be the smallest subgroup of $V_N$ such that $\tilde{Z}_V\subset V^\dagger(\R)+\tilde{z}_V$. In Step~I, we will prove $V^\dagger<V_0$.

\textit{\underline{Step~I(i).}}
We know that $A=\Gamma_V\backslash V(\R)$ and $V(\Q)\cong\Gamma_V\otimes_{\Z}\Q$. Consider any $\Q$-quotient group $V^\prime$ of $V$ of dimension $1$
\[
p^\prime\colon V\rightarrow V^\prime
\]
such that $\dim(p^\prime(V^\dagger))=1$. By abuse of notation, we shall denote its induced map $V(\R)\rightarrow V^\prime(\R)$ also by $p^\prime$. Now let $\Gamma_{V^\prime}:=p^\prime(\Gamma_V)$, then $\Gamma_{V^\prime}\cong\Z$ since $p^\prime$ is defined over $\Q$. Write $\Gamma_{V^\prime}=\Z e^\prime$, and let $\cF_{V^\prime}:=(-1,1)e^\prime$. Then $\cF_{V^\prime}$ is a fundamental set for the action of $\Gamma_{V^\prime}$ on $V^\prime(\R)$. Define $A^\prime=\Gamma_{V^\prime}\backslash V^\prime(\R)\cong\Z\backslash\R$, $\unif_{V^\prime}\colon V^\prime(\R)\rightarrow A^\prime$ the uniformization and $[p^\prime]\colon A\rightarrow A^\prime$ the map induced by $p^\prime$. Then $\unif_{V^\prime}|_{\cF_{V^\prime}}$ is definable (even in $\R_{an}$). Define $Y_{V^\prime}:=[p^\prime](Y_V)$ and $\tilde{Y}_{V^\prime}:=p^\prime(\tilde{Y}_V)$.

Let $V^{\prime\prime}:=\ker(p^\prime)$. The exact sequence of free $\Z$-modules
\[
1\rightarrow\Gamma_{V^{\prime\prime}}:=\Gamma_V\cap V^{\prime\prime}(\Q)\cong\Z^{2n-1}\rightarrow\Gamma_V\cong\Z^{2n}\rightarrow\Gamma_{V^\prime}\cong\Z\rightarrow1
\]
splits, and hence $\Gamma_V\cong\Gamma_{V^{\prime\prime}}\oplus\Gamma_{V^\prime}$. This induces $V\cong V^{\prime\prime}\oplus  V^\prime$. Write $\Gamma_{V^{\prime\prime}}=\sum_{i=2}^{2n}\Z e_i^{\prime\prime}$ and take $\cF_{V^{\prime\prime}}:=\sum_{i=2}^n(-1,1)e_i^{\prime\prime}$. Define $\cF_V:=\cF_{V^{\prime\prime}}\oplus\cF_{V^\prime}$. Then $\cF_V$ is a fundamental set for the action of $\Gamma_V$ on $V(\R)$ such that $\unif_V|_{\cF_V}$ is definable (even in $\R_{an}$). Define $\cF$ as in \eqref{FundamentalSetForW}.

Since $p(V^\dagger)=V^\prime$ by choice of $V^\prime$, $\dim_\R p^\prime(\tilde{Z}_V)>0$ by minimality of $V^\dagger$. Hence $p^\prime(\tilde{Z}_V)=V^\prime(\R)$ since $p^\prime(\tilde{Z}_V)$ is connected.

\begin{rmk}\label{WhyWeShouldDoBothSteps}
If we only request $(V^\prime,p^\prime)$ to satisfy $p^\prime(V_N)=1$, then we do not know whether $\dim_\R(p^\prime(\tilde{Z}_V))>0$. This is because we are considering the real analytic topology (i.e. the topology whose closed sets are real analytic sets) on $A^\prime$ and the complex analytic topology (i.e. the topology whose closed sets are complex analytic sets) on $A$, and hence $\bar{\unif_V(\tilde{Z}_V)}=\bar{Y_V}$ does NOT imply $\bar{\unif_{V^\prime}(\tilde{Z}_{V^\prime})}=\bar{Y_{V^\prime}}$. To overcome this problem, we introduce the seemingly strange subgroup $V^\dagger$ of $V_N$. We will prove (\textit{Step~II}) that $V_0$ is a $\MT(b)$-module with the help of $V^\dagger$. Then we prove the comparable result of Theorem~\ref{conclusion}(1) for the unipotent part in \textit{Step~III}.
\end{rmk}

Let $C$ be an $\R$-algebraic subvariety of $\tilde{Z}$ of $\R$-dimension 1 such that $p^\prime\pi(C)=V^\prime(\R)$. Define furthermore
\[
\Sigma(C):=\{g\in W(\R)|\dim_\R(gC\cap \unif^{-1}(Y)\cap\cF)=1\}
\]
and
\[
\Sigma^\prime(C):=\{g\in W(\R)|g^{-1}\cF\cap C\neq\emptyset\}.
\]
The set $\Sigma(C)$ is by definition definable. 
By Lemma~\ref{the two Sigma-sets coincide},
\begin{equation}\label{the two Sigma-sets coincide semi-abelian}
\Sigma^\prime(C)\cap\Gamma_W=\Sigma(C)\cap\Gamma_W \\
\end{equation}
For $M>0$, define
\[
\Theta_{V^\prime}(V^\prime(\R),M)=\{\gamma_{V^\prime}\in\Gamma_{V^\prime}|H(\gamma_{V^\prime})\leqslant M\}.
\]
Then
\begin{equation}\label{the theta is big enough semi-abelian}
|\Theta_{V^\prime}(V^\prime(\R),M)|\gg M.
\end{equation}

\textit{\underline{Step~I(ii)}} is quite similar to the end of $\mathsection$\ref{Ax-Lindemann Part 2: Estimate}. 
Consider the definable set $V^\prime(\R)$. By the theorem of Pila-Wilkie (\cite[Theorem~3.6]{PilaO-minimality-an}), there exist $J$ definable block families
\[
B^j\subset V^\prime(\R)\times\R^l,\qquad j=1,...,J
\]
and $c>0$ such that for all $M\gg0$, $\Theta_{V^\prime}(V^\prime(\R),M^{1/4})$ is contained in the union of at most $cM^{\delta/8}$ definable blocks of the form $B^j_y$ ($y\in\R^l$). By \eqref{the theta is big enough semi-abelian}, there exist a $j\in\{1,...,J\}$ and a block $B_{V^\prime}:=B^j_{y_0}$ of $V^\prime(\R)$
containing at least $M^{\delta/8}$ elements of $\Theta_{V^\prime}(V^\prime(\R),M^{1/4})$.

Let $\Sigma:=\Sigma(C)\cap\Sigma(\tilde{Z})$, which is by definition a definable set. Consider $X^j:=((p^\prime\pi)\times 1_{\R^l})^{-1}(B^j)\cap(\Sigma\times\R^l)$, which is a definable family since $p^\prime\pi$ is $\R$-algebraic.

By \cite[Ch. 3, 3.6]{DriesTame-Topology-a}, there exists a number $n_0>0$ such that each fibre $X^j_y$ has at most $n_0$ connected components.
So the definable set $\pi^{-1}(B_{V^\prime})\cap\Sigma$ has at most $n_0$ connected components. Now
\small
\[
p^\prime\pi((p^\prime\pi)^{-1}(B_{V^\prime})\cap\Sigma\cap\Gamma_W) =B_{V^\prime}\cap p^\prime\pi(\Sigma(C)\cap\Gamma_W)=B_{V^\prime}\cap(V^\prime(\R)\cap\Gamma_{V^\prime}) =B_{V^\prime}\cap\Gamma_{V^\prime}
\]
\normalsize
by \eqref{two Sigma coincide Z semi-abelian}, \eqref{the two Sigma-sets coincide semi-abelian} and the choice of $\cF$ (remember that $\Gamma_V=\Gamma_{V^{\prime\prime}}\oplus\Gamma_{V^\prime}$ and $\cF_V=\cF_{V^{\prime\prime}}\oplus\cF_{V^\prime}$). So there exists a connected component $B$ of $(p^\prime\pi)^{-1}(B_{V^\prime})\cap\Sigma$ such that $p^\prime\pi(B\cap\Gamma_W)$ contains at least $M^{\delta/8}/n_0$ elements of $\Theta_{V^\prime}(V^\prime(\R),M^{1/4})$.

We have $B\tilde{Z}\subset \unif^{-1}(Y)$ since $B\subset\Sigma(\tilde{Z})$ by (complex) analytic continuation, and $\tilde{Z}\subset\sigma_W^{-1}B\tilde{Z}$ for any $\sigma_W\in B\cap\Gamma_W$. But $B$ is connected, and therefore $\sigma_W^{-1}B\tilde{Z}=\tilde{Z}$ by maximality of $\tilde{Z}$ and \cite[Lemma~4.1]{PilaAbelianSurfaces}. So 
\[
 B\subset\sigma_W\stab_{W(\R)}(\tilde{Z}).
\]

Fix a $\sigma_W\in B\cap\Gamma_W$ such that $p^\prime\pi(\sigma_W)\in\Theta_{V^\prime}(V^\prime(\R),M^{1/4})$. We have shown that $p^\prime\pi(B\cap\Gamma_W)$ contains at least $M^{\delta/8}/n_0$ elements of $\Theta_{V^\prime}(V^\prime(\R),M^{1/4})$. For any $\sigma_{V^\prime}\in p^\prime\pi(B\cap\Gamma)\cap\Theta_{V^\prime}(V^\prime(\R),M^{1/4})$, let $\sigma_W^\prime$ be one of its pre-images in $B\cap\Gamma_W$. Then $\gamma_W:=\sigma_W^{-1}\sigma_W^\prime$ is an element of $\Gamma_W\cap\stab_{W(\R)}(\tilde{Z})$ and $H(p^\prime\pi(\gamma_W))\ll M^{1/2}$. Therefore for $M\gg0$, 
$p^\prime\pi(\Gamma_W\cap\stab_{W(\R)}(\tilde{Z}))$ contains at least $M^{\delta/8}/n_0$ elements $\gamma_{V^\prime}$ such that $H(\gamma_{V^\prime})\leqslant M$. Therefore $\dim(p^\prime\pi(W_0))>0$ since it is an infinite set. So $p^\prime\pi(W_0)=V^\prime$ since $\dim(V^\prime)=1$. But $V^\prime$ is an arbitrary $1$-dimensional quotient of $V$ such that $p^\prime(V^\dagger)=V^\prime$. Therefore $V^\dagger<\pi(W_0)=V_0$.

\textbf{\textit{\underline{Step~II.}}} We prove in this step that $V_0$ is a $\MT(b)$-module. This implies that $W_0$ is a $\MT(b)$-subgroup of $W$ by Proposition~\ref{properties of irreducible mixed Shimura datum}(2).

By definition of $V^\dagger$, $\tilde{Z}_V\subset V^\dagger(\R)+\tilde{z}_V$. By definition of $V_0$, $V_0(\R)+\tilde{z}_V\subset\tilde{Z}_V$.
Now the conclusion of \textit{Step~I} implies $V_0=V^\dagger$ and $\tilde{Z}_V=V_0(\R)+\tilde{z}_V$. However $\tilde{Z}_V$ is complex, so $V_0(\R)$ is a complex subspace of $V(\R)$. Therefore by considering the complex structure of $V(\R)$, we get that $V_0(\R)$ is a $\MT(b)(\R)$-module. So $V_0$ is a $\MT(b)$-module.

\textbf{\textit{\underline{Step~III.}}} can be seen as an analogue to the proof of Theorem~\ref{conclusion}(1).
Consider a fibre of $\tilde{Z}$ over a point $v\in \pi(\tilde{Z})$ such that $\pi\colon W(\C)/F_b^0W(\C)\rightarrow\lie(A)$ is flat at $v$ (such a point exists by generic flatness). Let $\tilde{W}$ be an irreducible algebraic component of $\tilde{Z}_{v}$ such that $\dim(\tilde{Z}_{v})=\dim(\tilde{W})$, then since $\pi$ is flat at $v$, 
\[
\dim(\tilde{Z})=\dim(\pi(\tilde{Z}))+\dim(\tilde{Z}_{v})=\dim(\pi(\tilde{Z}))+\dim(\tilde{W}).
\]

Consider the set $\tilde{F}:=W_0(\R)U_0(\C)\tilde{W}$. It is semi-algebraic. The fact $\tilde{W}\subset\tilde{Z}$ implies that $\tilde{F}\subset\tilde{Z}$. So by \cite[Lemma~4.1]{PilaAbelianSurfaces}, there exists an irreducible algebraic subvariety of $W(\C)/F_b^0W(\C)$, say $\tilde{F}_{\mathrm{alg}}$, which contains $\tilde{F}$ and is contained in $\tilde{Z}$. Since
\[
\pi(\tilde{F})=\pi(W_0)(\R)+v=\bar{\pi(\tilde{Z})}
\]
and every fiber of $\pi|_{\tilde{F}_{\mathrm{alg}}}$ has $\R$-dimension at least $\dim_\R(\tilde{W})$,
we have
\[
\dim(\tilde{F}_{\mathrm{alg}})\geqslant\dim(\pi(\tilde{F}))+\dim(\tilde{W})=\dim(\pi(\tilde{Z}))+\dim(\tilde{W})=\dim(\tilde{Z}).
\]
So $\tilde{F}=\tilde{Z}$ since $\tilde{Z}$ is irreducible. In other words, $\tilde{Z}=W_0(\R)U_0(\C)\tilde{Z}_v$ and $\tilde{Z}_v$ is irreducible for any $v\in\pi(\tilde{Z})$.

Next for any $v\in\pi(\tilde{Z})$, let $\tilde{W}^\prime$ be an irreducible algebraic subvariety which contains $\tilde{Z}_v$ and is contained in $\unif^{-1}(Y)_v$, maximal for these properties. Then $\tilde{W}^\prime$ is weakly special by \textit{Case~ii}. Consider $\tilde{F}^\prime:=W_0(\R)U_0(\C)\tilde{W}^\prime$. Let $\tilde{Y}$ be the irreducible component of $\unif^{-1}(Y)$ which contains $\tilde{Z}$, then $\tilde{W}^\prime\subset\tilde{Y}$ and so $\tilde{F}^\prime\subset \tilde{Y}$ by Lemma~\ref{stabiliser of Z stabilises Y}. But $\tilde{F}^\prime$ is semi-algebraic, and hence by \cite[Lemma~4.1]{PilaAbelianSurfaces} there exists an irreducible algebraic subvariety of $W(\C)/F_b^0W(\C)$, say $\tilde{F}^\prime_{\mathrm{alg}}$, which contains $\tilde{F}^\prime$ and is contained in $\tilde{Y}$.
So $\tilde{Z}=W_0(\R)U_0(\C)\tilde{Z}_v\subset\tilde{F}^\prime_{\mathrm{alg}}\subset \unif^{-1}(Y)$, and hence $\tilde{Z}=\tilde{F}^\prime_{\mathrm{alg}}=\tilde{F}^\prime$ by the maximality of $\tilde{Z}$. So $\tilde{Z}_v=\tilde{W}^\prime$, i.e.
\begin{equation}\label{fiber still maximal algebraic}
\begin{array}{c}
\text{For any }v\in\pi(\tilde{Z}),~\tilde{Z}_v\text{ is a maximal irreducible algebraic} \\
\text{subvariety of }W(\C)/F^0W(\C)\text{ contained in }\unif^{-1}(Y)_v.
\end{array}
\end{equation}

Now that $\tilde{Z}_v=\tilde{W}^\prime$ is weakly special, we can write $\tilde{Z}_v=U^{\prime}(\C)+\tilde{z}$ with $U^{\prime}<U$ and $\tilde{z}\in\tilde{Z}_v$. Then $U_0<U^\prime$. The product $W^\prime:=W_0U^\prime$ is a subgroup of $W$, and hence
\[
\tilde{Z}=W_0(\R)U_0(\C)\tilde{Z}_v=W_0(\R)U^\prime(\C)\tilde{z}=W^\prime(\R)U^\prime(\C)\tilde{z}.
\]
So $W_0=W^\prime$ and $U_0=U^\prime$. In other words,
\begin{equation}\label{form of Z tilde}
\tilde{Z}=\tilde{E}=W_0(\R)U_0(\C)\tilde{z}
\end{equation}
for some point $\tilde{z}\in\tilde{Z}_{v}$.

\textbf{\textit{\underline{Step~IV.}}} Let us now conclude that $\tilde{Z}$ is weakly special.

First of all, $U_0\lhd P$ by Proposition~\ref{properties of irreducible mixed Shimura datum}(2). Consider $(P,\cX^+)\xrightarrow{\rho}(P,\cX^+)/U_0$, then by definition $\tilde{Z}$ is weakly special iff $\rho(\tilde{Z})$ is. Replace $(P,\cX^+)$ (resp. $W$, $\tilde{Z}$, $W_0$, $\tilde{z}$) by $(P,\cX^+)/U_0$ (resp. $W/U_0$, $\rho(\tilde{Z})$, $W_0/U_0=V_0$, $\rho(\tilde{z})$), then $V_0$ is a subgroup of $W$ and $\tilde{Z}=V_0(\R)\tilde{z}$. Use the notation of $\mathsection$\ref{structure of the underlying group} and $\mathsection$\ref{Realization of X} and suppose $\tilde{z}=(\tilde{z}_U,\tilde{z}_V)$. By Proposition~\ref{weakly special for semi-abelian}, $\tilde{Z}$ is weakly special iff $\tilde{z}_V\in(N_W(V_0)/U)(\R)$ iff $\Psi(V_0(\R),\tilde{z}_V)=0$. We shall prove the last claim.

Define $Z:=\unif(\tilde{Z})$, $z=\unif(\tilde{z})$ and $z_V=[\pi](z)\in A$, then $\pi(\tilde{Z})=V_0(\R)+\tilde{z}_V$ and $[\pi](Z)=A_0+z_V$ where $A_0=\Gamma_{V_0}\backslash V_0(\R)$ is an abelian subvariety of $A$. We can compute the fiber
\begin{equation}\label{fiber over abelian variety}
Z_{z_V}=\left(\unif(\Gamma_W\tilde{Z})\right)_{z_V}=\tilde{z}_U+\frac{1}{2}\Psi(\Gamma_{V},\tilde{z}_V)+\Gamma_U\quad\bmod\Gamma_U.
\end{equation}

We have $\Psi(V(\R),V(\R))\subset U(\R)$ since $\Psi$ is defined over $\Q$. Let us prove $\Psi(\Gamma_{V},\tilde{z}_V)\subset U(\Q)$. Fix an isomorphism $\Gamma_U\cong\Z^m$, which induces an isomorphism $U(\Q)\cong\Q^m$. Suppose that there exists a $u\in\Psi(\Gamma_{V},\tilde{z}_V)\setminus U(\Q)$, then at least one of the coordinates of $u$ is irrational. Without loss of generality we may assume that its first coordinate $u_1\in\R\setminus\Q$. Denote by $U_1$ the $\Q$-subgroup of $U$ corresponding to the first factor of $U(\Q)\cong\Q^m$, then
\[
\unif\big(\tilde{z}_U+U_1(\R)\big)\subset\bar{Z_{z_V}}
\]
since $\{lu_1\bmod\Z|l\in\Z\}$ is dense in $[0,1)$. So $\bar{Z_{z_V}}$ contains
\[
\unif\big(\tilde{z}_U+U_1(\C)\big),
\]
and so does $Y_{z_V}$ since $\bar{Z}\subset Y$. Let $v:=v_0+\tilde{z}_V\in V(\R)$, then $\tilde{z}_U+U_1(\C)\subset \unif^{-1}(Y)_v$. However $\tilde{Z}_{\tilde{z}_v}=\tilde{z}_U$ by \eqref{form of Z tilde} (recall that we have reduced to $W_0=V_0$ and $U_0=0$), which contradicts \eqref{fiber still maximal algebraic}. Hence $\Psi(\Gamma_{V},\tilde{z}_V)\subset U(\Q)$, and therefore $(1/2)\Psi(N\Gamma_{V},\tilde{z}_V)\subset\Gamma_U$ for some $N\gg0$ (since $\rank\Gamma_{V}<\infty$). Now we can construct a new lattice $\Gamma_W^\prime$ with $N\Gamma_V$ and $\Gamma_U$. $\Gamma_W^\prime$ is of finite index in $\Gamma_W$. Replacing $\Gamma_W$ by $\Gamma_W^\prime$ does not change the assumption or the conclusion of Ax-Lindemann-Weierstra{\ss}, so we may assume 
$(1/2)\Psi(\Gamma_{V},\tilde{z}_V)\subset\Gamma_U$. Now we can define $C^\infty$-morphisms
\[
\begin{diagram}
f\colon &A_0+z_V &\rTo &T \\
&a_0+z_V &\mapsto &\tilde{z}_U+(1/2)\Psi(v_0,\tilde{z}_V)\bmod\Gamma_U
\end{diagram}
\]
and
\[
\begin{diagram}
s\colon &A_0+z_V &\rTo &E|_{A_0+z_V} \\
&a_0+z_V &\mapsto &(\tilde{z}_U+(1/2)\Psi(v_0,\tilde{z}_V),a_0+z_V)\bmod\Gamma_W
\end{diagram}
\]
where $v_0$ is any point of $V_0(\R)$ such that $\unif_V(v_0)=a_0$. But $Z_a$ is a single point for all $a\in A_0+z_V$ by \eqref{fiber over abelian variety}, so $s$ is the inverse of $[\pi]|_Z$, and therefore $s$ is a holomorphic section of $E|_{A_0+z_V}\rightarrow A_0+z_V$. Locally on $U_i\subset A_0+z_V$, $s$ is represented by a holomorphic morphism $U_i\rightarrow T$, which must equal to $f|_{U_i}$ by definition. Hence $f$ is holomorphic since being holomorphic is a local condition. So $f$ is constant.

But $\Psi(0,\tilde{z}_V)=0$, and therefore $(1/2)\Psi(V_0(\R),\tilde{z}_V)\subset\Gamma_U$. But $\Psi(V_0(\R),\tilde{z}_V)$ is continuous and $\Psi(0,\tilde{z}_V)=0$, so $\Psi(V_0(\R),\tilde{z}_V)=0$. Hence we are done.
\end{proof}

\section{Consequence of Ax-Lindemann-Weierstra{\ss}}\label{Consequence of Ax-Lindemann}

\subsection{Weakly special subvarieties defined by a fixed $\Q$-subgroup}
Let $S=\Gamma\backslash\cX^+$ be a connected mixed Shimura variety associated with the connected mixed Shimura datum $(P,\cX^+)$ and let $\unif\colon \cX^+\rightarrow S$ be the uniformization. Suppose that $N$ is a connected subgroup of $P$ s.t. $N/(W\cap N)\hookrightarrow G$ is semi-simple. A subvariety of $S$ is said to be weakly special defined by $N$ if it is of the form $\unif(i(\varphi^{-1}(y^\prime)))$ under the notation of Definition \ref{definition of Pink} s.t. $N=\ker(\varphi)$.
Let $\mathfrak{F}(N)$ be the set of all weakly special subvarieties of $S$ defined by $N$. The goal of this subsection is to prove:
\begin{prop}\label{finiteness of strict special subvarieties}
If $\mathfrak{F}(N)\neq\emptyset$ and $N\ntriangleleft P$, then $\cup_{Z\in\mathfrak{F}(N)}Z$ is a finite union of proper special subvarieties of $S$.
\begin{proof} Take any $F\in\mathfrak{F}(N)$. Let $\cF$ be a fundamental domain for the action $\Gamma$ on $\cX^+$. Suppose that $x^\prime\in\cF$ is such that $F=\unif(N(\R)^+U_N(\C)x^\prime)$.
Consider $Q^\prime:=N_P(N)$, the normalizer of $N$ in $P$. By definition of weakly special subvarieties, there exists $(R^\prime,\cZ^+)\hookrightarrow(P,\cX^+)$ such that $h_{x^\prime}:\S_\C\rightarrow P_\C$ factors through $R^\prime_\C$ and $N\lhd R^\prime$. Hence $R^\prime<Q^\prime$. Define $G_{Q^\prime}:=Q^\prime/(W\cap Q^\prime)$. Then $G_{Q^\prime}/(Z(G)\cap G_{Q^\prime})$ is reductive by \cite[Lemma~4.3]{ClozelEquidistributio} or \cite[Proposition~3.28]{UllmoAutour-de-la-co}, and hence $G_{Q^\prime}$ is reductive. Write
\[
1\rightarrow W\cap Q^\prime\rightarrow Q^\prime\xrightarrow{\pi_{Q^\prime}} G_{Q^\prime}\rightarrow 1.
\]
The group $G_{Q^\prime}=Z(G_{Q^\prime})^\circ G_{Q^\prime}^{\mathrm{nc}}G_{Q^\prime}^{\mathrm{c}}$ is an almost-direct product, where $G_{Q^\prime}^{\mathrm{nc}}$ (resp. $G_{Q^\prime}^{\mathrm{c}}$) is the product of the $\Q$-simple factors whose set of $\R$-points is non-compact (resp. compact). Let $G_Q:=Z(G_{Q^\prime})^\circ G_{Q^\prime}^{\mathrm{nc}}$ and then define $Q:=\pi_{Q^\prime}^{-1}(G_Q)$, then $h_{x^\prime}$ factors through $Q_\C$ and $R^\prime<Q$ by Definition~\ref{connected mixed Shimura datum}(4). So $N\lhd Q$ and $(Q,\cY^+)$, where $\cY^+:=Q(\R)^+U_Q(\C)x^\prime$, is a connected mixed Shimura subdatum of $(P,\cX^+)$. But then $F\subset\unif(\cY^+)\subset\cup_{Z\in\mathfrak{F}(N)}Z$.

Define $\mathfrak{Y}_Q:=\{x\in\cX^+|h_x\text{ factors through }Q_{\C}\}$, then $Q(\R)^+U_Q(\C)\mathfrak{Y}_Q=\mathfrak{Y}_Q$. The discussion of last paragraph tells us that $F\subset\unif(\mathfrak{Y}_Q)$ for any $F\in\mathfrak{F}(N)$. On the other hand, for any $x\in\mathfrak{Y}_Q$, $(Q,\cY^+)$, where $\cY^+:=Q(\R)^+U_Q(\C)x$, is a connected mixed Shimura subdatum of $(P,\cX^+)$ and hence $\unif(N(\R)^+U_N(\C)x)\in\mathfrak{F}(N)$. Therefore $\unif(\mathfrak{Y}_Q)\subset\cup_{Z\in\mathfrak{F}(N)}Z$. To sum it up, $\cup_{Z\in\mathfrak{F}(N)}Z=\unif(\mathfrak{Y}_Q)$.

Now we are done if we can prove
\begin{claim} The set $\mathfrak{Y}_Q$ is a finite union of $Q(\R)^+U_Q(\C)$-conjugacy classes. In other words, $\mathfrak{Y}_Q$ is a finite union of connected mixed Shimura subdata of $(P,\cX^+)$.
\end{claim}
Fix a special point $x$ of $\cX^+$ contained in $\mathfrak{Y}_Q$. There exists by definition a torus $T_x\subset Q$ such that $h_x:\S_\C\rightarrow Q_\C$ factors through $T_{x,\C}$. Furthermore, we may and do assume that $T_{x,\C}$ is a maximal torus of $Q_\C$. Let $T$ be a maximal torus of $P_\C$ defined over $\Q$ such that $T>T_x$. Take a Levi decomposition $P=W\rtimes G$ such that $T<G<P$. Then the composite $\S_\C\xrightarrow{h_x}T_{x,\C}<P_\C\xrightarrow{\pi}G_\C<P_\C$ equals $h_x$ and is defined over $\R$ by Definition~\ref{connected mixed Shimura datum}(1).

For any other special point $y$ of $\cX^+$ contained in $\mathfrak{Y}_Q$, there exists $g\in Q(\C)$ such that $gT_{x,\C}g^{-1}=T_{y,\C}$. The number of the $Q(\R)$-conjugacy classes of maximal tori of $Q_\R$ defined over $\R$ is at most
\[
\#(\ker(H^1\left(\R,N_{Q(\R)}(T_{x,\R})\right)\rightarrow H^1(\R,Q)))<\infty,
\]
where $N_{Q(\R)}(T_{x,\R})$ is the normalizer of $T_{x,\R}$ in $Q(\R)$. So it is equivalent to prove the finiteness of the $Q(\R)^+U_Q(\C)$-conjugay classes in $\mathfrak{Y}_Q$ and to prove the finiteness of the $Q(\R)^+$-conjugacy classes of the morphisms $\S\rightarrow T_{x,\R}$. But $T_x<T<G$, so the $Q(\R)^+$-conjugacy classes of the morphisms $\S\rightarrow T_{x,\R}$ equals the $G_Q(\R)^+$-conjugacy classes of the morphisms $\S\rightarrow T_{x,\R}$. In otherwords, it suffices to prove the claim for $(G,\cX^+_G)$. Now the result follows from \cite[Lemma~4.4(ii)]{ClozelEquidistributio} (or \cite[2.4]{MoonenLinearity-prope} or \cite[Lemma~3.7]{UllmoGalois-orbits-a}).
\end{proof}
\end{prop}

\subsection{Consequence of Ax-Lindemann-Weierstra{\ss}}
Now we use the result of the previous subsection to prove the following theorem, which will be used in the next section to prove the Andr\'{e}-Oort Conjecture.
\begin{thm}\label{consequenceOfAx-Lindemann}
 Let $S=\Gamma\backslash\cX^+$ be a connected mixed Shimura variety associated with the connected mixed Shimura datum $(P,\cX^+)$. Let $Y$ be a Hodge generic irreducible subvariety of $S$. Then there exists an $N\lhd P$ (denote by $U_N:=U\cap N$) s.t. for the diagram
\begin{equation}\label{QuotientDiagram}
\begin{diagram}
(P,\cX^+) &\rTo^{\rho} &(P,\cX^{\prime+}):=(P,\cX^+)/N \\
\dTo_{\unif} & &\dTo_{\unif^\prime} \\
S &\rTo^{[\rho]} &S^\prime
\end{diagram},
\end{equation}
\begin{itemize}
\item 
the union of positive-dimensional weakly special subvarieties which are contained in $Y^\prime:=\bar{[\rho](Y)}$ is NOT Zariski dense in $Y^\prime$;
\item $Y=[\rho]^{-1}(Y^\prime)$.
\end{itemize}
\begin{proof}
Without any loss of generality, we assume that the union of positive-dimensional weakly special subvarieties which are contained in $Y$ is Zariski dense in $Y$.

Take a fundamental domain $\cF$ for the action of $\Gamma$ on $\cX^+$ s.t. $\unif|_{\cF}$ is definable. Such an $\cF$ exists by e.g. $\mathsection$\ref{Ax-Lindemann Part 2: Estimate}.1.

By Reduction Lemma (Lemma \ref{reduction lemma}), we may assume $
\begin{diagram}
(P,\cX^+) &\rInto^{\lambda} &(G_0,\cD^+)\times\prod_{i=1}^r(P_{2g},\cX^+_{2g})
\end{diagram}
$, i.e. replace $(P,\cX^+)$ by $(P^\prime,\cX^{\prime+})$ in the lemma if necessary. Identify $(P,\cX^+)$ with its image under $\lambda$.

Let $\cT$ be the set of the triples $(U^\prime,V^\prime,G^\prime)$ consisting of an $\R$-subgroup of $U_{\R}$, an $\R$-sub-Hodge structure of $V_{\R}$ and a connected $\R$-subgroup of $G_{\R}$ which is semi-simple and has no compact factors. Let
\[
\cG:=\G_m(\R)^r\times\GSp_{2g}(\R)\times G(\R),
\]
then $\cG$ acts on $\cT$ by $(g_U,g_V,g)\cdot(U^\prime,V^\prime,G^\prime):=(g_UU^\prime,g_VV^\prime,gG^\prime g^{-1})$. Also we define the action of a triple $(U^\prime(\R),V^\prime(\R),G^\prime(\R))$ on $\cX^+\cong U(\C)\times V(\R)\times\cX^+_G$ as \eqref{action of P on X}. This action is algebraic.

\begin{lemma}\label{finitely many real subgroups}
Up to the action of $\cG$ on $\cT$, there exist only finitely many such triples.
 \begin{proof}
  First of all by root system and Galois cohomology, there exist only finitely many semi-simple subgroups of $G_{\R}$ up to conjugation by $G(\R)$.
  
  Secondly, $V^\prime$ is by definition a symplectic subspace of $V_{\R}$. Hence a symplectic base of $V^\prime$ extends to a symplectic base of $V_{\R}=V_{2g,\R}$. But $\GSp_{2g}(\R)$ acts transitively on the set of symplectic bases of $V_{2g,\R}$, so there are only finitely many choices for $V^\prime$ up to the action of $\GSp_{2g}(\R)$ (in fact, there are only $g$ choices).
  
  Finally, observe that $\forall(\lambda_1,...,\lambda_r)\in\G_m(\R)^r$ and $(u_1,...,u_r)\in U\cong\oplus_{i=1}^rU_{2g}^{(i)}$,
\[
(\lambda_1,...,\lambda_r)\cdot(u_1,...,u_r)=(\lambda_1u_1,...,\lambda_ru_r)
\]
Now it is clear that $(u_1,...,u_r)$ and $(u_1^\prime,...,u_r^\prime)$ are under the same orbit of the action of $\G_m(\R)^r$ if and only if $u_iu_i^\prime\geqslant0$ with $u_iu_i^\prime=0\Rightarrow u_i=u_i^\prime=0$ for all $i=1,...,r$. Hence up to the action of $\G_m(\R)^r$, there are only finitely many $U^\prime$'s (in fact, there are $2\binom{r}{s}$ $U^\prime$'s of dimension $s$).
 \end{proof}
\end{lemma}

Let $\mathfrak{W}(Y)$ (resp. $\mathfrak{W}_l(Y)$) be the union of weakly special subvarieties of positive dimension (resp. of real dimension $l$) contained in $Y$.

For any $l$ s.t. $\mathfrak{W}_l(Y)\neq\emptyset$, there exist by definition (and Proposition \ref{particular choices of i and varphi}) a subgroup $N_l$ of $P^{\der}$ and a point $x_0\in\cF$ s.t. $\unif(N_l(\R)^+U_{N_l}(\C)x_0)$ is a weakly special subvariety of dimension $l$ contained in $Y$. Note that the triple $(U_{N_l,\R},V_{N_l,\R},G_{N_l,\R}^{+\mathrm{nc}})\in\cT$, where $G_{N_l,\R}^{+\mathrm{nc}}$ is the product of the $\R$-simple factors of $G_{N_l,\R}^+$ which are non-compact. We say that two such subgroups $N_l$, $N^\prime_l$ of $P$ are equivalent if $(U_{N_l,\R},V_{N_l,\R},G_{N_l,\R}^{+\mathrm{nc}})=(U_{N^\prime_l,\R},V_{N^\prime_l,\R},G_{N^\prime_l,\R}^{+\mathrm{nc}})$. By condition (4) of Definition \ref{connected mixed Shimura datum}, $\unif(N_l(\R)^+U_{N_l}(\C)x_0)=\unif(N_l^\prime(\R)^+U_{N_l^\prime}(\C)x_0)$ iff $N_l$ and $N_l^\prime$ are equivalent.

Define
\[
\begin{array}{lr}
 B(N_{l,\R},Y):= &\{(g_U,g_V,g,x)\in \cG\times\cF|~\unif((g_UU_{N_l}(\C),g_VV_{N_l}(\R),gG_{N_l}(\R)^{+\mathrm{nc}}g^{-1}) x)\subset Y \\
 &\text{ and is not contained in }\cup_{l^\prime>l}\mathfrak{W}_{l^\prime}(Y)\}.
\end{array}
\]
Then by analytic continuation,
\begin{equation}\label{set B definable}
\begin{array}{lr}
 B(N_{l,\R},Y)= &\{(g_U,g_V,g,x)\in \cG\times\cF|~pr|_{\cF}((g_UU_{N_l}(\R),g_VV_{N_l}(\R),gG_{N_l}(\R)^{+\mathrm{nc}}g^{-1}) x)\subset Y \\
 &\text{ and is not contained in }\cup_{l^\prime>l}\mathfrak{W}_{l^\prime}(Y)\}.
\end{array}
\end{equation}
 
\begin{lemma}\label{translate of weakly special still weakly special}
For any $(g_U,g_V,g,x)\in B(N_{l,\R},Y)$, define
\[
\tilde{Z}:=(g_UU_{N_l}(\C),g_VV_{N_l}(\R),gG_{N_l}(\R)^{+\mathrm{nc}}g^{-1})x.
\]
Then $\unif(\tilde{Z})$ is a weakly special subvariety of $Y$.
\begin{proof} The set $\tilde{Z}$ is a connected irreducible semi-algebraic subset of $\cX^+$ which is contained in $\unif^{-1}(Y)$. Let $\tilde{Z}^\dagger$ be a connected irreducible semi-algebraic subset of $\cX^+$ which is contained in $\unif^{-1}(Y)$ and which contains $\tilde{Z}$, maximal for these properties. By \cite[Lemma 4.1]{PilaAbelianSurfaces} (there is a typo in the proof: $\C^{2n}$ should be $\C^n$) and Ax-Lindemann-Weierstra{\ss} (Theorem \ref{Ax-Lindemann for type star}), $\tilde{Z}^\dagger$ is complex analytic and each of its complex analytic irreducible component is weakly special. But $\tilde{Z}$ is smooth, so $\tilde{Z}$ is contained in one complex analytic irreducible component of $\tilde{Z}^\dagger$ which we denote by $\tilde{Z}^\prime$. Now we have
\small
\begin{align*}
\dim(\tilde{Z})-\dim(N_l(\R)^+U_{N_l}(\C)x_0) &=\dim(gG_{N_l}(\R)^+g^{-1}\cdot x_G)-\dim(G_{N_l}(\R)^+x_{0,G}) \\
&=\dim(\stab_{G_{N_l}(\R)^+}(x_{0,G}))-\dim(\stab_{gG_{N_l}(\R)^+g^{-1}}(x_G)) \\
&\geqslant 0
\end{align*}
\normalsize
because $\stab_{gG_{N_l}(\R)^+g^{-1}}(x_G)$ is a compact subgroup of $gG_{N_l}(\R)^+g^{-1}$ and $\stab_{G_{N_l}(\R)^+}(x_{0,G})$ is a maximal compact subgroup of $G_{N_l}(\R)^+$. Hence
\[
\dim(\tilde{Z}^\prime)\leqslant l=\dim(N_l(\R)^+U_{N_l}(\C)x_0)\leqslant \dim(\tilde{Z})\leqslant \dim(\tilde{Z}^\prime)
\]
where the first inequality follows from the definition of $B(N_{l,\R},Y)$. Therefore $\tilde{Z}=\tilde{Z}^\prime$ is weakly special.
So $\unif(\tilde{Z})$ is weakly special.
\end{proof}
\end{lemma}

Define
\[
\begin{array}{lr}
C(N_{l,\R},Y):= &\{\underline{t}:=(g_UU_{N_l}(\R),g_VV_{N_l}(\R),gG_{N_l}(\R)^{+\mathrm{nc}}g^{-1})|(g_U,g_V,g)\in \cG\text{ s.t. }\exists x\in\cF\\
&\text{ with } \unif(\underline{t}\cdot x)\subset Y\text{ and is not contained in }\cup_{l^\prime>l}\mathfrak{W}_{l^\prime}(Y)\}
\end{array}.
\]
Let
\[
\begin{diagram}
\mbox{\footnotesize $B({N_{l,\R}},Y)$} &\rTo^{\psi_l} &\mbox{\footnotesize $(\G_m(\R)^r/\stab_{\G_m(\R)^r}U_{N_l}(\R))\times\GSp_{2g}(\R)/\stab_{\GSp_{2g}(\R)}V_{N_l}(\R)\times G(\R)/N_{G(\R)}G_{N_l}(\R)^{+\mathrm{nc}}$} \\
\normalsize
(g_U,g_V,g,x) &\mapsto &(g_UU_{N_l}(\R),g_VV_{N_l}(\R),gG_{N_l}(\R)^{+\mathrm{nc}}g^{-1})
\end{diagram},
\]
then there is a bijection between $\psi_l(B(N_{l,\R},Y))$ and $C(N_{l,\R},Y)$.
\begin{lemma}\label{the set is countable}
The set $C(N_{l,\R},Y)$ (hence $\psi_l(B(N_{l,\R},Y))$) is countable.
\begin{proof}
By Lemma \ref{translate of weakly special still weakly special}, $\unif((g_UU_{N_l}(\C),g_VV_{N_l}(\R),gG_{N_l}(\R)^{+\mathrm{nc}}g^{-1})\cdot x)$ is weakly special. Hence by Proposition \ref{particular choices of i and varphi} there exists a $\Q$-subgroup $N^\prime$ of $P^\der$ s.t.
\begin{equation}\label{new group over Q}
 (g_UU_{N_l}(\C),g_VV_{N_l}(\R),gG_{N_l}(\R)^{+\mathrm{nc}}g^{-1})=(U_{N^\prime}(\C),V_{N^\prime}(\R),G_{N^\prime}(\R)^{+\mathrm{nc}}).
\end{equation}
But $g_UU_{N_l}(\R)=g_UU_{N_l}(\C)\cap U(\R)$ and $U_{N^\prime}(\R)=U_{N^\prime}(\C)\cap U(\R)$, so
\[
 (g_UU_{N_l}(\R),g_VV_{N_l}(\R),gG_{N_l}(\R)^{+\mathrm{nc}}g^{-1})=(U_{N^\prime}(\R),V_{N^\prime}(\R),G_{N^\prime}(\R)^{+\mathrm{nc}}).
\]
So $C(N_{l,\R},Y)$, and therefore $\psi_l(B(N_{l,\R},Y))$ is countable.
\end{proof}
\end{lemma}

\begin{prop}\label{UnionOfWeaklySpecials} For any $l>0$ and $N_l$,
\begin{enumerate}
\item the set $C(N_{l,\R},Y)$ (hence $\psi_l(B(N_{l,\R},Y))$) is finite;
\item the set $\cup_{l^\prime\geqslant l}\mathfrak{W}_{l^\prime}(Y)$ is definable;
\end{enumerate}
\begin{proof}
We prove the two statements together by induction on $l$.

\underline{\textit{Step I.}} Let $d$ be the maximum of the dimensions of weakly special subvarieties of positive dimension contained in $Y$. For any $N_d$, $B(N_{d,\R},Y)$ is definable by \eqref{set B definable}, and hence $\psi_d(B(N_{d,\R},Y))$ is definable since $\psi_d$ is algebraic. So $\psi_d(B(N_{d,\R},Y))$, and therefore $C(N_{d,\R},Y)$, is finite by Lemma \ref{the set is countable}.

Consider all the triples 
\[
\begin{array}{lc}
\mathfrak{W}_d(Y,\cT):=\{(U^\prime,V^\prime,G^\prime)\in\cT| &\exists x\in\cF\text{ with }\unif((U^\prime(\C),V^\prime(\R),G^\prime(\R)^+)x) \\
 &\text{ weakly special of dimension }d\text{ contained in }Y\}.
\end{array}
\]
By Lemma \ref{finitely many real subgroups}, any $\underline{t}\in\mathfrak{W}_d(Y,\cT)$ is of the form $\underline{g}\cdot(U^\prime_i,V^\prime_i,G^\prime_i)$ for $\underline{g}\in\cG$ and $i=1,...,n$.
Furthermore, by Proposition \ref{particular choices of i and varphi}, we may assume
\[
(U^\prime_i,V^\prime_i,G^\prime_i)=(U_{N^\prime_i,\R},V_{N^\prime_i,\R},G_{N^\prime_i,\R}^{+\mathrm{nc}})
\]
for some $N^\prime_i<Q$ ($i=1,....,n$). But we just proved that $C(N^\prime_{i,\R},Y)$ is finite ($\forall i=1,...,n$). Hence $\mathfrak{W}_d(Y,\cT)$ is a finite set. Again by Propostition \ref{particular choices of i and varphi}, each triple of $\mathfrak{W}_d(Y,\cT)$ equals $(U_{N^\prime,\R},V_{N^\prime,\R},G_{N^\prime,\R}^{+\mathrm{nc}})$ for some $N^\prime<P$. We shall denote this triple by $N^\prime$ for simplicity.

Hence
\[
\mathfrak{W}_d(Y)=\bigcup_{N^\prime\in\mathfrak{W}_d(Y,\cT)}\bigcup_{\substack{(0,0,1,x)\\ \in B(N^\prime_{\R},Y)}}\unif((N^\prime(\R)^+U_{N^\prime}(\C)x)
\]
is definable.

\underline{\textit{Step II.}} For any $l$ and $N_l$, $B(N_{l,\R},Y)$ is definable by \eqref{set B definable} and induction hypothesis (2). Arguing as in the previous case we get that $C(N_{l,\R},Y)$ is finite. Define
\[
\begin{array}{lc}
\mathfrak{W}_l(Y,\cT):=\{(U^\prime,V^\prime,G^\prime)\in\cT| &\exists x\in\cF\text{ with }\unif((U^\prime(\C),V^\prime(\R),G^\prime(\R)^+)x) \text{ weakly special} \\
& \text{ of dimension }l\text{ contained in }Y\text{ but not contained in }\cup_{l^\prime>l}\mathfrak{W}_{l^\prime}(Y)\}.
\end{array}
\]
Arguing as in the previous case we can get that $\mathfrak{W}_l(Y,\cT)$ is a finite set and each element of it equals $(U_{N^\prime,\R},V_{N^\prime,\R},G_{N^\prime,\R}^{+\mathrm{nc}})$ for some $N^\prime<P$. Hence
\[
\bigcup_{l^\prime\geqslant l}\mathfrak{W}_{l^\prime}(Y)=\bigcup_{l^\prime>l}\mathfrak{W}_{l^\prime}(Y)\cup\bigcup_{N^\prime\in\mathfrak{W}_l(Y,\cT)}\bigcup_{\substack{(0,0,1,x)\\ \in B(N^\prime_{\R},Y)}}\unif(N^\prime(\R)^+U_{N^\prime}(\C)x)
\]
is definable by induction hypothesis (2).
\end{proof}
\end{prop}

From now on, for any connected subgroup $N^\dagger$ of $P$, we will denote by $\mathfrak{F}(N^\dagger)$ the set of all weakly special subvarieties of $S$ defined by the group $N^\dagger$ (see the beginning of this section) and $\mathfrak{F}(N^\dagger,Y):=\{Z\in\mathfrak{F}(N^\dagger)~s.t.~Z\subset Y\}$.
Remark that when proving Proposition \ref{UnionOfWeaklySpecials}, we have also given the following description of $\mathfrak{W}(Y)=\cup_{l=1}^d\mathfrak{W}_l(Y)$:
\begin{equation}\label{UnionOfPositiveDimensionalWeaklySpecial}
\mathfrak{W}(Y)=\bigcup_{N^\prime}\unif(N^\prime(\R)^+U_N^\prime(\C)\text{-orbits contained in }\unif^{-1}(Y))=\bigcup_{N^\prime}\bigcup_{Z\in\mathfrak{F}(N^\prime,Y)}Z
\end{equation}
which is a finite union on $N^\prime$'s and each $N^\prime$ is of positive dimension. We have assumed that $\mathfrak{W}(Y)$ is Zariski dense in $Y$ (otherwise there is nothing to prove). Therefore by \eqref{UnionOfPositiveDimensionalWeaklySpecial}, there exists an $N_1$ of positive dimension s.t.
\begin{equation}\label{NBigEnoughStep1}
\bigcup_{Z\in\mathfrak{F}(N_1,Y)}Z
\end{equation}
is Zariski dense in $Y$.

Prove now $N_1\lhd P$. If not, then by Proposition \ref{finiteness of strict special subvarieties}, $\cup_{Z\in\mathfrak{F}(N_1)}Z$ equals a finite union of proper special subvarieties of $S$. The intersection of this union and $Y$ is not Zariski dense in $Y$ since $Y$ is Hodge generic in $S$. This is a contradiction. Hence $N_1\lhd P$.

Consider the diagram
\begin{equation}\label{QuotientDiagramStep1}
\begin{diagram}
(P,\cX^+) &\rTo^{\rho_1} &(P_1,\cX^+_1):=(P,\cX^+)/N_1 \\
\dTo^{pr} & &\dTo^{\unif_1} \\
S &\rTo^{[\rho_1]} &S_1
\end{diagram}
\end{equation}
and let $Y_1:=\bar{[\rho_1](Y)}$, which is Hodge generic in $S_1$. Since $\dim(N_1)>0$, $\dim(S_1)<\dim(S)$. It is not hard to prove $[\rho]^{-1}(Y_1)=Y$ by the fact \eqref{NBigEnoughStep1}. If the union of positive-dimensional weakly special subvarieties contained in $Y_1$ is not Zariski dense in $Y_1$, then take $N=N_1$. Otherwise by the same argument, there exists a normal subgroup $N_{1,2}$ of $P_1$ s.t. $\dim(N_{1,2})>0$ and $\cup_{Z\in\mathfrak{F}(N_{1,2},Y_1)}Z$ is Zariski dense in $Y_1$. Let $N_2:=\rho_1^{-1}(N_{1,2})$, then $N_2\lhd P$. Draw the same diagram \eqref{QuotientDiagramStep1} with $N_2$ instead of $N_1$, then we get a mixed Shimura variety $S_2$ with $\dim(S_2)<\dim(S_1)$ and a Hodge generic subvariety $Y_2$ of $S_2$. Continue the process (if the union of positive-dimensional weakly special subvarieties contained in $Y_2$ is Zariski dense in $Y_2$).

Since $\dim(S)<\infty$, this process will end in a finite step. Hence there exists a number $k>0$ s.t. the union of positive-dimensional weakly special subvarieties contained in $Y_k$ is not Zariski dense in $Y_k$. Then $N:=N_k$ is the dezired subgroup of $P$.
\end{proof}
\end{thm}

\section{From Ax-Lindemann-Weierstra{\ss} to Andr\'{e}-Oort}\label{From Ax-Lindemann to Andre-Oort}

For pure Shimura varieties, Ullmo and Pila-Tsimerman have explained separately in \cite[$\mathsection$5]{UllmoQuelques-applic} \cite[$\mathsection$7]{PilaAxLindemannAg} how to deduce the Andr\'{e}-Oort conjecture from the Ax-Lindemann-Weierstra{\ss} theorem with a suitable lower bound for Galois orbits of special points. In this section we first prove that in order to get a suitable lower bound for Galois orbits of special points for a mixed Shimura variety, it is enough to have one for its pure part. Then we show that the idea of Ullmo also works for mixed Shimura varieties.

\subsection{Lower bounds for Galois orbits of special points}
In this subsection, we will consider mixed Shimura data (resp. varieties) instead of only connected ones. See Definition \ref{connected mixed Shimura datum}.

Let $(P,\cX)$ be a mixed Shimura datum. Let $\pi\colon(P,\cX)\rightarrow(G,\cX_G)$ be the projection to its pure part. We use the notation of $\mathsection$\ref{structure of the underlying group}. In particular, we fix a Levi decomposition $P=W\rtimes G$ and an embedding $(G,\cX_G)\hookrightarrow(P,\cX)$ as in \cite[pp 6]{WildeshausThe-canonical-c}.

Let $K$ be an open compact subgroup of $P(\A_f)$ defined as follows: for $M>3$ even, $K_U:=MU(\hat{\Z})$, $K_V:=MV(\hat{\Z})$, $K_W:=K_U\times K_V$ with the group law as in $\mathsection$\ref{structure of the underlying group}, $K_G:=\{g\in G(\hat{\Z})|g\equiv 1\bmod M\}$ and $K:=K_W\rtimes K_G$.

Let $s$ be a special point of $M_K(P,\cX)$ which corresponds to a special point $x\in\cX$. The group $\MT(x)$ is of the form $wTw^{-1}$ for a torus $T\subset G$ and $w\in W(\Q)$. Let $\ord(w)\in\Z_{>0}$ be the smallest integer such that $\ord(w)w\in W(\Z)$. Define \textit{the order of $s$} to be $N(s):=\ord(w)$.
\begin{rmk}\label{SpecialPointOnTheFiberTorsionExplanation}
It is not hard to show that if the fiber of $S\xrightarrow{[\pi]}S_G$ is a semi-abelian variety, then $N(s)$ coincides with the order of $s$ as a torsion point on the fiber (up to a constant).
\end{rmk}

Attached to $(P,\cX)$ there is a number field $E=E(P,\cX)$ called the \textbf{reflex field} and $M_K(P,\cX)$ is defined over $E$ (cf. \cite[11.5]{PinkThesis}). We want a comparison of $|\gal(\bar{\Q}/E)s|$ and $|\gal(\bar{\Q}/E)[\pi](s)|$.

Define $(G^w,\cX_{G^w}):=(wGw^{-1},w^{-1}\cdot\cX_G)$, $K_{G^w}:=G^w(\A_f)\cap K$ and $K^\prime_G:=w^{-1}K_{G^w}w$, then
we have the following commutative diagram:
\[
\begin{diagram}
M_{K_{G^w}}(G^w,\cX_{G^w}) &\rInto &M_K(P,\cX) \\
\dTo^{\wr}_{[w^{-1}\cdot]} & &\dTo_{[\pi]} \\
M_{K^\prime_G}(G,\cX_G) &\rOnto^{\rho} &M_{K_G}(G,\cX_G)
\end{diagram}.
\]
All the morphisms in this diagram are defined over $E$ since the reflex field of $(P,\cX)$, $(G,\cX_G)$ and $(G^w,\cX_{G^w})$ are all $E$. Denote by $s^\prime:=[w^{-1}]\cdot s$. Let $T^w:=wTw^{-1}$. Let $K_T^\prime:=K\cap T^w(\A_f)$ and let $K_T:=K\cap T(\A_f)$. The following inequality follows essentially from \cite[$\mathsection$2.2]{UllmoGalois-orbits-a} (note that we do not need GRH for this inequality since \cite[Lemma 2.13, 2.14]{UllmoGalois-orbits-a} are not used!). We refer to \cite[Theorem 1(1)]{NoteGaloisOrbitOfPureShimuraVariety} for a more precise version.
\begin{equation}\label{ComparisonMixedPure}
|\gal(\bar{\Q}/E)s|=|\gal(\bar{\Q}/E)s^\prime|\geqslant B^{i(T)}|K_T/K_T^\prime||\gal(\bar{\Q}/E)\rho(s^\prime)|=B^{i(T)}|K_T/K_T^\prime||\gal(\bar{\Q}/E)[\pi](s)|
\end{equation}
for some $B\in(0,1)$ depending only on $(P,\cX)$.

Write $w=(u,v)$ under the identification $W\cong U\times V$ in $\mathsection$\ref{structure of the underlying group}. All elements of $w^{-1}Kw$ are of the form
\[
(-u,-v,1)(u^\prime,v^\prime,g^\prime)(u,v,1) \\
=(u^\prime-(u-g^\prime u)-\Psi(v,v^\prime),v^\prime-(v-g^\prime v),g^\prime)
\]
with $(u^\prime,v^\prime,g^\prime)\in K$. Since $K_T^\prime=w^{-1}K_{T^w}w=w^{-1}Kw\cap T(\A_f)$, this element is in $K_T^\prime$ iff
\begin{itemize}
\item $u^\prime=u-g^\prime u+\Psi(v,v^\prime)\in K_U$
\item $v^\prime=v-g^\prime v\in K_V$
\item $g^\prime\in T(\A_f)\cap K_G=K_T$.
\end{itemize}
So
\begin{align}
&t\in K_T; \notag\\
t\in w^{-1}K_{T^w}w\iff & v-tv\in K_V=MV(\hat{\Z}); \label{KT and KprimeT}\\
&u-tu+\Psi(v,v-tv)\in K_U=MU(\hat{\Z}). \notag
\end{align}

\begin{lemma}\label{difference between KT and wKTww}
$|K_T/K_T^\prime|\geqslant \ord(w)\prod_{p|\ord(w)}(1-\frac{1}{p})$.
\begin{proof}
Let $T^\prime$ be the image of $\G_{m,\R}\xrightarrow{\omega}\S\xrightarrow{w^{-1}\cdot x}G_{\R}$, then it is an algebraic torus defined over $\Q$ by Remark~\ref{sufficiently small congruence subgroup cotained in derivative}(1). We always have $T^\prime<T$. If $T^\prime$ is trivial, then $P=G$ is adjoint by reason of weight, and $\ord(w)=1$. If not, $T^\prime\cong \G_{m,\Q}$ and
\[
T^\prime(M):=\{t^\prime\in T^\prime(\hat{\Z})|t^\prime\equiv 1\bmod(M)\}\subset K_G\cap T(\A_f)=K_T.
\]
So
\[
T^\prime(M)/(T^\prime(M)\cap w^{-1}K_{T^w}w)\hookrightarrow K_T/w^{-1}K_{T^w}w.
\]
Hence it is enough to prove that LHS is of cardinality$\geqslant\ord(w)$.

Since $T^\prime$ acts on $V$ and $U$ via a scalar, $t^\prime\in T^\prime(M)\cap w^{-1}K_{T^w}w$ iff
\begin{enumerate}
\item $t^\prime\in T^\prime(M)$
\item $v-t^\prime v\in MV(\hat{\Z})$
\item $u-t^\prime u\in MU(\hat{\Z})$.
\end{enumerate}

Let $t^\prime\in T^\prime(M)\subset T^\prime(\hat{\Z})=\hat{\Z}^*$. Suppose $\ord(w)=\prod p^{n_p}$ and $M=\prod p^{m_p}$. If $n_p=0$, then condition (2) and (3) are automatically satisfied. If $n_p>0$, then condition (2) and (3) imply that
$t^\prime_p=1+a_{n_p+m_p}p^{n_p+m_p}+...\in\Z_p^*$, hence
\begin{equation}
|T^\prime(\Z_p)\cap T^\prime(M)/(T^\prime(\Z_p)\cap T^\prime(M)\cap w^{-1}K_{T^w,p}w)|=p^{n_p-1}(p-1).
\end{equation}
To sum up,
\begin{equation}
|T^\prime(M)/(T^\prime(M)\cap w^{-1}K_{T^w}w)|=\ord(w)\prod_{p|\ord(w)}(1-\frac{1}{p}).
\end{equation}
\end{proof}
\end{lemma}

\begin{thm}\label{galois orbit for the fiber}
For any $\epsilon\in(0,1)$, there exist a positive constant $C_{\epsilon}$ (depending only on $(P,\cX)$ and $\epsilon$) such that
\[
|\gal(\bar{\Q}/E)s|\geqslant C_{\epsilon}N(s)^{1-\epsilon}|\gal(\bar{\Q}/E)[\pi](s)|.
\]
\begin{proof}
We have proved in Lemma~\ref{difference between KT and wKTww}
\begin{equation}
p|\ord(w)\iff K_{T,p}\neq K^\prime_{T,p}.
\end{equation}
Hence denoting by $\varsigma(M):=|\{p,p|M\}|$ for any $M\in\Z_{>0}$, we have by Lemma~\ref{difference between KT and wKTww} and \eqref{ComparisonMixedPure}
\[
|\gal(\bar{\Q}/E)s|\geqslant B^{\varsigma(N(s))}N(s)\prod_{p|N(s)}(1-\frac{1}{p})|\gal(\bar{\Q}/E)\rho(s^\prime)|.
\]
Now the theorem follows from the basic facts of elementary math:
\begin{equation}
\forall \epsilon\in(0,1)\text{, there exists }C_{\epsilon}>0\text{ such that }B^{\varsigma(N(s))}N(s)^\epsilon\geqslant C_{\epsilon}.
\end{equation}
\begin{equation}
\forall \epsilon\in(0,1)\text{, there exists }C^\prime_{\epsilon}>0\text{ such that }N(s)^\epsilon\prod_{p|N(s)}(1-\frac{1}{p})\geqslant C^\prime_{\epsilon}.
\end{equation}
\end{proof}
\end{thm}

\begin{cor}\label{result of Silverberg}
For $A$ an abelian variety over a number field $k\subset\C$ and $t$ a torsion point of $A(\C)$, denote by $N(t)$ its order and $k(t)$ the field of definition of $t$ over $k$.

Let $g\in\N_+$ and let $\epsilon\in(0,1)$. There exists $c>0$ such that for all number fields $k\subset\C$, all $g$-dimensional CM abelian varieties $A$ with definition field $k$ and all torsion points $t$ in $A(\C)$,
\[
[k(t):k]\geqslant cN(t)^{1-\epsilon}.
\]
\begin{proof} (compare with \cite{SilverbergTorsion-points-})
By Zarhin's trick, it suffices to give a proof for $A$ principally polarized. Such an $A$ can be realized as a fiber of $\mathfrak{A}_g(4)\rightarrow\cA_g(4)$, and any torsion point $t$ of $A$ is a special point of $\mathfrak{A}_g(4)$. Now this result is a direct consequence of Proposition~\ref{galois orbit for the fiber}.
\end{proof}
\end{cor}

\begin{rmk}\label{lower bounds for Galois orbits}
The lower bound of the Galois orbit of a special point for pure Shimura varieties is given by \cite[Conjecture 2.7]{UllmoQuelques-applic}.
It has been proved under the Generalized Riemann Hypothesis by Ullmo-Yafaev \cite{UllmoGalois-orbits-a}. For the case of $\cA_g$, it is equivalent to the following conjectural lower bound (suggested and proved for $g=2$ by Edixhoven \cite{EdixhovenOpen-problems-i, EdixhovenOn-the-Andre-Oo}): suppose that $x\in\cA_g$ is a special point. Let $A_x$ denote the CM abelian variety parametrised by $x$ and let $R_x$ be the center of $\End(A_x)$, then there exists $\delta(g)>0$ such that
\begin{equation}\label{lower bound Siegel}
|\gal(\bar{\Q}/\Q)x|\gg_g |\disc(R_x)|^{\delta(g)}.
\end{equation}
For their equivalence see \cite[Theorem~7.1]{TsimermanBrauer-Siegel-f}. The best unconditional result is given by Tsimerman \cite[Theorem~1.1]{TsimermanBrauer-Siegel-f}: \eqref{lower bound Siegel} is true when $g\leqslant6$ (and for $g\leqslant3$ by a similar method in \cite{UllmoNombre-de-class}).

Hence for a mixed Shimura variety of Siegel type of genus $g$ and any special point $x$, Theorem~\ref{galois orbit for the fiber} tells us that if \cite[Conjecture 2.7]{UllmoQuelques-applic} is verified for the pure part, then for any $\epsilon\in(0,1)$, there exists $\delta(g)>0$ such that
\[
|\gal(\bar{\Q}/\Q)x|\gg_{g,\epsilon}N(x)^{1-\epsilon}|\disc(R_{[\pi](x)})|^{\delta(g)}.
\]
\end{rmk}

\subsection{Andr\'{e}-Oort}
\begin{thm}\label{Andre-Oort for L6}
Let $S$ be a connected mixed Shimura variety of abelian type (i.e. its pure part is of abelian type). Let $Y$ be a closed irreducible subvariety of $S$ containing a Zariski-dense set of special points. If \eqref{lower bound Siegel} holds for the pure part of $S$, then $Y$ is special.

In particular, by the main result of \cite{TsimermanBrauer-Siegel-f}, the Andr\'{e}-Oort conjecture holds unconditionally for any mixed Shimura variety whose pure part is a subvariety of $\cA_6^n$.
\begin{proof} Suppose $S$ is associated with $(P,\cX^+)$. Replacing $\Gamma$ by a neat subgroup does not change the assumption or the conclusion, so we may assume that $\Gamma=\{\gamma\in P(\Z)|\gamma\equiv1\bmod N\}$ for some $N>3$ even. Replacing $S$ by the smallest connected mixed Shimura subvariety does not change the assumption or the conclusion, so we may assume that $Y$ is Hodge generic in $S$.  Since $Y$ contains a Zariski-dense set of special points, we may assume that $Y$ is defined over a number field $k$. Suppose that $Y$ is not special.

If the set of positive-dimensional weakly special subvarieties of $Y$ is Zariski dense in $Y$, then let $N$ be the normal subgroup $P$ as in Theorem \ref{consequenceOfAx-Lindemann}.
Consider the diagram \eqref{QuotientDiagram}, then $Y$ is special iff $Y^\prime:=\bar{[\rho](Y)}$ is. It is clear that $S^\prime$ is again of abelian type. Replacing $(S,Y)$ by $(S^\prime,Y^\prime)$, we may assume that the set of positive-dimensional special subvarieties of $Y$ is not Zariski dense in $Y$.

Now we are left prove that the set of special points of $Y$ which do not lie in any positive-dimensional special subvariety is finite.

By definition, there exists a Shimura morphism $(G,\cX^+_G)\rightarrow\prod_{i=1}^r(\GSp_{2g}^{(i)},\H^{+(i)}_g)$ (the upper-index $(i)$ is to distinguish different factors) s.t. $G\rightarrow \prod_{i=1}^r\GSp_{2g}^{(i)}$ has a finite kernel (contained in the center) and $\cX^+_G\hookrightarrow\prod_{i=1}^r\H^{+(i)}_g$. Therefore under Proposition \ref{realization of the uniformizing space}, we can identify $\cX^+$ as a subspace of $U(\C)\times V(\R)\times\H_g^{+r}$. Then any special point in contained in $U(\Q)\times V(\Q)\times(\H_g^{+r}\cap M_{2g}(\bar{\Q})^r)$ and hence we can define its height (for $\bar{\Q}$-points, see \cite[Definition 1.5.4 multiplicative height]{BombieriHeights-in-diop}).

Now take $\cF$ as in $\mathsection$10.1. For any special point $x\in S$, take a representative $\tilde{x}\in \unif^{-1}(x)$ in $\cF$,  then by \cite[Theorem 3.1]{PilaAbelianSurfaces}, $H(\tilde{x}_{G,i})\ll |\disc(R_{[\pi](x)_i})|^{B_g}$ for a constant $B_g$ ($\forall i=1,...,r$). By choice of $\cF$, $H(\tilde{x}_V),~H(\tilde{x}_U)\ll N(x)$ (see Remark \ref{GoodFundamentalDomainForUniversalFamily}). If \eqref{lower bound Siegel} holds, then by Proposition \ref{galois orbit for the fiber}
\[
\#(\gal(\bar{\Q}/k)x)\gg_g H(\tilde{x})^{\epsilon(g)}
\]
for some $\epsilon(g)>0$. Hence for $H(\tilde{x})\gg 0$, Pila-Wilkie \cite[3.2]{PilaO-minimality-an} implies that $\exists\sigma\in\gal(\bar{\Q}/k)$ s.t. $\tilde{\sigma(x)}$ is contained in a semi-algebraic subset of $\unif^{-1}(Y)\cap\cF$ of positive dimension. But by \cite[pp6, last line]{UllmoQuelques-applic},
\[
\bigcup_{\substack{\tilde{Z}\subset \unif^{-1}(Y)\text{ semi-algebraic}\\ \dim(\tilde{Z})>0}}\tilde{Z}=\bigcup_{\substack{\tilde{Z}\subset \unif^{-1}(Y)\text{ irreducible algebraic}\\ \dim(\tilde{Z})>0}}\tilde{Z}
\]
So $\tilde{\sigma(x)}$ is contained in some maximal algebraic subset $\tilde{Z}$ of $\unif^{-1}(Y)$ of positive dimension. Theorem \ref{Ax-Lindemann for type star} tells us that $\tilde{Z}$ is weakly special. Hence $\sigma^{-1}(Z)$ ($Z:=\unif(\tilde{Z})$) is weakly special containing a special point $x$. Hence $\sigma^{-1}(Z)$ is special of positive dimension. To sum up, the heights of the elements of
\[
\{\tilde{x}\in \unif^{-1}(Y)\cap\cF\text{ special and }\unif(\tilde{x})\text{ is not contained in a positive-dimensional special subvariety}\}
\]
is uniformly bounded, and hence this set is finite by Northcott's theorem \cite[Theorem 1.6.8]{BombieriHeights-in-diop}.
\end{proof}
\end{thm}

\section{Appendix}
We prove here Theorem \ref{Ax-Lindemann for complex semi-abelian varieties} when $E=T$ is an algebraic torus over $\C$ (which corresponds to the case $W=U$) and when $E=A$ is a complex abelian variety (which corresponds to the case $W=V$). The proof is a rearrangement of existing proofs (combine the point counting of Pila-Zannier \cite{PilaRational-points} and volume calculation of Ullmo-Yafaev \cite{UllmoThe-Hyperbolic-}). Use notation in $\mathsection$\ref{Ax-Lindemann Part 3: The unipotent part}.

\textit{\boxed{Case~i:~$E=A$.}}
In this case, $W=V$ and $\Gamma_V=\oplus_{i=1}^{2n}\Z e_i\subset\lie(A)=\C^n=\R^{2n}$ is a lattice. Denote by $univ\colon\lie(A)\rightarrow A$. Let $\cF_V:=\Sigma_{i=1}^{2n}(-1,1)e_i$, then $\cF_V$ is a fundamental set for the action of $\Gamma_V$ on $\lie(A)$ s.t. $univ|_{\cF_V}$ is definable. Define the norm of $z=(x_1,y_1,...,x_n,y_n)\in\lie(A)=\R^{2n}$ to be
\[
\parallel z\parallel :=\Max(|x_1|,|y_1|,...,|x_n|,|y_n|).
\]
It is clear that $\forall z\in\lie(A)$ and $\forall \gamma_V\in\Gamma_V$ s.t. $\gamma_V z\in\cF_V$,
\begin{equation}\label{height and norm V part}
H(\gamma_V)\ll\parallel x_V\parallel .
\end{equation}

Let $\omega_V:=dz_1\wedge d\bar{z}_1+...+dz_n\wedge d\bar{z}_n$ be the canonical $(1,1)$-form of $\lie(A)=\C^n$. Let $p_i$ ($i=1,...,n$) be the $n$ natural projections of $\lie(A)=\C^n$ to $\C$. Let $C$ be an algebraic curve of $\tilde{Z}$ and define $C_M:=\{z\in C|\parallel z\parallel \leqslant M\}$. We have
\begin{align*}
 \int_{C\cap\cF_V}\omega_V
 &\leqslant d\sum_{i=1}^n\int_{p_i(C\cap\cF_V)}dz_i\wedge d\bar{z}_i \numberthis \label{volume 3}\\
 &\leqslant d\sum_{i=1}^n\int_{p_i(\cF_V)}dz_i\wedge d\bar{z}_i=d\cdot O(1)
\end{align*}
and
\begin{equation}\label{volume 4}
\int_{C_M}\omega_V\geqslant O(M^2)
\end{equation}
with $d=\deg(C)$ by \cite[Theorem 0.1]{HwangVolumes-of-comp}.

By \eqref{height and norm V part}
\[
C_M\subset\bigcup_{\gamma_V\in\Theta(\tilde{Z},M)}(C\cap\gamma^{-1}\cF).
\]
Integrating both side w.r.t. $\omega_V$ we have
\[
M^2\ll \#\Theta(\tilde{Z},M)
\]
by \eqref{volume 3} and \eqref{volume 4}.

Let $\stab_V(\tilde{Z}):=\bar{\Gamma_V\cap\stab_{V(\R)}(\tilde{Z})}^{\Zar}$. Now by Pila-Wilkie \cite[Theorem 3.4]{UllmoThe-Hyperbolic-}, the exists an semi-algebraic block $B\subset\Sigma(\tilde{Z})$ of positive dimension containing arbtrarily many points $\gamma_V\in\Gamma_V$. We have $B\tilde{Z}\subset univ^{-1}(Y)$ since $\Sigma(\tilde{Z})\tilde{Z}\subset univ^{-1}(Y)$ by definition. Hence for any $\gamma_V\in\Gamma_V\cap B$, $\tilde{Z}\subset\gamma_V^{-1}B\tilde{Z}\subset univ^{-1}(Y)$, and therefore $\tilde{Z}=\gamma_V^{-1}B\tilde{Z}$ by maximality of $\tilde{Z}$. So $\gamma_V^{-1}(B\cap\Gamma_V)\subset\stab_V(\tilde{Z})(\Q)$, and hence $\dim(\stab_V(\tilde{Z}))>0$. For any point $\tilde{z}\in\tilde{Z}$, $\stab_V(\tilde{Z})(\R)+\tilde{z}\subset\tilde{Z}$. By \cite[Lemma 2.3]{PilaRational-points}, $\stab_V(\tilde{Z})(\R)$ is full and complex. Define $V^\prime:=V/\stab_V(\tilde{Z})$ and $\Gamma_{V^\prime}:=\Gamma_V/(\Gamma_V\cap\stab_V(\tilde{Z})(\Q))$, and then $A^\prime:=V^\prime(\R)/\Gamma_{V^\prime}$ is a quotient abelian variety of $A$. Let $Y^\prime$ (resp. $\tilde{Z}^\prime$) be the Zariski closure of the projection of $Y$ (resp. $\tilde{Z}$) in $A^\prime$ (resp. $V^\prime(\R)$). We prove that the image of $\tilde{Z}^\prime$ is a point. If not, then proceeding as before for the triple $(A^\prime,Y^\prime,\tilde{Z}^\prime)$ can we prove $\dim(\stab_{V^\prime}(\tilde{Z}^\prime))>0$. This contradicts to the definition (maximality) of $\stab_V(\tilde{Z})$. Hence $\tilde{Z}$ is a translate of $\stab_V(\tilde{Z})(\R)$. So $\tilde{Z}$ is weakly special.

\textit{\boxed{Case~ii:~$E=T$.}}
Define the norm of $x_U=(x_{U,1},x_{U,2},...,x_{U,m})\in U(\C)$ to be
\[
\parallel x_U\parallel :=\Max(\parallel x_{U,1}\parallel ,\parallel x_{U,2}\parallel ,...,\parallel x_{U,m}\parallel ).
\]
It is clear that $\forall x_U\in U(\C)$ and $\forall \gamma_U\in\Gamma_U$ s.t. $\gamma_U x_U\in\cF_U$,
\begin{equation}\label{height and norm U part}
H(\gamma_U)\ll\parallel x_U\parallel .
\end{equation}

Let $\omega|_T=dz_1\wedge d\bar{z}_1+...+dz_m\wedge d\bar{z}_m$ be the canonical $(1,1)$-form of $U(\C)\cong\C^m$. Let $p_i$ ($i=1,...,m$) be the $m$ natural projections of $U(\C)\cong\C^m$ to $\C$. Let $C$ be an algebraic curve of $\tilde{Z}$ and define $C_M:=\{x\in C|\parallel x\parallel \leqslant M\}$. We have
\begin{align*}
 \int_{C_M\cap\cF_U}\omega|_T 
 &\leqslant d\sum_{i=1}^m\int_{p_i(C_M\cap\cF_U)}dz_i\wedge d\bar{z}_i \numberthis \label{volume 1}\\
 &\leqslant d\sum_{i=1}^m\int_{\{s\in\C|-1<\Re(s)<1,\parallel s\parallel \leqslant M\}}dz_i\wedge d\bar{z}_i=d\cdot O(M)
\end{align*}
where $d:=\deg(C)$. On the other hand by \cite[Theorem 0.1]{HwangVolumes-of-comp},
\begin{equation}\label{volume 2}
\int_{C_M}\omega|_T \geqslant O(M^2).
\end{equation}

By \eqref{height and norm U part}
\[
C_M\subset\bigcup_{\gamma\in\Theta(\tilde{Z},M)}(C_M\cap\gamma^{-1}\cF).
\]
Integrating both side w.r.t. $\omega|_T$ and taking into account that 
\[
\gamma\cdot C_M\subset(\gamma C)_{2M}\qquad\text{if }H(\gamma)\leqslant M,
\]
we have
\[
M^2\ll \#\Theta(\tilde{Z},M)\cdot M
\]
by \eqref{volume 1} and \eqref{volume 2}. Hence $\#\Theta(\tilde{Z},M)\gg M$.

Let $\stab_U(\tilde{Z}):=\bar{\Gamma_U\cap\stab_{U(\C)}(\tilde{Z})}^{\Zar}$. Now by Pila-Wilkie \cite[Theorem 3.6]{PilaO-minimality-an}, the exists an semi-algebraic subset $B\subset\Sigma(\tilde{Z})$ of positive dimension containing arbtrarily many points $\gamma_U\in\Gamma_U$. We have $B\tilde{Z}\subset univ^{-1}(Y)$ since $\Sigma(\tilde{Z})\tilde{Z}\subset univ^{-1}(Y)$ by definition. Hence for any $\gamma_U\in\Gamma_U\cap B$, $\tilde{Z}\subset\gamma_U^{-1}B\tilde{Z}\subset univ^{-1}(Y)$, and therefore $\tilde{Z}=\gamma_U^{-1}B\tilde{Z}$ by maximality of $\tilde{Z}$. So $\gamma_U^{-1}(B\cap\Gamma_U)\subset\stab_U(\tilde{Z})(\Q)$, and hence $\dim(\stab_U(\tilde{Z}))>0$. Let $U^\prime:=U/\stab_U(\tilde{Z})$, $\Gamma_{U^\prime}:=\Gamma_U/(\Gamma_U\cap\stab_U(\tilde{Z})(\Q))$ and $T^\prime:=U^\prime(\C)/\Gamma_{U^\prime}$. $T^\prime$ is an algebraic torus over $\C$. Let $Y^\prime$ (resp. $\tilde{Z}^\prime$) be the Zariski closure of the projection of $Y$ (resp. $\tilde{Z}$) in $T^\prime$ (resp. $U^\prime(\C)$). We prove that $\tilde{Z}^\prime$ is a point. If not, then proceeding as before for the triple $(T^\prime,Y^\prime,\tilde{Z}^\prime)$ we can prove $\dim(\stab_{U^\prime}(\tilde{Z}^\prime))>0$. This contradicts the definition (maximality) of $\stab_U(\tilde{Z})$. Hence $\tilde{Z}$ is a translate of $\stab_U(\tilde{Z})(\C)$. So $\tilde{Z}$ is weakly special.

\end{document}